\documentclass[10pt]{article}
\usepackage{amsfonts,latexsym,amssymb,amsthm,amsmath,graphicx,cases,comment}
\usepackage{paralist}
\usepackage{graphics} 
\usepackage{epsfig} 
\usepackage{epstopdf}

\usepackage[affil-it]{authblk}

\usepackage{ulem}
\usepackage[titletoc,title]{appendix}
\usepackage{empheq}
\usepackage{cancel}
\usepackage{tikz}
\usetikzlibrary{arrows,shapes}
\usetikzlibrary{decorations.pathmorphing,decorations.pathreplacing}
\usetikzlibrary{calc,patterns,angles,quotes}
\usepackage{todonotes}

\usepackage[colorlinks=true]{hyperref}
\hypersetup{urlcolor=blue, citecolor=red}

\makeatletter
\renewcommand{\subsection}{%
	\@startsection{subsection}
	{2}
	{\z@}
	{-21dd plus-8pt minus-4pt}
	{10.5dd}
	{\normalsize\bfseries\boldmath}%
}
\makeatother

\newtheorem{theorem}{Theorem}[section]
\newtheorem{thm}{Theorem}

\newtheorem{lemma}[thm]{Lemma}
\newtheorem{prop}[thm]{Proposition}

\theoremstyle{definition}
\newtheorem{definition}[theorem]{Definition}
\newtheorem{rmk}{Remark}

\newcommand{\be}{\begin{equation}}
\newcommand{\ee}{\end{equation}}
\newcommand{\bsubeq}{\begin{subequations}}
	\newcommand{\esubeq}{\end{subequations}}
\renewcommand{\div}{\text{div}}
\newcommand{\ds}{\displaystyle}

\newcommand{\cd}{{\, \cdot \, }}

\newcommand{\calL}{{\mathcal{L}}}

\newcommand{\calN}{{\mathcal{N}}}

\newcommand{\calF}{{\mathcal{F}}}

\newcommand{\calD}{{\mathcal{D}}}

\newcommand{\calA}{{\mathcal{A}}}

\newcommand{\calC}{{\mathcal{C}}}

\newcommand{\calM}{{\mathcal{M}}}

\newcommand{\calV}{{\mathcal{V}}}
\newcommand{\BA}{\mathbb{A}}

\newcommand{\BR}{\mathbb{R}}
\newcommand{\BC}{\mathbb{C}}
\newcommand{\BF}{\mathbb{F}}

\newcommand{\wti}{\widetilde}

\newcommand{\bpm}{\begin{pmatrix}}
	\newcommand{\epm}{\end{pmatrix}}

\newcommand{\bbm}{\begin{bmatrix}}
	\newcommand{\ebm}{\end{bmatrix}}

\numberwithin{equation}{section}
\numberwithin{thm}{section}
\numberwithin{rmk}{section}

\newtheorem{clr}[thm]{Corollary}

\newcommand{\lso}{L^{q}_{\sigma}(\Omega)}
\newcommand{\lo}[1]{L^{#1}_{\sigma}(\Omega)}
\newcommand{\ls}{L^{q}_{\sigma}}

\newcommand{\Bso}{B^{2-\rfrac{2}{p}}_{q,p}(\Omega)}
\newcommand{\Bsos}{B^{2-\rfrac{2}{p}}_{q',p}(\Omega)}

\newcommand{\Bt}{\widetilde{B}^{2-\rfrac{2}{p}}_{q,p}}
\newcommand{\Bto}{\widetilde{B}^{2-\rfrac{2}{p}}_{q,p}(\Omega)}
\newcommand{\Btos}{\widetilde{B}^{2-\rfrac{2}{p}}_{q',p}(\Omega)}
\newcommand{\xtpq}{X^T_{p,q}}
\newcommand{\xtpqs}{X^T_{p,q,\sigma}}
\newcommand{\xipqs}{X^{\infty}_{p,q,\sigma}}
\newcommand{\ytpq}{Y^T_{p,q}}
\newcommand{\yipq}{Y^{\infty}_{p,q}}
\newcommand{\xipq}{X^{\infty}_{p,q}}

\newcommand\rfrac[2]{{}^{#1}\!/_{#2}}
\newcommand{\Fs}{F_{\sigma}}

\newcommand{\lqaq}{\big( \lso, \calD(A_q) \big)_{1-\frac{1}{p},p}}
\newcommand{\lqaqs}{\big( \lo{q'}, \calD(A_q^*) \big)_{1-\frac{1}{p},p}}

\newcommand{\lqafq}{\big( \lso, \calD(\BA_{_{F,q}}) \big)_{1-\frac{1}{p},p}}
\newcommand{\lqafqs}{\big( \lo{q'}, \calD(\BA_{_{F,q}}^*) \big)_{1-\frac{1}{p},p}}

\newcommand{\lplsq}{L^p \big( 0,\infty; \lso \big)}
\newcommand{\lplqd}{L^p \big( 0,\infty; \lo{q'} \big)}

\newcommand{\norm}[1]{\left\lVert#1\right\rVert}
\newcommand{\abs}[1]{\left\lvert#1\right\rvert}

\newcommand{\lplqs}{L^p \big( 0,\infty; \lso \big)}

\newcommand{\ip}[2]{\left\langle #1, #2 \right\rangle}

\newcommand{\thistheoremname}{}

\newtheorem*{genericthm*}{\thistheoremname}
\newenvironment{namedthm*}[1]
{\renewcommand{\thistheoremname}{#1}%
	\begin{genericthm*}}
	{\end{genericthm*}}

\begin{document}

	\title{Uniform stabilization of 3D Navier-Stokes equations in critical Besov spaces with finite dimensional, tangential-like boundary, localized feedback controllers}
	
	\author{Irena Lasiecka, Roberto Triggiani}
	\affil{Department of Mathematical Sciences, University of Memphis, Memphis, TN 38152, USA.}
	
	\author{Buddhika Priyasad}
	\affil{Institute for Mathematics and Scientific Computing, 621, Heinrichstrasse 36, 8010 Graz, Austria.}
	
	\date{}
	\maketitle

	\begin{abstract}
	\noindent \noindent The present paper provides a solution in the affirmative to a recognized open problem in the theory of uniform stabilization of 3-dimensional Navier-Stokes equations in the vicinity of an unstable equilibrium solution, by means of a `minimal' and `least' invasive feedback strategy which consists of a control pair $\{ v,u \}$ \cite{LT2:2015}. Here $v$ is a tangential boundary feedback control, acting on an arbitrary small part $\wti{\Gamma}$ of the boundary $\Gamma$; while $u$ is a localized, interior feedback control, acting tangentially on an arbitrarily small subset $\omega$ of the interior supported by $\wti{\Gamma}$. The ideal strategy of taking $u = 0$ on $\omega$ is not sufficient. A question left open in the literature was: Can such feedback control $v$ of the pair $\{ v,u \}$ be asserted to be finite dimensional also in the dimension $d = 3$? We here give an affirmative answer to this question, thus establishing an optimal result.
	To achieve the desired finite dimensionality of the feedback tangential boundary control $v$, it is here then necessary to abandon the Hilbert-Sobolev functional setting of past literature and replace it with a critical Besov space setting. These spaces are `close' to $L^3(\Omega)$ for $d=3$. This functional setting is significant. It is in line with recent critical well-posedness in the full space of the non-controlled N-S equations. A key feature of such Besov spaces with tight indices is that they do not recognize compatibility conditions. The proof is constructive and is ``optimal" also regarding the ``minimal" number of tangential boundary controllers needed. The new setting requires the solution of novel technical and conceptual issues. These include establishing maximal regularity in the required critical Besov setting for the overall closed-loop linearized problem with tangential feedback control applied on the boundary. This result is also a new contribution to the area of maximal regularity. It escapes perturbation theory. Finally, the minimal amount of tangential boundary action is linked to the issue of unique continuation of over-determined Oseen eigenproblems. 
	\end{abstract}
	
	
	\section{Introduction}\label{B-Sec-1}
	
	\subsection{Controlled Dynamic Navier-Stokes Equations}
		Let $\Omega$ be an open  connected bounded domain in $\mathbb{R}^d$, $d=2,3$, with boundary $\Gamma = \partial\Omega$. Unless otherwise stated, $\Gamma$ will be assumed of class $C^2$ throughout the paper. Some results will require less/different boundary assumptions, as noted. For purposes of illustration, let $\omega$ be at first an arbitrary collar (layer) of the boundary $\Gamma$ in the interior of $\Omega$, $\omega \subset \Omega$ [Fig.\,1].  For each point $\xi \in \omega$, we consider the (sufficiently smooth)  curve ($d=2$) or surface ($d=3$) $\Gamma_\xi$, which is the parallel translation of the boundary $\Gamma$, passing through $\xi \in \omega$ and lying in $\omega$.  Let $\tau(\xi)$ be a unit tangent vector to the oriented curve $\Gamma_\xi$ at $\xi$, if $d=2$; and let $\tau(\xi) = [\tau_1(\xi),\tau_2(\xi)]$ be an orthonormal system of oriented tangent vectors lying on the tangent plane to the surface $\Gamma_\xi$ at $\xi$, if $d=3$, and obtained as isothermal parametrization via a 1-1 conformal mapping of a suitable open set in $\mathbb{R}^2$ with canonical basis $e_1=\{1,0\}, e_2=\{0,1\}$. See \cite[Appendix]{LT1:2015} for details and references.  We shall in particular focus on and study the case where $\omega$ is a localized collar based on an arbitrarily small, connected portion $\wti{\Gamma}$ of the boundary $\Gamma$ [Fig.\,2]. Let $m$ denote the characteristic function of the collar set $\omega: \ m \equiv 1$ in $\omega$, $m \equiv 0$ in $\Omega/\omega$.
	\begin{center}
		\begin{tabular}{c c}
			\begin{tikzpicture}[x=10pt,y=10pt,>=stealth, scale=.53]
			\draw[thick]
			(-15,5)
			.. controls (-20,-9) and (-5,-18) .. (15,-7)
			.. controls  (24,-1) and (15,13).. (6,6)
			.. controls (3,3) and (-3,3) .. (-6,6)
			.. controls (-9,9) and (-13,9) .. (-15,5);
			
			\draw[white,ultra thin,opacity=.5]
			(-14.9,4.9)
			.. controls (-20,-8) and (-7,-13.15) .. (0,-11.7)
			.. controls (,) and (,) .. (6,6)
			.. controls (3,3) and (-3,3) .. (-6,6)
			.. controls (-9,9) and (-13,9) .. (-14.9,4.9);
			
			\draw[thick]
			(-14.3,4)
			.. controls (-18.7,-8.3) and (-4.5,-16.7) .. (14,-6.5)
			.. controls  (22.8,-1.75) and (15,13).. (5,4.25)
			.. controls (2,2.5) and (-2,2.5) .. (-5,4.25)
			.. controls (-10,8.5) and (-13,8) .. (-14.3,4);
			\draw[thick,fill=white]
			(-5,3.3)
			.. controls (-2,1.8) and (2,1.8) .. (5,3.3)
			.. controls (10,7) and (13,7.3) .. (15.8,3.2)
			.. controls (17,1) and (17.2,-1.5) .. (15.6,-4)
			.. controls (10,-9.5) and (-8,-14) .. (-13,-5)
			.. controls (-14.5,-2) and (-14.5,1) .. (-13.5,4)
			.. controls (-12.5,7) and (-10,7.5) .. (-5,3.3);
			\draw[thick,->] (17.4,2) -- (15.5,8);
			\draw[thick] (17.4,2) -- (21,3);
			\draw[thick,->] (21,-2) -- (17.73,-2);
			\draw[thick,->] (21,-6) -- (17.2,-5);
			\draw[thin,->](-10,10) -- (-10,7.2);
			\draw (17.4,2) node {$\bullet$};
			\draw (-10,11) node[scale=1.5] {$\omega$};
			\draw (18.5,7.5) node[scale=1.25] {$\tau(\xi)$};
			\draw (22,3) node[scale=1.2] {$\xi$};
			\draw (22.5,-2) node[scale=1.2] {$\Gamma_{\xi}$};
			\draw (22,-6) node[scale=1.2] {$\Gamma$};
			\end{tikzpicture} &
			
			\begin{tikzpicture}[x=10pt,y=10pt,>=stealth, scale=.53]
			\draw[thick]
			(-15,5)
			.. controls (-20,-9) and (-5,-18) .. (15,-7)
			.. controls  (24,-1) and (15,13).. (6,6)
			.. controls (3,3) and (-3,3) .. (-6,6)
			.. controls (-9,9) and (-13,9) .. (-15,5);
			\draw[thick,shade,opacity=.5]
			(-15.8,-2)
			.. controls (-11,-4) and (-10,-7) .. (-10,-10)
			.. controls (-15,-7) and (-15.5,-4) .. (-15.8,-2);
			\draw[thick] (-15.8,-2) .. controls (-15.5,-4) and (-15,-7) .. (-10,-10);
			\draw[thick] (-14.3,-2.7) .. controls (-14.5,-4) and (-13,-7) .. (-10,-9);
			\draw[thick] (-13.2,-3.4) .. controls (-13.5,-4) and (-12,-7) .. (-10.2,-8);

			\draw[->,thin]  (-13.2,-5.9) -- (-18.9,4);
			\draw[thin]  (-13.2,-5.9) -- (-6.5,-1.8);
			\draw(-18,5) node {$\tau(\xi)$};
			\draw(-6.5, -0.5) node {$\xi$};
			
			\draw (-19.25,-7.75) node[scale=1] {$\omega$};
			\draw (-14.5,-8) node[scale=1] {$\wti{\Gamma}$};
			\draw (-15.8,-2) node {$\bullet$};
			\draw (-10,-10) node {$\bullet$};
			\draw[<-]  (-14,-5) -- (-19,-7);
			\end{tikzpicture}\\
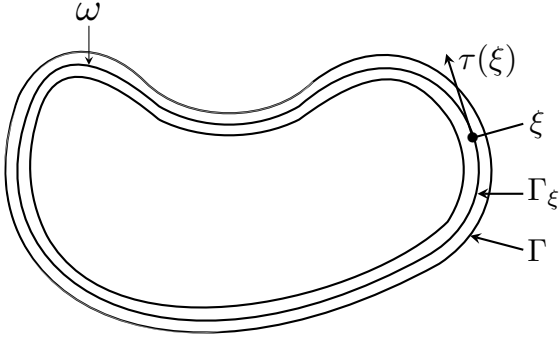
			\mbox{Fig. 1: Internal Collar $\omega$ of Full Boundary $\Gamma$}&
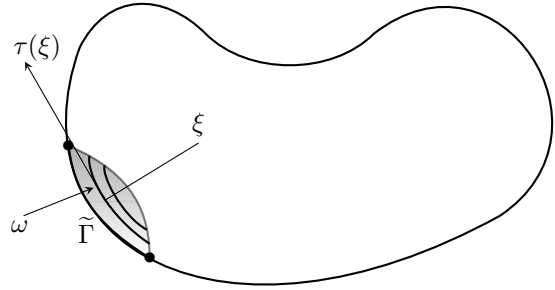
			\mbox{Fig. 2: Internal Localized Collar $\omega$ of} \\
				& \mbox{Subportion $\wti{\Gamma}$ of Boundary $\Gamma$}
		\end{tabular}
		\end{center}

	\noindent We consider the following Navier-Stokes equations perturbed by a force $f$ and subject to the action of a pair $\{v,u\}$ of controls, to be described below	
	\begin{subequations}\label{B-1.1}
		\begin{align}
		y_t(t,x) - \nu_o \Delta y(t,x) +  (y \cdot \nabla)y + \nabla \pi(t,x) - (m(x)u)\tau &= f(x)  &\text{ in } Q \label{B-1.1a}\\ 
		\text{div} \ y &= 0  &\text{ in } Q \label{B-1.1b}\\
		y &= v &\text{ on } \Sigma \label{B-1.1c}\\
		y(0,x) &= y_0(x) &\text{ in } \Omega \label{B-1.1d}
		\end{align}	
	
	\noindent where $Q = (0,\infty) \times \Omega, \ \Sigma = (0,\infty) \times \Gamma$ and the constant $\nu_o > 0$ is the viscosity coefficient. In (\ref{B-1.1c}), $v$ is a $d$-dimensional \textit{tangential} boundary control $v \cdot \nu \equiv 0$ on $\Gamma$, possibly supported on an arbitrarily small connected part $\wti{\Gamma}$ of the boundary,  where $\nu$ is the unit outward normal to $\Gamma$.  Instead, $u$ is a scalar ($d=2$) or a two dimensional vector $u = [u^1,u^2]$ ($d=3$) interior ``tangential" control acting in the `tangential direction' $\tau$ (that is, parallel to the boundary) in the small boundary layer $\omega$: $(mu) \tau$, where for $d=3$ (Fig. 2), 	
	\begin{equation} \label{B-1.1e}
		(mu)\tau = [( mu^1)\tau_1+(mu^2)\tau_2] \ \mbox{ for short}, \quad m \equiv 1 \mbox{ on } \omega; \quad m \equiv 0 \mbox{ in } \Omega \backslash \omega. 
	\end{equation}	
	\end{subequations}
	\noindent See \cite[Appendix]{LT1:2015}. The scalar function $\pi$ is the unknown pressure. \\

	\noindent \textbf{Notation:} As already done in the literature, for the sake of simplicity, we shall adopt the same notation for function spaces of scalar functions and function spaces of vector valued functions. Thus, for instance, for the vector valued ($d$-valued) velocity field $y$ or external force $f$, we shall simply write say $y,f \in L^q(\Omega)$ rather than $y,f \in (L^q(\Omega))^d$ or $y,f \in \mathbf{L}^q(\Omega)$. This choice is unlikely to generate confusion. The initial condition $y_0$ and the body force $f \in L^q(\Omega)$ are given. The scalar function $\pi$ is the unknown pressure.	
	\subsection{Stationary Navier-Stokes equations}\label{B-Sec-1.2}
	The following result represents our basic starting point.
	
	\begin{thm}\label{B-Thm-1.1}
		Consider the following steady-state Navier-Stokes equations in $\Omega$
		\begin{subequations}\label{B-1.2}
			\begin{align}
			-\nu_o \Delta y_e +  (y_e.\nabla)y_e + \nabla \pi_e &= f &\text{ in }  \Omega \label{B-1.2a}\\ 
			div \ y_e &= 0 &\text{ in }  \Omega\label{B-1.2b} \\
			y_e &= 0 &\text{ on } \Gamma. \label{B-1.2c}
			\end{align}
		\end{subequations}
		Let $1 < q < \infty$. For any $f \in L^q(\Omega)$ there exits a solution (not necessarily unique) $(y_e,\pi_e) \in (W^{2,q}(\Omega) \cap W^{1,q}_{0}(\Omega)) \times W^{1,q}(\Omega)$.
	\end{thm}
	
	\noindent For the Hilbert case $q=2$, see \cite[Thm 7.3 p 59]{CF:1980} . For the general case $1 < q < \infty$, see \cite[ Thm 5.iii p 58]{AR:2010}.
	
	\begin{rmk}\label{B-rmk-1.1}
		It is well-known \cite{Lad:1969}, \cite{Li:1969}, \cite{Te:1979} that the stationary solution is unique when ``the data is small enough, or the viscosity is large enough" \cite[p 157; Chapt 2]{Te:1979} that is, if the ratio $\ds \rfrac{\norm{f}}{\nu_o^2}$ is smaller than some constant that depends only on $\Omega$ \cite[p 121]{FT:1984}. When non-uniqueness occurs, the stationary solutions depend on a finite number of parameters \cite[Theorem 2.1, p 121]{FT:1984} asymptotically, in the time dependent case.
	\end{rmk}
	
	\begin{rmk}\label{B-rmk-1.2}
		The case where $f(x)$ in (\ref{B-1.1a}) is replaced by $\ds \nabla g(x)$ is noted in the literature as arising in certain physical situations, where $f$ is a conservative vector field. In this case, a solution of the stationary problem (\ref{B-1.2}) is $y_e \equiv 0, \pi_e = g$. The analysis of this relevant case will be discussed in the Orientation, Case 3 in Section \ref{B-Sec-1.5}; in Remark \ref{B-rmk-2.2} at the end of Section \ref{B-Sec-2}; and in Problem \#3, Appendix \ref{B-app-C}.
	\end{rmk}

	\subsection{The stabilization problem.}\label{B-Sec-1.3}
	
	\subsubsection{Its physical and mathematical importance}
	With reference to the viscous Navier-Stokes fluid in a bounded region of the two- or three-dimensional space under the action of a given time-independent driving mechanism such a body force, an attractive physical description of the feedback stabilization problem in fluid dynamics has been given in the report of Referee \#1 of paper \cite{LT2:2015}. We quote from such review (see also \cite[Check the page below Eq (1.4) ?]{LT2:2015}):\\
	\textit{``It is experimentally observed, and analytically and numerically validated, that if the magnitude of the driving mechanism, call it $\abs{\calD}$, is below a certain critical value, $\calC$, then the corresponding flow of the liquid is time-independent as well, and it is also unique and stable. However, if $\abs{\calD} > \calC$, then another motion (not necessarily of steady nature) appears and is stable.Eventually, when $\abs{\calD}$ becomes very large, the corresponding motion is of chaotic nature, and turbulence sets in. The stabilization problem consists in avoiding the occurrence of the above process by forcing the flow to keep its original steady-state regime though a suitable feedback control."}\\
	
	\noindent Stated in other words: For large Reynolds numbers $\ds \rfrac{1}{\nu_0}$, the steady state solution $y_e$ becomes unstable in a quantitative sense to be made more precise below in \eqref{B-1.3}, and may cause turbulence: it is therefore important to be able to suppress turbulence asymptotically in time by selecting a suitable feedback control action. As to turbulence theory, one of its main features is the so called
	phenomenon of energy cascade, dating back to Kolmogorov, whereby the average energy at any given scale is governed by three elements: the	input from the driving force; the inertial effects that transfer energy
	toward lower scales; and dissipation due to viscosity. A recent	contribution aimed at a better understanding of the mathematical	mechanisms due to turbulence is \cite{DG:2016}. Our goal in this paper is to	suppress turbulence - potentially caused by an external force - asymptotically in time.\\
	
	\noindent \uline{Assumption of instability.} Let $d = 2,3$.  We label ``unstable" to mean that the corresponding Oseen operator $\calA_q$ in \eqref{B-1.11}, which depends on $y_e$ via \eqref{B-1.10}, has a finite number, say $N$, of not necessarily distinct eigenvalues $\lambda_1, \lambda_2, \lambda_3, \dots, \lambda_N$ on the complex half plane $\{ \lambda \in \BC: Re \ \lambda \leq 0 \}$ which we then order according to their real parts, so that
	\begin{equation}\label{B-1.3}
	\ldots \leq Re~\lambda_{N+1} < 0 \leq Re~\lambda_N \leq \ldots \leq Re~\lambda_1,
	\end{equation}
	\noindent each $\lambda_i, \ i=1,\dots,N$, being an unstable eigenvalue repeated according to its geometric multiplicity $\ell_i$. We shall next indicate with $M (\leq N)$ the number of distinct unstable eigenvalues of $\calA_q$. Thus, the corresponding uncontrolled linearized $w$-problem \eqref{B-1.28} ($v = 0, u = 0$) is described by a strongly continuous analytic semigroup generated by such Oseen operator, Appendix \ref{B-app-A}, which moreover is unstable. We seek to counteract such instability by devising a suitable feedback [acting on the state $y(t)$ only] control strategy $\{v,u\}$, that not only produces the corresponding linearized $w$-problem to be (globally) uniformly stable, but - in addition - forces the overall nonlinear problem \eqref{B-1.1} to be uniformly stable in the vicinity of such `unstable' equilibrium solution $y_e$. Henceforth, unless otherwise stated, we consider the pair $\{ \wti{\Gamma}, \omega \}$, where $\wti{\Gamma}$ is {\it an arbitrarily small connected portion of the boundary} $\Gamma = \partial \Omega$, and {\it $\omega$ is an arbitrarily small interior subset} of $\Omega$, supported by $\wti{\Gamma}$ as in Fig 2. Moreover, following \cite{LT2:2015}, the pair $\{v,u\}$ of stabilizing feedback controls will be acting on $\{ \wti{\Gamma}, \omega \}$. This means that: the boundary feedback control $v$ will be taken as acting tangentially (no normal component) on the arbitrary small connected portion $\wti{\Gamma}$ of the boundary $\Gamma = \partial \Omega$; the interior control $u$ will be taken as acting tangential-like (parallel to the boundary) on the arbitrary small interior subset $\omega$ of $\Omega$, supported by $\wti{\Gamma}$ (Fig. 2).
	
	\subsubsection{Main features of present solution.} \label{B-Sec-1.3.2} The following features of the stabilization problem solved in this paper - both state-of-the-art of the literature and new contributions - need to be stressed at the outset. A more extensive explanation is postponed to Section \ref{B-Sec-1.4.4}, following the main Theorem \hyperref[B-Thm-A]{A}.
	
	\begin{enumerate}[(1)]
		\item \uline{Insufficiency of the pair $\{ v, u \equiv 0 \}$ on $\{ \wti{\Gamma}, \omega \}$.} Use of the \uline{sole} tangential boundary feedback control $v$ as localized on $\wti{\Gamma}$ (that is, $u \equiv 0$ on $\omega$) is \uline{not} sufficient as a stabilizing control, regardless of whether $v$ is finite or infinite dimensional. This is due to the counterexample in \cite{FL:1996}, to be discussed in the Orientation Section \ref{B-Sec-1.5} and in Problem \#1 of Appendix \ref{B-app-C}. 
		
		\item \uline{Minimal extra condition to be added to $v$.} Thus, the addition of an interior, tangential-like feedback control $u$, as localized on the companion domain $\omega$ is a \uline{minimal extra condition}, with which  to supplement such $v$. 
		
		\begin{rmk}\label{B-rmk-1.3}
			To put the above point (2) in proper perspective, we recall that uniform stabilization of the Navier-Stokes equations $d = 2,3,$ by means of a localized feedback interior control (with no constraint on being tangential-like, if supported on a patch of the boundary) was solved in \cite{BT:2004} in a Hilbert space setting. The recent solution \cite{LPT.1} in the Besov setting of the present paper $d = 2,3,$ improves \cite{BT:2004} in both results and methods; for instance on the test for the finite dimensional controllability condition and the nonlinear analysis.
		\end{rmk}
	
	\item \uline{Main contribution: finite dimensionality of $v$.} Regarding the described pair $\{ v, u \}$ on $\{ \wti{\Gamma}, \omega \}$ present state-of-the-art has succeeded \cite{LT1:2015}, \cite{LT2:2015} in establishing local exponential stabilization (asymptotic turbulence suppression) near an unstable equilibrium solution $y_e$ by means of a \textit{localized finite-dimensional tangential} feedback boundary control $v$ of the pair $\{ v, u \}$ on $\{ \wti{\Gamma}, \omega \}$, in the Hilbert setting in two cases: (i) when the dimension $d = 2$, (ii) when the dimension $d = 3$ but the initial condition $y_0$ in (\ref{B-1.1d}) is compactly supported. In the general $d = 3$ case, handling of the non-linearity of the N-S problem forces a Hilbert space setting with a high-topology $H^{\rfrac{1}{2} + \varepsilon}(\Omega), \ \varepsilon > 0$, whereby the compatibility conditions kick in. These then cannot allow the boundary stabilizing feedback control $v$ to be finite dimensional in general for $d=3$. In the case $d = 3$, the obstruction due to the compatibility conditions in the Hilbert setting of all the past literature was recognized also by other authors. In \cite{JP:2007}, the compatibility condition between the initial state and the feedback controller at $t = 0$ is achieved by choosing a \textit{time-varying control operator}. Same in \cite{BT:2004}: this is the \textit{dynamic} controller of the title of these authors' paper. Moreover, in both these works, the boundary control has tangential as well as normal components. In contrast, in our present paper, we take of course the feedback boundary control $v = F(\cdot)$ on $\wti{\Gamma}$ to be tangential with $F$ a \uline{static} operator. The main goal and contribution of the present work is to remove the deficiency noted in (ii) on the localized tangential stabilizing boundary control $v$ of the pair $\{ v, u \}$ on $\{ \wti{\Gamma}, \omega \}$ in the case $d = 3$, and thus obtain constructively local uniform feedback stabilization of (\ref{B-1.1}) near an unstable equilibrium solution $y_e$, by means of a control pair $\{ v, u \}$ on $\{ \wti{\Gamma}, \omega \}$ with a stabilizing tangential, boundary localized static feedback control operator $v = F(\cdot)$ in (\ref{B-1.1d}) \underline{which is also finite dimensional} in the case $d = 3$. \uline{This involves the affirmative an established and recognized open problem in this area.}
	
	\item \uline{Strategy: from the Hilbert setting of the literature to a new tight Besov space setting.} To this end, we need therefore to go beyond the Hilbert setting of the literature and thus achieve local uniform stabilization near an equilibrium solution $y_e$ in the case $d=3$ in a space enjoying the following two features: on the one hand, it must accommodate the N-S nonlinearity for $d=3$; and on the other hand, it must not recognize the boundary conditions, in order not to be subject to compatibility conditions. Thus, the present paper will provide a feedback stabilization pair $\{v,u\}$, in (\ref{B-1.1c}) and in (\ref{B-1.1a}) respectively, \underline{both finite-dimensional} also in the case $d = 3$ (in the case of $u$, this is already known \cite{LT1:2015}, \cite{BT:2004}) and spectral based, this time however within a critical Besov-setting. In particular, local exponential stability for the velocity field $y$ near an unstable equilibrium solution $y_e$ will be achieved for $d = 3$ in the topology of the Besov space $\Bto$ in \eqref{B-1.15b} which does not recognize compatibility conditions. See Remark \ref{B-rmk-1.4}.
	
\end{enumerate}
	
	\subsection{Main Results} \label{B-Sec-1.4}
	Before stating the main results, we need to introduce the necessary mathematical setting.
	
	\subsubsection{Preliminaries: Helmholtz decomposition} \label{B-Sec-1.4.1}
	A first difficulty one faces in extending the local exponential stabilization result near an equilibrium solution $y_e$ with tangential control pair $\{ v,u \}$ of the original problem (\ref{B-1.1}) from the Hilbert-space setting in \cite{BT:2004}, \cite{BLT1:2006}, \cite{LT2:2015} to the $L^q$/Besov setting is the question of the existence of a Helmholtz (Leray) projection for the domain $\Omega$ in $\mathbb{R}^d$. More precisely: Given an open set $\Omega \subset \mathbb{R}^d$, the Helmholtz decomposition answers the question as to whether $L^q(\Omega)$ can be decomposed into a direct sum of the solenoidal vector space $\lso$ and the space $G^q(\Omega)$ of gradient fields. Here,
	
	\begin{equation}\label{B-1.4}
	\begin{aligned}
	\lso &= \overline{\{y \in C_c^{\infty}(\Omega): \div \ y = 0 \text{ in } \Omega \}}^{\norm{\cdot}_q}\\
	&= \{g \in L^q(\Omega): \div \ g = 0; \  g\cdot \nu = 0 \text{ on } \partial \Omega \},\\
	& \hspace{3cm} \text{ for any locally Lipschitz domain } \Omega \subset \mathbb{R}^d, d \geq 2 \ \cite[p \ 119]{Ga:2011}\\
	G^q(\Omega) &= \{y \in L^q(\Omega):y = \nabla p, \ p \in W_{loc}^{1,q}(\Omega) \ \text{where } 1 \leq q < \infty \}.
	\end{aligned}
	\end{equation}
	
	Both of these are closed subspaces of $L^q$.
	
	\begin{definition}\label{B-Def-1.1}
		Let $1 < q < \infty$ and $\Omega \subset \mathbb{R}^n$ be an open set. We say that the Helmholtz decomposition for $L^q(\Omega)$ exists whenever $L^q(\Omega)$ can be decomposed into the direct sum
		\begin{equation}
		L^q(\Omega) = \lso \oplus G^q(\Omega).\label{B-1.5}
		\end{equation}
		The unique linear, bounded and idempotent (i.e. $P_q^2 = P_q$) projection operator $P_q:L^q(\Omega) \longrightarrow \lso$ having $\lso$ as its range and $G^q(\Omega)$ as its null space is called the Helmholtz projection.
	\end{definition}

	Here below we collect a subset of known results about Helmholtz decomposition. We refer to \cite[Section 2.2]{HS:2016}, in particular for the comprehensive Theorem 2.2.5 in this reference, which collects domains for which the Helmholtz decomposition is known to exist. These include the following cases:
	
	\begin{enumerate}[(i)]
		\item any open set $\Omega \subset \mathbb{R}^d$ for $q = 2$, i.e. with respect to the space $L^2(\Omega)$; more precisely, for $q = 2$, we obtain the well-known orthogonal decomposition (in the standard notation, where $\nu=$unit outward normal vector on $\Gamma$) \cite[Prop 1.9, p 8]{CF:1980}
		\begin{subequations}\label{B-1.6}
			\begin{align}
			L^2(\Omega) &= H \oplus H^{\perp} \label{B-1.6a}\\
			H &= \{ \phi \in L^2(\Omega): \div \ \phi \equiv 0 \text{ in } \Omega; \ \phi \cdot \nu \equiv 0 \text{ on } \Gamma \} \label{B-1.6b}\\
			H^{\perp} &= \{ \psi \in L^2(\Omega): \psi = \nabla h, \ h \in H^1(\Omega) \}; \label{B-1.6c}
			\end{align}
		\end{subequations}
		\item a bounded $C^1$-domain in $\mathbb{R}^d$ \cite{FMM:1998}, $1 < q < \infty $, or \cite[Theorem 1.1 p 107, Theorem 1.2 p 114]{Ga:2011} for $C^2$-boundary;
		\item a bounded Lipschitz domain $\Omega \subset \mathbb{R}^d \ (d = 3)$ and for $\frac{3}{2} - \epsilon < q < 3 + \epsilon$ sharp range \cite{FMM:1998};
		\item a bounded convex domain $\Omega \subset \mathbb{R}^d, d \geq 2, 1 < q < \infty$ \cite{FMM:1998}.
	\end{enumerate}
	 
		On the other hand, on the negative side, it is known that there exist domains $\Omega \subset \mathbb{R}^d$ such that the Helmholtz decomposition does not hold for some $q \neq 2$ \cite{MB:1986}.\\
	
	\noindent \textbf{Assumption (H-D)} Henceforth in this paper, we assume that the bounded domain $\Omega \subset \mathbb{R}^d$ under consideration admits a Helmholtz decomposition for the values of $q, \ 1 < q < \infty$, here considered at first, for the linearized problem (\ref{B-1.28}) below. This is the case for domains of class $C^2$, as assumed. The final results Theorems \hyperref[B-Thm-A]{A} of Section \ref{B-Sec-1.4.4} for the non-linear problem (\ref{B-1.1}) will require $q > 3$, see (\ref{B-7.19}), in the case of interest $d = $ dim $\Omega = 3$.\\ 
	
	\noindent Let $\Omega \subset \BR^d$ be an open set and let $1 < q < \infty$. The Helmholtz decomposition exists for $L^q(\Omega)$ if and only if it exists for $L^{q'}(\Omega)$, and we have: (adjoint of $P_q$) = $P^*_q = P_{q'}$ (in particular $P_2$ is orthogonal), where $P_q$ is viewed as a bounded operator $\ds L^q(\Omega) \longrightarrow L^q(\Omega)$, and $\ds P^*_q = P_{q'}$ as a bounded operator $\ds L^{q'}(\Omega) \longrightarrow L^{q'}(\Omega), \ \rfrac{1}{q} + \rfrac{1}{q'} = 1$. See \cite[Prop 2.2.2 p6]{HS:2016}, \cite[Ex. 16 p115]{Ga:2011}, \cite{FMM:1998}, \cite{LPT.1}. Through out the paper we shall use freely that \cite[Appendix A]{LPT.1}
	\begin{equation}\label{B-1.7}
	(\lso)' = \lo{q'}, \quad \frac{1}{q} + \frac{1}{q'} = 1.
	\end{equation}
	
	\subsubsection{Preliminaries: The Stokes and Oseen Operators.}\label{B-Sec-1.4.2}
	
	\noindent First, for $1 < q < \infty$ fixed, the Stokes operator $A_q$ in $\lso$ with Dirichlet boundary conditions  is defined by \cite[p 1404]{GGH:2012}, \cite[p 1]{HS:2016}
	\begin{equation}\label{B-1.8}
	A_q z = -P_q \Delta z, \quad
	\mathcal{D}(A_q) = W^{2,q}(\Omega) \cap W^{1,q}_0(\Omega) \cap \lso.
	\end{equation}
	The operator $A_q$ has a compact inverse $A_q^{-1}$ on $\lso$, hence $A_q$ has a compact resolvent on $\lso$. Next, we introduce the first order Oseen perturbation $L_e$ is given by 
	\begin{equation}\label{B-1.9}
	L_e(z) = (y_e \cdot \nabla) z + (z \cdot \nabla) y_e.
	\end{equation}	
	\noindent Accordingly we define, the first order operator $A_{o,q}$, via (\ref{B-1.8}) and (\ref{B-1.9})
	\begin{equation}\label{B-1.10}
	A_{o,q} z = P_q L_e(z) = P_q[(y_e \ . \ \nabla )z + (z \ . \ \nabla )y_e], \ \mathcal{D}(A_{o,q}) = \mathcal{D}(A_q^{\rfrac{1}{2}}) = W^{1,q}_0 (\Omega) \cap \lso \subset \lso,
	\end{equation}
	where $\ds A_q^{\rfrac{1}{2}}$ is defined in (\ref{B-A.6}) below. Thus, $A_{o,q}A_q^{-\rfrac{1}{2}}$ is a bounded operator on $\lso$, and thus $A_{o,q}$ is bounded on $\ds \calD \big(A_q^{\rfrac{1}{2}} \big)$. This leads to the definition of the Oseen operator
	\begin{equation}\label{B-1.11}
	\calA_q  = - (\nu_o A_q + A_{o,q}), \quad \calD(\calA_q) = \calD(A_q) \subset \lso
	\end{equation}
	also with compact resolvent.	
		
	\subsubsection{Preliminaries: Definition of Besov spaces $B^s_{q,p}$ on domains of class $C^1$ as real interpolation of Sobolev spaces:}\label{B-Sec-1.4.3} 
	Let $m$ be a positive integer, $m \in \mathbb{N}, 0 < s < m, 1 \leq q < \infty,1 \leq p \leq \infty,$ then we define \cite[p 1398]{GGH:2012}
	
		\begin{equation}\label{B-1.12}
		B^{s}_{q,p}(\Omega) = (L^q(\Omega),W^{m,q}(\Omega))_{\frac{s}{m},p}
		\end{equation}
		\noindent \cite[p. xx]{W:1985}. This definition does not depend on $m \in \mathbb{N}$.
		This clearly gives
		\begin{equation}\label{B-1.13}
		W^{m,q}(\Omega) \subset B_{q,p}^s(\Omega) \subset L^q(\Omega) \quad \text{ and } \quad \norm{y}_{L^q(\Omega)} \leq C \norm{y}_{B_{q,p}^s(\Omega)}.
		\end{equation}

	\noindent We shall be particularly interested in the following special real interpolation space of $L^q$ and $W^{2,q}$ spaces $\Big( m = 2, s = 2 - \frac{2}{p} \Big)$:
	\begin{equation}\label{B-1.14}
	B^{2-\frac{2}{p}}_{q,p}(\Omega) = \big(L^q(\Omega),W^{2,q}(\Omega) \big)_{1-\frac{1}{p},p}.
	\end{equation}
	Our interest in (\ref{B-1.14}) is due to the following characterization \cite[Theorem 3.4]{HA:2000}, \cite[p1399]{GGH:2012}: if $A_q$ denotes the Stokes operator introduced in (\ref{B-1.8}), then 
	\begin{subequations}\label{B-1.15}
		\begin{align}
		\Big( \lso,\mathcal{D}(A_q) \Big)_{1-\frac{1}{p},p} &= \Big\{ g \in \Bso : \text{ div } g = 0, \ g|_{\Gamma} = 0 \Big\} \quad \text{if } \frac{1}{q} < 2 - \frac{2}{p} < 2 \label{B-1.15a}\\
		\Big( \lso,\mathcal{D}(A_q) \Big)_{1-\frac{1}{p},p} &= \Big\{ g \in \Bso : \text{ div } g = 0, \ g\cdot \nu|_{\Gamma} = 0 \Big\} \equiv \Bt(\Omega) \label{B-1.15b}\\
		&\hspace{4cm} \text{ if } 0 < 2 - \frac{2}{p} < \frac{1}{q}; \text{ or } 1 < p < \frac{2q}{2q - 1}.\nonumber
		\end{align}	
	\end{subequations}
	
	\begin{rmk}\label{B-rmk-1.4} The intended goal of the present paper is to obtain the sought-after stabilization result in a function space that does not recognize boundary conditions of the I.C. Thus, we need to avoid the case in (\ref{B-1.15a}), as this implies a Dirichlet homogeneous B.C. Instead, we need to fit into the case (\ref{B-1.15b}), where the conditions div $g \equiv 0$ and $\ds g \cdot \nu|_{\Gamma} = 0$ are just features of the underlying space $\lso$, see (\ref{B-1.4}). We shall then impose the condition $\ds 2 - \rfrac{2}{p} < \rfrac{1}{q} $ Moreover, for the linearized feedback $w$-problem (\ref{B-5.3}) below of the translated non-linear feedback $z$-problem (\ref{B-1.29}), the \uline{final well-posedness and global uniform stabilization result}, Theorems \ref{B-Thm-5.1} and \ref{B-Thm-5.4} hold in general for $2 \leq q < \infty$. However, in the analysis of well-posedness and stabilization of the nonlinear N-S translated feedback $z$-problem (\ref{B-1.29}), we shall need to impose a constraint $q > d = $ dim $\Omega$, in particular $q > 3$, see Eq \eqref{B-7.19}, to obtain the embedding $\ds W^{1,q} \hookrightarrow L^{\infty}(\Omega)$ in case of interest $d = 3$. In conclusion, via \eqref{B-1.15b}, the range of $p$ is $\ds 1 < p < \rfrac{6}{5}$ for $d = 3$; and $\ds 1 < p < \rfrac{4}{3}$ for $d = 2$.In such setting, the compatibility conditions on the boundary of the initial conditions are not recognized. This feature is precisely our key objective within the stabilization problem and removes the shortcoming of the prior Hilbert space setting noted below point (ii) above. For $d = 3$, the space is (\ref{B-1.15b}) in `close' to $L^3(\Omega)$. To appreciate this relationship, we note that $L^3(\BR^3)$ is a `critical space' for the $3-d$ \uline{uncontrolled} Navier-Stokes equations on the full space. See Remark \ref{B-rmk-1.6} below.
	\end{rmk}
	
		\subsubsection{Main contributions of the present paper: for dim $\Omega = d = 2, 3,$ local-in-space well-posedness on the space of maximal regularity $\ds \xipqs(\BA_{_{F,q}})$ of the N-S dynamics (\ref{B-1.1}) as well as local  exponential uniform stabilization near $y_e$ on the space $\ds \Bto, \ q > d, \ 1 < p < \rfrac{2q}{2q-1}$ with finite dimensional feedback control pair $\{ v,u \}$ on $\{ \wti{\Gamma}, \omega \}$.}\label{B-Sec-1.4.4}
	
	\noindent In line with Section \ref{B-Sec-1.3.2}, the present main contributions are obtained by means of the feedback control pair $\{ v, u \}$ acting on $\{ \wti{\Gamma}, \omega \}$, see Fig. 2, with $v$ (as well as $u$) finite dimensional also for $d = 3$. This is the content of the next main Theorem \hyperref[B-Thm-A]{A}.
	
	\begin{namedthm*}{Main Theorem A}\label{B-Thm-A}
		(On problem (\ref{B-1.1})). Let $\Omega$, dim $\Omega = d = 2,3,$ be a bounded domain of class $C^2$, thus satisfying the Helmholtz decomposition assumption of Definition \ref{B-1.1}. Let $\wti{\Gamma}$ be an arbitrary small, open connected subset of $\Gamma = \partial \Omega$, of positive measure, supporting the corresponding arbitrary small interior collar $\omega$ (Fig. 2). With reference to the N-S dynamics (\ref{B-1.1}), consider a given equilibrium solution $y_e$ of problem (\ref{B-1.2}), as guaranteed by Theorem \ref{B-Thm-1.1}. Assume $y_e$ is unstable, that is, the eigenvalues of the corresponding Oseen operator $\calA_q$ satisfy condition \eqref{B-1.3}. Let $\ds q > d, \ 1 < p < \rfrac{2q}{2q-1}$. Thus $\ds 1 < p < \rfrac{6}{5}$ for $d  = 3$ and $1 < p < \rfrac{4}{3}$ for $d = 2$. For $\rho > 0$, define the ball $\calV_{\rho}$ in $\Bto$ by
		\begin{equation}\label{B-1.16}
		\calV_{\rho} \equiv \Big\{ y_0 \in \Bto: \norm{y_0 - y_e}_{\Bto} < \rho \Big\}, \quad \rho > 0.
		\end{equation}		
		\noindent There exists $\rho_0 > 0$ sufficiently small, such that, if $0 < \rho < \rho_0$, for each initial condition $y_0 \in \calV_{\rho}$ there exist a tangential boundary feedback controller $v$ and a tangential-like interior feedback controller $u$, defined respectively by 
		\begin{equation}\label{B-1.17}
		v = F(y - y_e), \ v \cdot \nu |_{\Gamma} = 0; \quad u = \wti{G}(y - y_e) 
		\end{equation}
		\noindent through bounded operators $\ds F \in \calL \big(\lso, L^q(\wti{\Gamma}) \big)$, and $\ds \wti{G} \in \calL \big( \lso \big)$, defined explicitly in \eqref{B-1.23}, \eqref{B-1.24} below, both finite-dimensional, with $v$ supported on $\wti{\Gamma}$ and tangential along $\wti{\Gamma}$, and $u$ with tangential-like internal action $\ds (mu) \tau = (mu^1)\tau_1 + (mu^2)\tau_2$ for $d=3$, supported on a collar $\omega$ of $\wti{\Gamma}$, such that the corresponding closed loop system (\ref{B-1.1}) due to the action of such pair $\{ v,u \}$ 		
		\begin{subequations}\label{B-1.18}
		\begin{align}
		y_t(t,x) - \nu_o \Delta y(t,x) +  (y \cdot \nabla)y + \nabla \pi(t,x) - (m(x) \wti{G}(y - y_e))\tau &= f(x)  &\text{ in } Q\\ 
		\text{div} \ y &= 0  &\text{ in } Q\\
		\begin{picture}(275,0)
		\put(-35,10){ $\left\{\rule{0pt}{38 pt}\right.$}\end{picture}
		y &= F(y - y_e) &\text{ on } \Sigma\\
		y(0,x) &= y_0(x) &\text{ in } \Omega
		\end{align}
		\end{subequations}
		\noindent has the following two properties:\\
		
		\noindent (a) the feedback system (\ref{B-1.18}) is well-posed as a non-linear s.c. semigroup on the space of maximal regularity, see \eqref{B-6.8}, of the operator $\ds \BA_{_{F,q}}$:
		\begin{align}
		\xipqs\big( \BA_{_{F,q}} \big) &= L^p \big( 0, \infty; \calD \big( \BA_{_{F,q}} \big)  \big) \cap W^{1,p} \big( 0, \infty; \lso \big) \label{B-1.19}\\
		&\subset \xipq = L^p \big( 0, \infty; W^{2,q}(\Omega) \big) \cap W^{1,p} \big( 0, \infty; \lso \big) \label{B-1.20}\\
		\calD \big( \BA_{_{F,q}} \big) &\equiv \big\{ \varphi \in W^{2,q}(\Omega) \cap \lso : \varphi|_{\Gamma} = F \varphi \big\}, \label{B-1.21}
		\end{align}
		\noindent where $\ds \BA_{_{F,q}}$ defined in (\ref{B-5.4}) or \eqref{B-6.2}, is a-fortiori the generator of a strongly continuous analytic uniformly stable semigroup on either $\lso$ or $\Bto$, describing the linearized $w$-problem in (\ref{B-5.3}) in feedback form.\\
		
		\noindent (b) such closed loop system (\ref{B-1.18}) is exponentially stable on $\ds \Bto$: by taking, if necessary, $\rho_0 > 0$ further smaller, there exists a constant $\wti{\gamma}, \ 0 < \wti{\gamma} < \abs{Re \ \lambda_{N+1}}$, and a constant $C_{\wti{\gamma}} \geq 1$, depending on $q$, such that
		\begin{equation}\label{B-1.22}
		\norm{y(t) - y_e}_{\Bto} \leq C_{\wti{\gamma}} e^{- \wti{\gamma} t}\norm{y_0 - y_e}_{\Bto}, \quad t \geq 0, \ y_0 \in \calV_{\rho}.
		\end{equation}
		\noindent The bounded finite-dimensional operators $\ds F \in \calL \big( \lso, L^q(\wti{\Gamma}) \big)$ and $\ds \wti{G} \in \calL \big( \lso \big)$ have the following form:
		\begin{align}
		F(y - y_e) &= \sum_{k=1}^K \big< P_N (y-y_e), p_k \big>_{_{W^u_N}} f_k, \quad \text{supported on } \wti{\Gamma} \label{B-1.23}\\
		\wti{G}(y - y_e) &= \sum_{k=1}^K \big< P_N (y-y_e), q_k \big>_{_{W^u_N}} u_k; \ K = \sup \{ \ell_i: i = 1,\dots,M \}. \label{B-1.24}
		\end{align}
		\noindent Here, $P_N$ is the projector, given explicitly in (\ref{B-3.3a}), of $\lso$ onto $W^u_N, \ \lso = W^u_N \oplus W^s_N, \ W^u_N = P_N \lso$ being the finite dimensional subspace of $\lso$, spanned by the generalized eigenvectors corresponding to the unstable eigenvalues in \eqref{B-1.12} of the Oseen operator. Here, $\ds \ip{ \ }{ \ }_{_{W^u_N}}$ denotes the duality paring between $\ds W^u_N \in \lso$ and $(W^u_N)^* \in (\lso)' = \lo{q'}$, by (\ref{B-1.7}): $\ds \ip{h_1}{h_2}_{_{W^u_N}} = \int_{\Omega} h_1 h_2 \ d \Omega$. The vectors $p_k, q_k$ in $(W^u_N)^* \subset \lo{q'}$ as well as the boundary vectors $f_k$ are \underline{constructed explicitly} in Section \ref{B-Sec-4}, in the proof of Theorem \ref{B-Thm-4.1}. In particular (see Appendix \ref{B-app-B}, in particular (\ref{B-B.5}))
		\begin{equation}\label{B-1.25}
		f_k \in \calF = \text{span } \Bigg\{ \frac{\partial \varphi^*_{ij}}{\partial \nu} : \ i = 1, \dots, M;\ j = 1, \dots, \ell_i \Bigg\} \in W^{2-\rfrac{1}{q},q}(\Gamma), \ q \geq 2,
		\end{equation}
		\noindent $\rfrac{1}{q} + \rfrac{1}{q'} = 1$, where $\varphi^*_{ij} \in W^{3,q}(\Omega)$, see (\ref{B-B.5}) in Appendix \ref{B-app-B}, are the eigenfunctions of the adjoint of the Oseen operator, see (\ref{B-4.8a}), corresponding to the $M$ distinct unstable eigenvalue $\lambda_i$, with geometric multiplicity $\ell_i$. Finally, $K = \sup \{ \ell_i; \ i = 1, \dots, M \}$. \qedsymbol 
	\end{namedthm*}
		
		\noindent The \uline{key}, \uline{new} feature of the above Theorem \hyperref[B-Thm-A]{A} is that the localized tangential boundary feedback control $v$ (the operator $F$ supported on $\wti{\Gamma}$) can be chosen to be \uline{finite dimensional} also for $d = 3$, as documented by \eqref{B-1.23}. This was a recognized open problem in the literature, even for a boundary feedback control $v$, let alone a tangential one, see Section \ref{B-Sec-1.3.2}. To obtain such finite dimensionality, it was critical to abandon the prior Hilbert-Sobolev setting of the literature and descend to a lower level regularity setting of the Besov space in (\ref{B-1.15b}) where the compatibility conditions are not recognized, Remark \ref{B-rmk-1.4}.

	\begin{rmk}\label{B-Rmk-1.5}
		(On the constants $\gamma_0 \approx \abs{Re \lambda_{N+1}}, \wti{\gamma}$ of decay rates) With reference to the instability assumption \eqref{B-1.3}, throughout the paper we let $\gamma_0$ be a preassigned constant, just below $\ds \abs{Re  \lambda_{N+1}}: 0 < \gamma_0 = \abs{Re  \lambda_{N+1} } - \varepsilon, \ \varepsilon > 0$ fixed and arbitrarily small. Next in Theorem \ref{B-Thm-4.1}, we let $\gamma_1 > 0$ be an arbitrarily large constant, in particular $\gamma_1 > \gamma_0$. Then, the conclusion of Theorem \ref{B-Thm-4.1}, Eqts \eqref{B-4.4} and \eqref{B-4.5}, gives the global exponential decay $\ds e^{- \gamma_1 t}$ for the finite dimensional feedback $w_N$-dynamics \eqref{B-4.6} in, respectively, the $\lso$- and $\Bto$-norms. Next, precisely because $\gamma_1 > \gamma_0$, Theorem \ref{B-Thm-5.1}, Eqts \eqref{B-5.17} and \eqref{B-5.19} (where full proof is in \cite{LT2:2015}) yields the global exponential decay $\ds e^{- \gamma_0 t}$ for the feedback linearized $w$-dynamics \eqref{B-5.3}, equivalently the s.c. analytic semigroup generated by the operator $\BA_{_{F,q}}$ on either $\lso$ or $\Bto$. Finally, the proof of Theorem \hyperref[B-Thm-B]{B} given in Section \ref{B-Sec-8} yields the local exponential decay $\ds e^{- \wti{\gamma} t}$ for $z(t) = y(t) - y_e$ in feedback form, \eqref{B-1.18}, \eqref{B-1.23}, \eqref{B-1.24} for all I.C. sufficiently close to zero for the variable $z$, respectively to $y_e$ for the variable $y$. Remark \ref{B-rmk-8.1} supports the intuitive expectation that ``the larger the constant $\gamma_0 \approx \abs{Re \lambda_{N+1}}$,  the larger $\wti{\gamma}$". In conclusion, the desired exponential decay rate is dictated by $\gamma_0 \approx \abs{Re \lambda_{N+1}}$. See also \eqref{B-C.10} in Appendix \ref{B-app-C}.
	\end{rmk}

	\begin{rmk}\label{B-rmk-1.6}
		\noindent (Criticality of the space $L^3$) In the case of the uncontrolled N-S equations defined on the full space $\BR^3$, extensive research efforts have recently lead to the discovery that the space $L^3(\BR^3)$ is a `critical' space for the issue of well-posedness. Assuming that some divergence free initial data in $L^3(\BR^3)$ would produce finite time singularity, then there exists a so-called minimal blow-up initial data in $L^3(\BR^3)$ \cite{GKP:2010},\cite{JS:2013}. More precisely, let $y$ now be a solution of the N-S equations (\hyperref[B-1.1]{1.1a-b-c-d}) with $m \equiv 0, f \equiv 0$, as defined on the whole space $\BR^3$. For any divergence free I.C. $y_0 \in L^3(\BR^3)$, denote by $T_{max} (y_0)$ the maximal time of existence of the mild solution starting from $y_0$. Define	
		\begin{equation}
		\rho_{max} = \sup \big\{ \rho: \ T_{max}(y_0) = \infty \text{ for every divergence free } y_0 \in L^3(\BR^3), \text{ with } \norm{y_0}_{L^3(\BR^3)} < \rho \big\}.  \nonumber
		\end{equation}
		
		\noindent The following result holds \cite[Theorem 4.1, p 14]{JS:2013}: Suppose $\rho_{max} < \infty$. Then there exists some $\ds y_0 \in L^3(\BR^3)$, $\norm{y_0}_{L^3(\BR^3)} = \rho_{max}$, whose $T_{max}(y_0) < \infty$, i.e. the corresponding solution blows up in finite time. Of the numerous works that followed the pioneering work of \cite{Ler:1934} along this line of research, we quote in addition \cite{Ser:1962}, \cite{Ser:1963}, \cite{ESS:1991}, \cite{RS:2009}. Frequency localized regularity criteria for the N-S in $\BR^d$, $d = 3$, related also to \cite{ESS:1991} are given in \cite[p 127]{BG:2017}.
		
	\end{rmk}		
	\subsection{Orientation.}\label{B-Sec-1.5}
	Given the complexity of the problem, whose solution proceeds through various phases, we insert the present encompassing orientation at the very outset of our treatment, even though its full content can be documented and understood only after considerable further reading of the present paper. One may wish to refer back to it as reading proceeds.\\
	
	\noindent \textbf{The Stabilization Feedback Control Paradigm: purely boundary control action versus arbitrarily short portion of the boundary.} 
	
	\noindent \uline{Case 1: tangential boundary feedback control action on an arbitrarily small portion of the boundary.}  First, ideally, one would like to establish uniform stabilization of the above problem (\ref{B-1.1}) by use of \underline{only} the boundary feedback control $v$ (thus, with localized interior, tangential-like control $u \equiv 0$), subject to two additional desirable features (regardless, at this stage, of its finite dimensionality): 
	\begin{enumerate}[(i)]
		\item the boundary control $v$ is applied only on an (arbitrarily) small portion $\wti{\Gamma}$ of the boundary $\Gamma$; 
		\item such $v$ acts only tangentially along $\wti{\Gamma}$, so that the normal component is not needed (a sort of minimal control action). Tangential actuation is attractive and technologically feasible: it is described as implementable in the engineering community, by means of jets of air \cite[p1696]{BLK:2001}, \cite{Kee:1998}, \cite{Bo:2016}.
	\end{enumerate}
	
	\noindent \textit{\textbf{Is such idealized purely boundary, tangential control $v$ acting only on a portion $\wti{\Gamma}$ of the boundary $\Gamma$ a possible stabilizing control? The answer is in the negative.}} As is known since the studies of boundary feedback stabilization of classical parabolic equations with Dirichlet boundary traces in the feedback loop as acting on the Neumann boundary conditions \cite{LT:1983}, a critical potential obstruction arises already at the level of the finite-dimensional analysis: more precisely, at the level of enforcing feedback stabilization with large decay rate of the (assumed unstable) \uline{finite-dimensional projected $w_N$-system} (\hyperref[B-3.8a]{3.8a-b}) of the linearized $w$-problem (\ref{B-1.28}) (with $u = 0$). To achieve this requirement, one needs to verify the algebraic Kalman (or Hautus) rank conditions, corresponding to the unstable eigenvalues $\lambda_i$ in (\ref{B-1.3}) with geometric multiplicity $\ell_i$ of the linearized Oseen operator; actually, equivalently, of its adjoint. In the present case, these turn out to be: rank $W_i = \ell_i$, see the matrix $W_i$ in (\ref{B-4.9}), with entries restricted only on $\wti{\Gamma}$, for each distinct unstable eigenvalue $\lambda_i, i = 1, \dots, M$ in (\ref{B-1.12}). Such entries involve the normal derivatives $\ds \{ \partial_\nu \varphi_{i j}^*, \ j = 1, \dots, \ell_i \}$ on $\Gamma$ of the eigenvectors $\ds \{ \varphi_{i j}^*, \ j = 1, \dots, \ell_i \}$ of the adjoint $\calA_q^*$ of the Oseen operator $\calA_q$ in \eqref{B-1.11}. See \eqref{B-4.8a}. In turn, such algebraic controllability conditions are equivalent to the unique continuation property of the Oseen eigenproblem (\hyperref[B-C.1a]{C.1a,b,c})$\implies$(\ref{B-C.2}) of Appendix \ref{B-app-C}. Actually, and equivalently, of its adjoint eigenproblem. Such unique continuation property with over-determined conditions $\varphi|_{\wti{\Gamma}} \equiv 0, \ \partial_{\nu} \varphi|_{\wti{\Gamma}} \equiv 0$ only on a portion $\wti{\Gamma}$ of the boundary $\Gamma$ as in \eqref{B-C.1c} is \underline{false}. In fact, as collected in Appendix \ref{B-app-C}, reference \cite{FL:1996} provides a simple counterexample to such unique continuation property even for the Stokes problem $(y_e = 0)$ on the 2-dimensional half-space $\{ (x, y) :x \in \BR^+, y \in \BR \}$, with over-determination on the infinite boundary $\{ x = 0 \}$. Such counterexample on the half-space can then be transformed in a counterexample of the unique continuation property on a bounded domain $\Omega$ with over-determination on any sub-portion of its boundary $\partial \Omega$. Thus, stabilization (with large decay rate) of the (assumed unstable) finite dimensional projected $w_N$-system (\hyperref[B-3.8a]{3.8a,b}) - hence of the linearized $w$-problem (\ref{B-1.28}) - by means only of the boundary feedback control $v$ \uline{active} only on  the small portion $\wti{\Gamma}$ of the boundary $\Gamma$  is \underline{not possible} [and thus with localized interior, tangent-like control $u = 0$ on $\omega$]. This is \uline{in contrast} with purely parabolic (heat-type) stabilization, where the required unique continuation result is available \cite{RT:2008}, \cite{RT1:2009}.\\
	
	\noindent \uline{Case 2: the necessity of complementing the localized tangential boundary control $v$ with a corresponding localized interior tangential-like control $u$. See Fig. 2}. If one insists on a boundary control action $v$ active only on a portion of the boundary $\wti{\Gamma}$, one then needs an \uline{extra condition}. \uline{A weakest extra condition} is to complement such $v$ (as it was introduced in \cite{LT1:2015}, \cite{LT2:2015}) with a localized, interior, tangential-like control $u$, acting on an arbitrarily small patch $\omega$, supported by $\wti{\Gamma}$. This is a sort of minimal extra requirement for keeping $v$ acting only on $\wti{\Gamma}$. The role of this additional localized interior, tangential-like control $u$ is to guarantee that the corresponding unique continuation property \eqref{B-4.12}$\implies$\eqref{B-4.14} of Lemma \ref{B-lem-4.3}, augmented this time with the interior condition $\varphi^* \cdot \tau \equiv 0$ on $\omega$ in \eqref{B-4.12c}, now holds true. This is equivalent to the UCP of Problem \#2: (\hyperref[B-C.6a]{C.6a,b,c}) $\implies$ \eqref{B-C.7} in Appendix \ref{B-app-C}. In short: the unique continuation property of Problem \#1 (\hyperref[B-C.1a]{C.1a,b,c}) $\implies$ (\ref{B-C.2}) without the extra condition $\varphi \cdot \tau = 0$ on $\omega$, is \uline{false}, and this is then ``corrected" by falling into the unique continuation of Problem \#2, (\hyperref[B-C.6a]{C.6a,b,c}) $\implies$ \eqref{B-C.7} augmented with the interior condition $\varphi \cdot \tau \equiv 0$ on $\omega$, which is \uline{true}. Technically, $\ds rank \ [-\nu_o W_i | U_i] = \ell_i$ as in \eqref{B-4.11b} is \uline{true}; while $\ds rank \ [-\nu_o W_i] = \ell_i$ is \uline{false}. Consequently the correspondingly \uline{augmented controllability matrix} in (\ref{B-4.11b}) satisfies the required Kalman rank conditions. In this sense, therefore, \uline{the results of the present paper} (ultimately, Theorem \hyperref[B-Thm-B]{B} yielding tangential null-feedback stabilization in the vicinity of the unstable origin, of the translated $z$-problem (\ref{B-1.27})) are \uline{optimal, also in terms of the smallness of the required control action for $v$ and $u$}. Moreover, $v$ is shown here for the first time to be \uline{finite dimensional} also in the case $d = 3$. This is the key new contribution of the present work (finite dimensionality of the internal tangential-like control $u$ is not an issue, see \cite{LT2:2015}). Recall also Remark \ref{B-rmk-1.3}.\\	
	
	\noindent \uline{Case 3: tangential boundary control $v$ on the whole boundary $\Gamma$.} If, on the other hand, one insists on \uline{only} exercising tangential boundary control action $v$ - and thus dispensing altogether with the localized, interior, tangential-like control $u$ - then such boundary control action $v$ will have to be applied, as a first preliminary attempt, to the \uline{entire} boundary $\Gamma$. \underline{Would then be possible to establish uniform } \uline{stabilization with only a feedback control $v$ acting tangentially on the entire boundary $\Gamma$} (regardless of its finite dimensionality)? The answer is Yes, \textit{as long as} the corresponding unique continuation property (UCP) with over-determination on the whole boundary $\Gamma$ \textit{holds true}: by duality Problem \#3 (\hyperref[B-C.8a]{C.8a-b-c})$\implies$\eqref{B-C.9} in Appendix \ref{B-app-C}. The proof of such version of uniform stabilization with only feedback control $v \ (u \equiv 0)$ acting tangentially on the whole boundary $\Gamma$ and being finite dimensional is unchanged, subject only to invoking said UCP for the corresponding unstable equilibrium solution $y_e$. Thus, the obstruction is again the validity of the unique continuation property of the Oseen eigenproblem (corresponding to the unstable distinct eigenvalues $\lambda_i, i = 1,\dots,M$, in (\ref{B-1.12}), with - this time - over-determination $\varphi|_{\Gamma} \equiv 0, \ \partial_{\nu} \varphi|_{\Gamma} \equiv 0$ on the entire boundary $\Gamma$: that is Problem \#3, implication (\ref{B-C.8})$\implies$(\ref{B-C.9}) in Appendix \ref{B-app-C}. Is such UCP always true? Only partial results are presently known.
	
	\begin{enumerate}[a)]
		\item Such required unique continuation property is true in dimension $d=2,3,$ if the equilibrium solution $y_e = 0$ (Stokes eigenproblem) or, more generally,  $y_e$ is in a sufficiently small ball of the origin in the $W^{1, \infty}$-norm. Several very different proofs are given in \cite{RT:2008} and \cite{RT1:2009}: As noted in Remarks \ref{B-rmk-1.2} and \ref{B-rmk-2.2}, the case $y_e = 0$ is actually physically quite important as it occurs for instance when the forcing function $f$ in (\ref{B-1.1a}) or (\ref{B-1.2a}) is a conservative vector field $f = \nabla g$ (say an electrostatic field): in which case a solution of problem (\hyperref[B-1.2a]{1.2a,b,c}) is  $y_e = 0$ and $\pi = g$, modulo constant. Moreover, the ``good" equilibrium solutions (which yield the required unique continuation property with over- determination on the entire boundary $\Gamma$) form an open set in, say, the $W^{1,\infty}$ space topology: if $y_e$ is ``good", then there is a full ball in the $W^{1, \infty}$-topology that contains ``good" $y_e$ \cite{RT:2008}, \cite{RT1:2009}.\\
		
		\noindent What is the implication, in any, of the validity of the corresponding UCP in the case $y_e = 0$ on the problem of the present paper?\\
		
		\noindent \textbf{Enhancement of decay rate: See Remark \ref{B-rmk-2.2}.} Of course, with $y_e = 0$, the corresponding Stokes problem (which now replaces the general Oseen problem) is already uniformly stable, with, say a decay rate $- \abs{Re (\lambda_1)}$ where $Re \ \lambda_1 < 0$ for the Stokes operator $-A_q$ in \eqref{B-1.8}. A most valuable variation of the problem under investigation, whose solution is contained in the treatment of the present paper, is as follows: \uline{enhance the stability} of the linearized (uniformly stable) uncontrolled $w$-problem (\hyperref[B-1.28a]{1.28a,b,c,d}) (with $u \equiv 0, v \equiv 0$) from the given natural margin $-|Re(\lambda_1)|$ to an arbitrarily preassigned decay rate $-k^2$, by means of only a tangential boundary finite dimensional feedback control $v$ of the same form as the operator $F$ in  (\ref{B-1.23}) as applied this time to the entire boundary $\Gamma$. To this end, it suffices to apply the procedure of the present paper (with $u \equiv 0$) to a finite dimensional projected space spanned by the eigenvectors of the Stokes operator corresponding to its finitely many eigenvalues $\lambda_i$ with $ \abs{Re(\lambda_i)} \leq k^2$. This, in turn, will provide \underline{stability enhancement} of the non-linear problem (\ref{B-1.1}) in the vicinity of $y_e = 0$ (or small $y_e$), with only $v = F(y - y_e)$ on all of $\Gamma$ ($u \equiv 0$). 
		
		\item In the two dimensional case, $d=2$, there is a genericity result \cite{BL:2012} about the validity of the required unique continuation problem  (\ref{B-C.8})$\implies$(\ref{B-C.9}) in Appendix \ref{B-app-C} with over-determination on the whole boundary.
	\end{enumerate}
	
	\subsection{Comparison with the literature}\label{B-Sec-1.6}
	
	\noindent To put the present paper in the context of the literature, we first summarize its main contributions, expanding on the Orientation of Section \ref{B-Sec-1.5}.\\
	
	\noindent \textbf{Contributions of the present paper.}
	\begin{enumerate}[1.] 
		\item The stabilizing control strategy of the present paper is an \uline{optimal} result for the feedback uniform stabilization of the N-S dynamics \eqref{B-1.1} for $d = 3$. It consists of a pair $\{v,u\}$ of \uline{finite} dimensional feedback controls requiring a \uline{``minimum" control action} or \textit{support} $\{\wti{\Gamma}, \omega \}$ and minimal number $K$: a localized $K$-dimensional tangential boundary feedback control $v$ in (\ref{B-1.1d}) acting on an arbitrarily small open connected portion $\wti{\Gamma}$ of the boundary $\Gamma$, $v \cdot \nu = 0$ on $\wti{\Gamma}$, implemented as $v = F(\cdot)$, with $F$ a static operator, and a localized interior $K$-dimensional feedback control $u$ in (\ref{B-1.1a}) acting tangential-like (parallel to the boundary) on an arbitrarily small interior patch $\omega$ supported by $\wti{\Gamma}$ (Fig. 2). The number $K$ is $K = \sup \{ \ell_i: i = 1,\dots,M \}$, the maximal geometric multiplicity of the distinct unstable eigenvalues of the Oseen problem. As documented in the Orientation of Section \ref{B-Sec-1.5}, the interior tangential-like controller $u$ \uline{cannot be dispensed with}, if one insists in controlling from an arbitrarily small portion $\wti{\Gamma}$ of the boundary. This is due to the counter-example in \cite{FL:1996} to the unique continuation property of the over-determined Oseen eigenproblem in (\hyperref[B-C.1]{C.1a-b-c}), (\ref{B-C.2}) of Appendix \ref{B-app-C}, leading to the implication noted in (\ref{B-C.5}). Thus: minimal support $\{\wti{\Gamma}, \omega \}$, minimal number $K$, no normal component for $v$ and $u$. 
		
		\item The main contribution of the present paper over the literature is in asserting that the tangential boundary feedback control $v$ is \textit{finite dimensional also for $d = 3$ in full generality}, in fact $K$-dimensional, through a constructive algorithm. This is an affirmative solution to a recognized open problem. To achieve this desired goal, it was necessary to abandon the Hilbert-Sobolev setting of all prior literature on this problem and employ, for the first time, save for \cite{LPT.1}, a critical Besov space framework $\ds \Bto$ in \eqref{B-1.15b} with tight indices, $\ds 1 < p < \rfrac{6}{5}, \ q > 3$ for $d = 3$, `close' to $L^3(\Omega)$, which does not recognize compatibility conditions as explained in Remark \ref{B-rmk-1.4}.
		
		
		
		\item We have noted in 1 the positive feature in that the finite dimensionality $K$ of the feedback stabilizing controllers $v$ in (\ref{B-1.1c}) and $u$ in (\ref{B-1.1a}) is equal to the max of the \textit{geometric} multiplicity – not the \underline{algebraic} multiplicity as in \cite{BT:2004}, \cite{BLT1:2006}, \cite{BLT2:2007}, \cite{BLT3:2006}, \cite{B:2011}, \cite{B:2018} -- of the distinct unstable eigenvalues of the Oseen operator. This is due to the proof, given originally in \cite{LT1:2015}, for checking the controllability condition (\ref{B-4.11b}) of the finite dimensional projected system $w_N$ in (\ref{B-3.8}). Not only does this proof rest on the max geometric rather than the max algebraic multiplicity of the unstable eigenvalues, but it also much simplifies the somewhat complicated and unnecessary Gram-Schmidt orthogonalization process of \cite{BT:2004}, \cite{B:2011} by employing direct, explicit, sharp tests.
		
		\item Finally, the present work offers a much more attractive and preferable proof over past literature of the ultimate non-linear result: the well-posedness and uniform stabilization of the original (modulo translation) non-linear $z$-problem (\ref{B-7.1}), given in Sections \ref{B-Sec-7} and \ref{B-Sec-8}. This new proof now rests on the fundamental preliminary property of \uline{maximal regularity of the linearized boundary feedback problem} (\ref{B-5.3}) or \eqref{B-6.2a}, or generator $\BA_{_{F,q}}$, as stated in Theorem \ref{B-Thm-6.1}. Such \uline{maximal regularity-based proof} is much cleaner and effective over the original proof for the non-linear boundary stabilization result as given in \cite{BLT1:2006}; and even more so over the approximation argument of the nonlinear operator $\calN_q$ in (\ref{B-2.23}) given in the case of localized feedback control given in \cite{BT:2004}, \cite[Chapter 4]{B:2011}. For maximal regularity literature, see \cite{KW:2001}, \cite{KW:2004}, \cite{We:2001}, \cite{Dore:2000}, \cite{PW:2002}, following the original contributions \cite{CV:1986}, \cite{DPV:2002}.
	\end{enumerate}
	
	\noindent \textbf{The origin of the studies on the uniform stabilization problem of Navier-Stokes equations.} The problem of boundary feedback stabilization of unstable linear classical \underline{parabolic} equations was investigated extensively in the period, say 1974-1983, see \cite{RT:1975}, \cite{RT:1980}, \cite{RT:1980:2}, \cite{LT:1983}. The study of uniform stabilization of Navier-Stokes equations apparently initiated with the pioneering work of Fursikov \cite{F.1}, \cite{F.2}, \cite{F.3}, first in $2d$, next in $3d$. However, this work used \textit{open-loop boundary controls not closed loop feedback controls}. The nature and dimensionality of the obtained boundary controllers (whether finite or infinite dimensional, whether tangential or otherwise) was not an issue covered by the method of these papers. Fursikov's work was soon followed by paper \cite{BT:2004} which tackled and solved, instead, the (preliminary) problem of uniform stabilization of the Navier-Stokes equations, $d = 2, 3$, by means of a \underline{localized interior} finite dimensional control. This was implemented as a high-gain, Riccati-based feedback control. All these studies-and the subsequent ones till the present work, some of which are noted below - were carried out in a Hilbert-Sobolev-settings [An improvement over \cite{BT:2004} in both content of results and effectiveness of proofs with \uline{spectral based}, explicit interior localized controllers on the same \underline{interior} uniform stabilization problem is contained in the authors' paper \cite{LPT.1}. For the first time, its analysis is carried out in the same critical Besov setting of the present boundary stabilization study].\\
	
	\noindent \textbf{Tangential Boundary feedback stabilization.} Paper \cite{BT:2004} on uniform stabilization by localized \uline{interior} feedback control opened then the way to a first analysis of the \uline{tangential boundary} stabilization problem in \cite{BLT1:2006} via a high gain, Riccati-based boundary control, followed by an axiomatic approach, still Riccati-based, in \cite{BLT2:2007}, both low and high gain, as well as a complementary, spectral-based approach in \cite{BLT3:2006}. These works required some spectral assumptions of the Oseen eigenvalue problem, equivalent to   a unique continuation property for a corresponding overdetermined Oseen eigenproblem. See Appendix \ref{B-app-C}. It was only in \cite{LT1:2015}, \cite{LT2:2015} that uniform stabilization with a localized feedback control pair $\{v,u\}$, as described above, was resolved in an ``optimal" way regarding the minimal amount of their support $\{ \wti{\Gamma}, \omega \}$. Moreover, this setting of \cite{LT1:2015}, \cite{LT2:2015} had the advantage of not requiring any property or assumptions on the distinct unstable eigenvalues of the Oseen operator, as it was the case in prior literature, since the required corresponding unique continuation property can be shown in this context to hold true (Lemma \ref{B-4.3}), [equivalently, by duality, Problem \#2: (\hyperref[B-C.6a]{C.6a-b-c})$\implies$\eqref{B-C.7}] due to an extra condition $\varphi^* \cdot \tau \equiv 0$ in (\ref{B-4.12c}); or $\varphi \cdot \tau \equiv 0$ in $\omega$, in (\ref{B-C.6c}), dictated by the employment of the interior localized tangential-like control $u$. As noted in Section \ref{B-Sec-1.3.2}, in reference \cite{LT2:2015} the issue of finite dimensionality of the tangential boundary feedback controller component was resolved positively only for $d = 2$ and for $d = 3$ only in the case of Initial Conditions being compactly supported. The general case for $d = 3$ was left open. It is resolved here in the affirmative.\\
	
	\noindent \textbf{Oblique boundary stabilization; dynamic boundary feedback.} We have already noted references \cite{JP:2007}, \cite{BT:2011} which use for $d = 3$, oblique (rather than tangential) boundary feedbacks, which moreover are dynamic rather than static. 
	 Finally, reference \cite{B:2018} investigates stabilization with an oblique boundary control-that is one with an additional normal component. The normal component however is not expressed in feedback form. In addition, two strong assumptions K1 and K.2 are made. The first is the simplifying assumption that the distinct unstable eigenvalues of the Oseen operators be semisimple (geometric = algebraic multiplicity). The second assumes that \underline{all} $N$ unstable eigenvalues have dual eigenvectors whose normal derivatives are linearly independent as $L^2$ functions on the whole boundary $\Gamma$. This is much stronger than the conditions (already given in \cite{BLT1:2006}) that require the much weaker property that for each distinct unstable eigenvalue $\lambda_i$ with geometric multiplicity $\ell_i$, only the traces $\ds \partial_{\nu} \varphi_{ij}^*$ as in (\ref{B-C.5}) of Appendix \ref{B-app-C} be linearly independent, $i = 1,2,\dots ,M$; $j = 1, \dots, \ell_i$. In both \cite{B:2011} and \cite{B:2018}, the number of controls equals the max \underline{algebraic} multiplicity of the unstable eigenvalues of the Oseen operator, see \cite[Eq (3.19)]{B:2011}. 

	\subsection{Beginning of the proof of Theorem \hyperref[B-Thm-A]{A}: translated Nonlinear Navier-Stokes $z$-problem and corresponding linearized $w$-problem. Reduction to zero equilibrium}\label{B-Sec-1.7}
	\noindent We return to Theorem \ref{B-Thm-1.1} which provides an equilibrium pair $\{y_e, \pi_e\}$. Then, as in \cite{BLT1:2006}, \cite{LT1:2015} we translate by $\{y_e, {\pi}_e\}$ the original N-S problem (\ref{B-1.1}). Thus we introduce new variables
	\begin{align}\label{B-1.26}
	z = y - y_e, \quad \chi = \pi - \pi_e
	\end{align}
	and obtain the translated problem in $\{z, \chi \}$		
	\begin{subequations}\label{B-1.27}
		\begin{align}
		z_t - \nu_o \Delta z + L_e(z) + (z \cdot \nabla) z + \nabla \chi - (m(x)u)\tau &= 0  &\text{ in } Q \label{B-1.27a}\\ 
		\begin{picture}(0,0)
		\put(-220,-6){ $\left\{\rule{0pt}{33pt}\right.$}\end{picture}
		\text{div} \ z &= 0   &\text{ in } Q \label{B-1.27b}\\
		z &= v &\text{ on } \Sigma \label{B-1.27c}\\
		z(0,x) = z_0 (x) &= y_0(x) - y_e(x) &\text{ on } \Omega \label{B-1.27d}
		\end{align}
		where $v \cdot \nu = 0$ on $\Sigma$ and the first order Oseen perturbation $L_e$ is given by $\ds L_e(z) = (y_e \cdot \nabla)z +(z \cdot \nabla)y_e$ as defined in \eqref{B-1.9}. We shall accordingly first study the local null feedback stabilization of the $z$-problem (\ref{B-1.27}), that is, feedback stabilization in a neighborhood of the origin.\\
	\end{subequations}

	\noindent Our strategy will be to select \textit{constructively} feedback control operators $v = F(z)$ and $u = \wti{G}(z)$, with $v$ tangential $v \cdot \nu = 0$ on $\Gamma$ and supported only on $\wti{\Gamma}$, and $u$ supported only on $\omega$, and both $F$ and $\wti{G}$ bounded and finite dimensional also for $d = 3$.  For $d = 2$ this was achieved in the Hilbert space setting \cite{LT2:2015}. To this end, it will be critical to show global uniform stabilization of the following
	
	\noindent linearization of the non-linear $z$-problem (\ref{B-1.27}) near the equilibrium solution $y_e$
	\begin{subequations}\label{B-1.28}
		\begin{align}
		w_t - \nu_o \Delta w + L_e(w) + \nabla \chi - (m(x)u)\tau &= 0  &\text{ in } Q \label{B-1.28a}\\ 
		\begin{picture}(0,0)
		\put(-180,-5){ $\left\{\rule{0pt}{33pt}\right.$}\end{picture}
		\text{div} \ w &= 0   &\text{ in } Q \label{B-1.28b}\\
		w &= v &\text{ on } \Sigma \label{B-1.28c}\\
		w(0,x) & = w_0 (x) &\text{ on } \Omega \label{B-1.28d}
		\end{align}
	\end{subequations}
	\noindent $v \cdot \nu = 0$ on $\Sigma$. The above main Theorem \hyperref[B-Thm-A]{A} for problem (\ref{B-1.1}) is an immediate corollary of the following main Theorem \hyperref[B-Thm-B]{B} for the translated non-linear $z$-problem (\ref{B-1.27}). 
	
	\noindent \begin{namedthm*}{Main Theorem B}\label{B-Thm-B}(On problem (\ref{B-1.27})) Under the same assumptions and in the same notation of Theorem \hyperref[B-Thm-A]{A}, in particular, $q > 3, 1 < p < \rfrac{6}{5}$ for dim $\Omega = 3$, consider the following feedback version of the translated non-linear $z$-problem (\ref{B-1.27}), corresponding to the abstract version (\ref{B-7.1})
		\begin{subequations}\label{B-1.29}
		\begin{align}
		z_t - \nu_o \Delta z + L_e(z) + (z \cdot \nabla) z + \nabla \chi - \Bigg(m \Bigg(\sum^K_{k=1} \big< P_Nz,q_k \big>_{_{W^u_N}} u_k \Bigg) \Bigg) \cdot \tau &= 0  &\text{in } Q \label{B-1.29a}\\ 
		\begin{picture}(0,0)
		\put(-320,-6){ $\left\{\rule{0pt}{55pt}\right.$}\end{picture}
		\text{div} \ z &= 0   &\text{in } Q \label{B-1.29b}\\
		z = \sum^K_{k=1} \big<P_Nz,p_k \big>_{_{W^u_N}} & f_k  &\text{on } \Sigma \label{B-1.29c}\\
		z(0,x) = z_0 (x) = y_0(x) -& y_e(x). &\text{on } \Omega \label{B-1.29d}
		\end{align}
		\end{subequations}		
		\noindent There exists a positive constant $r_1 > 0$ (identified in (\ref{B-7.27}), (\ref{B-8.18}) such that if 
		\begin{equation}\label{B-1.30}
		\norm{z_0}_{\Bto} \leq r_1,
		\end{equation}  
		\noindent then, the following local well-posedness and uniform feedback stabilization results hold true: 
		\begin{enumerate}[(i)]
			\item the feedback problem (\ref{B-1.29}) admits a unique (fixed point nonlinear semigroup) solution $z$ in the space $\ds \xipqs\big( \BA_{_{F,q}} \big)$ of maximal regularity in \eqref{B-1.19}.						
			\item Moreover, if the constant $r_1 > 0$ in (\ref{B-1.30}) is sufficiently small as in \eqref{B-8.18}, then there exists a constant $\wti{\gamma} > 0$ and a corresponding constant $C_{\wti{\gamma}} \geq 1$, (depending on $q$) such that the guaranteed solution $z$ satisfies the exponential decay
			\begin{equation}\label{B-1.31}
			\norm{z(t)}_{\Bto} \leq C_{\wti{\gamma}} e^{-\wti{\gamma}t} \norm{z_0}_{\Bto}, \quad t \geq 0. \ \qedsymbol
			\end{equation}
		\end{enumerate} 	
		
	\end{namedthm*}  	
	\noindent Remark \ref{B-rmk-8.1} at the end of Section \ref{B-Sec-8} supports qualitatively the intuitive expectation that ``the larger the global decay rate $\gamma_0 \approx \abs{Re \lambda_{N+1}},\ \gamma_0 > 0$ in (\ref{B-5.17}) of Theorem \ref{B-Thm-5.4} of the linearized $w$-problem (\ref{B-1.11}) in feedback form as in (\ref{B-5.3}), the larger the local decay rate $\wti{\gamma}$ in (\ref{B-1.31}).\\	
	
	\noindent The proof of the well-posedness part in $\xipqs$, see \eqref{B-1.19}, of Theorem \hyperref[B-Thm-B]{B} is given (in its concluding arguments) in Section \ref{B-Sec-7}, while the exponential decay (\ref{B-1.31}) is established (in its concluding arguments) in Section \ref{B-Sec-8}. Recalling $z = y - y_e, \chi = \pi - \pi_e$ from (\ref{B-1.9}), we see at once that Theorem \hyperref[B-Thm-B]{B} implies Theorem \hyperref[B-Thm-A]{A}.

 	\section{Abstract models for the non-linear $z$-problem (\ref{B-1.27}) and the linearized $w$-problem (\ref{B-1.28}) in the $L^q$-setting}\label{B-Sec-2}
 	
 	\noindent We shall next provide abstract models for the translated non-linear $z$-problem (\ref{B-1.27}) and its corresponding linearized $w$-problem (\ref{B-1.28}) in the $L^q$-setting. This will be the counterpart (extensions) of these introduced in \cite{BLT1:2006} and used in \cite{BLT3:2006} \cite{LT2:2015}. The $L^q$-setting will require a wealth of non-trivial additional results: from the well-posedness and regularity from the boundary of the stationary Oseen problem (that is, the definition of the Dirichlet map $D$ with range in $\lso: g \longrightarrow Dg = \psi$ in (\ref{B-2.1}) below) to the definition of the adjoint $\ds (\calA_q)^* = \calA_q^*$ for short, (in the $L^{q'} \longrightarrow L^q$ sense) of the Oseen operator $\calA_q$ in \eqref{B-1.11}, to the critical meaning of $D^* \calA_q^* \varphi, \ \varphi \in \calD(\calA_q^*)$. These results will be provided below. They will be the perfect counterpart of those obtained in \cite{BLT1:2006}, in the Hilbert setting. For the original idea of Dirichlet map we refer to \cite{L-T.1}.
 	
 	\subsection{Well-posedness in the $L^q$-setting of the non-homogeneous stationary Oseen problem: the Dirichlet map $D:$ boundary $\longrightarrow$ interior}\label{B-Sec-2.1}
 	
 	\noindent Recalling the first order operator $L_e(\psi) = (\psi \cdot \nabla) y_e + (y_e \cdot \nabla) \psi $ from (\ref{B-1.9}) and introducing the differential expression $\BA \psi = -\nu_0 \Delta \psi + L_e(\psi)$, we consider the stationary, boundary non-homogeneous Oseen problem on $\Omega$: 	
 	\begin{subequations}\label{B-2.1}
 		\begin{align}
 		\BA \psi + \nabla \pi^* &= -\nu_o \Delta \psi + L_e(\psi) + \nabla \pi^* = 0 \label{B-2.1a}\\
 		\begin{picture}(0,0)
 		\put(-35,10){ $\left\{\rule{0pt}{20pt}\right.$}\end{picture}
 		\text{ div } \psi &= 0 \text{ in } \Omega; \quad \psi = g \text{ on } \Gamma, \ g \cdot \nu = 0 \text{ on } \Gamma \label{B-2.1b}
 		\end{align}
 	\end{subequations}

	\begin{rmk}
		Postponing regularity issues to the second part of the present sub-section, our purpose here is to introduce the Dirichlet map $g \longrightarrow \psi$, from the boundary datum to the interior solution of the above Oseen problem, following \cite[Chapter 3]{BLT1:2006}
	\end{rmk}

	\noindent As noted and discussed in \cite[Ch 3, Orientation at p. 21; Appendix A.2, pp 99-102]{BLT1:2006}, problem (\ref{B-2.1}) may not define a unique solution $\psi$; that is, the operator $g \to \psi$ may have a nontrivial (finite dimensional) null space. To overcome this, one replaces in (\ref{B-2.1}) the differential expression $\BA \psi= -\nu_o \Delta \psi + L_e (\psi)$  with its translation $ k +\BA$, for a positive constant $k$, sufficiently large as to obtain a unique solution $\psi$. As seen in the subsequent results below, we can take $k=0$ whenever the Stokes operator is perturbed \underline{only} by  a first order operator such as $\BA \psi= -\nu_o \Delta \psi + (a.\nabla) \psi$, div $a = 0$, with $a$ sufficiently regular. Moreover, as documented in \cite{BLT1:2006} in the Hilbert setting $q=2$ and restated below in the general $L^q$-setting, the expression $D^*(k) \calA_q^*(k)$ does not depend on the translation parameter $k$. Thus, at the end, also in name of simplicity of notation, we are here justified to admit henceforth that problem (\ref{B-2.1}) (with $k=0$) defines a unique solution $\psi$. We shall then denote the `Dirichlet' map $g \longrightarrow \psi$ by $D: \ Dg = \psi$ in the notation of (\ref{B-2.1}).\\
	
	\noindent The following two regularity results of the Oseen equation below are critical for our subsequent development. The are the perfect counterpart of the results given in \cite{BLT1:2006} in the Hilbert setting. To state properly the conclusion of uniqueness, they will refer to the Oseen equation with \underline{only} a first order term, such as
	\begin{subequations}\label{B-2.2}
		\begin{align}
		- \nu_o \Delta \psi + (a \cdot \nabla) \psi + \nabla \pi^* &= 0  \text{ in } \Omega \\ 
		\begin{picture}(0,0)
		\put(-115,0){ $\left\{\rule{0pt}{25pt}\right.$}\end{picture}
		\text{div} \ \psi &= 0 \text{ in } \Omega \\
		\psi &= g \text{ on } \Gamma
		\end{align}
	\end{subequations}
	\begin{thm}\cite[Thm 15, p 37, where a more general result is given]{AR:2010}\label{B-Thm-2.1}
	\ \\Let 
	\begin{equation}
	a \in L^3(\Omega), \ \text{div } a \equiv 0; \ g \in W^{1-\rfrac{1}{q},q}(\Gamma), \ 2 < q < \infty, \ g \cdot \nu = 0 \ on \ \Gamma.
	\end{equation}
	Then problem (\ref{B-2.2}) has a unique solution $(\psi, \pi^*) \in W^{1,q}(\Omega) \times L^q(\Omega)/ \BR$ continuously: there is a constant $C > 0$ such that
	\begin{equation}
		\norm{\psi}_{W^{1,q}(\Omega)} + \norm{\pi^*}_{L^q(\Omega)/ \BR} \leq C \big( 1 + \norm{a}_{L^3(\Omega)} \big)^2 \norm{g}_{ W^{1-\rfrac{1}{q},q}(\Gamma)} \qed
	\end{equation}	
	\end{thm}

	\begin{thm}\cite[Thm 2, p 6, where a more general result is given]{AR:2010}\label{B-Thm-2.2}
		\ \\Let 
		\begin{equation}
		a \in L^3(\Omega), \ \text{div } a \equiv 0; \ g \in W^{-\rfrac{1}{q},q}(\Gamma), \ \rfrac{3}{2} < q < \infty, \ g \cdot \nu = 0 \ on \ \Gamma.
		\end{equation}
		Then problem (\ref{B-2.2}) has a unique solution 
		\begin{equation}
		(\psi, \pi^*) \in L^q(\Omega) \times W^{-1,q}(\Omega)/ \BR
		\end{equation}
		continuously: there is a constant $C_a > 0$ (explicitly depending on the norm of $\ds \norm{a}_{L^3(\Omega)} \ $) such that
		\begin{equation}
		\norm{\psi}_{L^q(\Omega)} + \norm{\pi^*}_{W^{-1,q}(\Omega)/ \BR} \leq C_a \norm{g}_{ W^{-\rfrac{1}{q},q}(\Gamma)} \qed
		\end{equation}	
	\end{thm}

	\noindent We note that, in Theorem \ref{B-Thm-2.2}, we have also $\ds \Delta \psi \in \big( Y_{r',p'}(\Omega) \big)' = $ dual of $\ds Y_{r,p}(\Omega) = \big\{ \varphi \in W^{1,r}_0 (\Omega), \text{div } \varphi \in W^{1,q}_0 (\Omega) \big\}, \ 1 < r,q < \infty,$ but we shall not need this result \cite[p 6]{AR:2010}.\\
	
	\noindent Returning to our Oseen problem (\ref{B-2.1}) of interest, we have $y_e \in W^{2,q}(\Omega) \cap W^{1,q}_0(\Omega)$ from Theorem \ref{B-Thm-1.1}, hence the embedding $W^{2,q}(\Omega) \hookrightarrow C(\overline{\Omega})$ holds true for $d = 3, q > \rfrac{3}{2}$ \cite[p 79]{SK:1989}, \cite[p 97]{A:1975}. Thus, we can apply Theorems \ref{B-Thm-2.1} and \ref{B-Thm-2.2} to problem (\ref{B-2.1}) and obtain the following results, where, with $\psi = Dg$, the range of $D$ is in $\lso$, since div$(Dg) \equiv 0$ in $\Omega$, $(Dg)\cdot \nu |_{\Gamma} = g \cdot \nu |_{\Gamma} = 0$, see (\ref{B-1.4}):	
	
	
		
		\begin{equation}\label{B-2.12}
		\begin{matrix}
		g \in W^{\big(1-\frac{1}{q} \big) (1 - \theta) - \frac{\theta}{q},q}(\Gamma)\\
		g \cdot \nu = 0 \text{ on } \Gamma; \ 0 < \theta < 1
		\end{matrix}
		\longrightarrow 
		\begin{matrix}
		Dg = \psi \in W^{(1-\theta),q}(\Omega) \cap \lso\\
		\end{matrix}\hspace{4.5cm}
		\end{equation}		
		\noindent so that, as $\ds \bigg( 1 - \frac{1}{q} \bigg)(1-\theta) - \frac{\theta}{q}  = 0$ for $\ds \theta = 1 - \frac{1}{q}$, we also obtain 
		\begin{equation}\label{B-2.13}
		g \in U_q \equiv \big\{ \wti{g} \in L^q(\Gamma),\ \wti{g}\cdot \nu = 0 \text{ on } \Gamma \big\} \longrightarrow Dg \in W^{\rfrac{1}{q},q}(\Omega) \cap \lso,
		\end{equation}
		\noindent all continuously. This property will be further complemented by additional information in (\ref{B-3.41}) below. In the Hilbert space setting, $q=2$, we re-obtain the regularity results, that were derived in \cite[Theorem A.2.2 p 102]{BLT1:2006}, where we recall (\hyperref[B-1.6a]{1.6a--c})	
		\begin{equation}\label{B-2.16}
		\begin{matrix}
		g \in H^s(\Gamma), \ -\rfrac{1}{2} \leq s \leq \rfrac{1}{2}\\
		g \cdot \nu = 0 \text{ on } \Gamma
		\end{matrix}
		\longrightarrow 
		\begin{matrix}
		Dg = \psi \in H^{s+\rfrac{1}{2}}(\Omega) \cap H. \hspace{5cm}
		\end{matrix}
		\end{equation}

\subsection{Abstract model for the non-linear translated $z$-problem (\ref{B-1.27}).}

	\noindent We re-write Eq (\ref{B-1.27a}) as $\ds z_t + \BA z + (z \cdot \nabla) z + \nabla \chi - (mu) \tau = 0$ recalling the differential expression $\BA$ defined above (\ref{B-2.1a}), and next subtract $\BA \psi = \BA Dg = - \nabla \pi^*$ from (\ref{B-2.1a}), where presently $g = v$ on $\Gamma, v \cdot \nu = 0$ on  $\Gamma$. We obtain 
	\begin{equation}\label{B-2.17}
	z_t + \BA(z-Dv) + (z \cdot \nabla)z + \nabla (\chi - \pi^*) - (m(x)u) \tau = 0 \quad \text{in } Q
	\end{equation}
	Next we apply to (\ref{B-2.17}) the Helmholtz projector $P_q$, and obtain [notice that $P_q z_t = z_t$, since $z_t \in \lso$ [div$z_t \equiv 0, \ z_t \cdot \nu = v_t \cdot \nu = 0$ on $\Gamma$] since $P_q \nabla (\chi - \pi^*) \equiv 0$:
	
	\begin{equation}\label{B-2.18}
	z_t + P_q \BA (z - Dv) + P_q(z \cdot \nabla z) - P_q ((m(x)u)\tau) \equiv 0
	\end{equation} 
	\noindent where via (\ref{B-2.1a}), (\ref{B-1.9})
	\begin{equation}\label{B-2.19}
		P_q \BA f = - \nu_o P_q \Delta f + P_q \big[ (y_e \cdot \nabla)f + (f \cdot \nabla) y_e \big].
	\end{equation}
	\noindent For $1 < q < \infty$ fixed, we recall the Stokes operator $A_q$ in $\lso$, the perturbation operator $A_{o,q}$ and the Oseen operator $\calA_q$, from \eqref{B-1.8}, \eqref{B-1.10}, \eqref{B-1.11}, respectively.
	\begin{align}
	A_q z = -P_q \Delta z, \quad
	\mathcal{D}(A_q) &= W^{2,q}(\Omega) \cap W^{1,q}_0(\Omega) \cap \lso. \label{B-2.20}\\
	A_{o,q} z = P_q L_e(z) = P_q[(y_e \ . \ \nabla )z + (z \ . \ \nabla )y_e], \ \mathcal{D}(A_{o,q}) &= \mathcal{D}(A_q^{\rfrac{1}{2}}) = W^{1,q}_0 (\Omega) \cap \lso \subset \lso. \label{B-2.21}\\
	\calA_q  = - (\nu_o A_q + A_{o,q}), &\quad \calD(\calA_q) = \calD(A_q) \subset \lso. \label{B-2.22}
	\end{align}
	Finally, we define the projection of the nonlinear portion of (\ref{B-1.27a})
	\begin{equation}\label{B-2.23}
	\calN_q(z) = P_q [(z \cdot \nabla) z]	
	\end{equation}
	
	\noindent Thus, after using (\ref{B-2.20})-(\ref{B-2.23}) in (\ref{B-2.18}), the N-S translated problem (\ref{B-2.18}) can rewritten as the following abstract equation on $\lso, 1 < q < \infty$:
	
	\begin{equation}\label{B-2.24}
	z_t - \calA_q(z - Dv) + \calN_q z - P_q ((mu)\tau) = 0, \ \text{on } \lso, \ v \cdot \nu = 0 \ \text{on } \Gamma
	\end{equation}
	
	\noindent in factor form on $\ds \lso$. Next, extending the original Oseen operator $\ds \lso \supset \calD(\calA_q) \longrightarrow \lso$ to $\ds \calA_q: \lso \longrightarrow \big[ \calD(\calA_q^*)' \big]$, by isomorphism, and retaining the same symbol, we arrive at the definitive abstract model 
	
	\begin{equation}\label{B-2.25}
		\begin{aligned}
		z_t - \calA_q z + \calN_q z + \calA_q Dv - P_q \big[ (mu) \tau \big] = 0 &\text{ on } \big[ \calD(\calA_q^*)' \big]\\
		\begin{picture}(0,0)
		\put(-80,8){ $\left\{\rule{0pt}{20pt} \right.$}\end{picture}
		z(x,0) = z_0(x) = y_0(x) - y_e &\text{ in } \lso
		\end{aligned}
	\end{equation}
	\noindent in additive form, on $\ds \big[ \calD(\calA_q^*) \big]'$.
	
	\subsection{Abstract model of the linearized $w$-problem (\ref{B-1.28}) of the translated $z$-model (\ref{B-1.27}).}
	
	\noindent Still for $1 < q < \infty$, the abstract model (in additive form) of the linearized $w$-problem in (\ref{B-1.28}) is obtained from (\ref{B-2.25}) by dropping the nonlinear term
	\begin{equation}\label{B-2.26}
	\begin{aligned}
	w_t - \calA_q w + \calA_q Dv - P_q \big[ (mu) \tau \big] = 0 &\text{ on } \big[ \calD(\calA_q^*) \big]'\\
	\begin{picture}(0,0)
	\put(-50,8){ $\left\{\rule{0pt}{20pt} \right.$}\end{picture}
	w(x,0) = w_0(x) = y_0(x) - y_e &\text{ in } \lso.
	\end{aligned}
	\end{equation}
	
	\subsection{The adjoint operators $D^*, \ (A_q)^* = A_q^*$ and $ (A_{o,q})^* = A^*_{o,q}, \ (\calA_q)^* = \calA_q^* = - (\nu_o A^*_q + A^*_{o,q})$, $1 < q < \infty$.}\label{B-Sec-2.4}
	
	\noindent (i) Regarding the Helmholtz projection $P_q$ and its adjoint $P_q^*$, we recall from the statement above \eqref{B-1.7} that $\ds P_q \in \calL \big(L^q(\Omega)\big)$, while $\ds P_q^* = P_{q'} \in \calL
	 \big( L^{q'}(\Omega) \big), \ \rfrac{1}{q} + \rfrac{1}{q'} = 1$,\\
	
	\noindent (ii) Define as in (\ref{B-2.13})
	\begin{equation}\label{B-2.27}
	U_q = \big\{ g \in L^q(\Gamma): g \cdot \nu = 0 \text{ on } \Gamma \big\}.
	\end{equation} 
	\noindent We have seen in (\ref{B-2.13}) that
	\begin{equation}\label{B-2.28}
	D: U_q = \big\{ g \in L^q(\Gamma): g \cdot \nu = 0 \text{ on } \Gamma \big\} \longrightarrow W^{\rfrac{1}{q},q}(\Omega) \cap \lso,
	\end{equation} 
	\noindent so that the dual $D^*$ satisfies
	\begin{equation}
	D^*: W^{-\rfrac{1}{q},q'} \longrightarrow L^{q'}(\Gamma).
	\end{equation}
	
	\noindent (iii) The adjoint $\ds A^*_q: L^{q'}_{\sigma}(\Omega) \supset \calD(A^*_q) \longrightarrow \lo{q'}, \ \rfrac{1}{q} + \rfrac{1}{q'} = 1$ of the Stokes operator $\ds A_q$ in (\ref{B-2.20})
	\begin{equation}
	\big< A_q f_1, f_2 \big>_{\ls, L^{q'}_{\sigma}} = \big< f_1, A_q^*f_2 \big>_{\ls, L^{q'}_{\sigma}}, \quad f_1 \in L^q_{\sigma}, \ f_2 \in L^{q'}_{\sigma}
	\end{equation}
	\noindent (duality pairing $\ls \longrightarrow L^{q'}_{\sigma}$) is
	\begin{equation}\label{B-2.31}
	A_q^* f_2 = -P_{q'} \Delta f_2, \ \calD(A_q^*) = W^{2,q'}(\Omega)\cap W^{1,q'}_0(\Omega)\cap \lo{q'}.
	\end{equation}
	\begin{proof}
		For $\ds f_1 \in \calD(A_q) \subset \lso \subset L^q(\Omega)$, so that $\ds A_q f_1 \in \lso$ and $f_2 \in \calD(A_q^*) \subset \lo{q'} \subset L^{q'}(\Omega)$ so that $\ds A_q^* f_2 \in \lo{q'},$ and $P_{q'} f_2 = f_2$, we compute from (\ref{B-2.20}): with $\ds P^*_q = P_{q'}$ by the statement above \eqref{B-1.7}
		\begin{align}
			-\big< A_q f_1, f_2 \big>_{\ls, L^{q'}_{\sigma}} &= \big< P_q \Delta f_1, f_2 \big>_{\ls, L^{q'}_{\sigma}}\\
			&= \big< \Delta f_1, P^*_qf_2 \big>_{L^q, L^{q'}} = \big< \Delta f_1, P_{q'}f_2 \big>_{L^q, L^{q'}} = \big< \Delta f_1, f_2 \big>_{L^q, L^{q'}}\\
			&= \int_{\Omega} f_1 \Delta f_2 \ d \Omega + \int_{\Gamma} \cancelto{}{\frac{\partial f_1}{\partial \nu} f_2} \ d \Gamma - \int_{\Gamma}  \cancelto{}{f_1 \frac{\partial f_2}{\partial \nu}} \ d \Gamma \nonumber\\
			&=\big< f_1, \Delta f_2 \big>_{L^q, L^{q'}} = \big< P_q f_1, \Delta f_2 \big>_{L^q, L^{q'}} \nonumber\\
			&=\big< f_1, P_{q'}\Delta f_2 \big>_{L^q, L^{q'}} = -\big< f_1, A_q^*f_2 \big>_{L^q, L^{q'}} \label{B-2.34}
		\end{align}
		\noindent since $f_1 \in W^{1,q}_0(\Omega)$ by (\ref{B-2.20}); and so $f_1|_{\Gamma} = 0;$ and since $ f_2 \in W^{1,q'}_0(\Omega)$ by (\ref{B-2.31}) and so $f_2|_{\Gamma} = 0$. Moreover, $P_q f_1 = f_1$, since $f_1 \in \lso$. Eqt (\ref{B-2.34}) proves (\ref{B-2.31}).
	\end{proof}

	\noindent (iv) Similarly from $\ds A_{o,q} = P_q L_e: \calD(A_{o,q}) = \calD(A_q^{\rfrac{1}{2}}) = W^{1,q}_0(\Omega) \cap \lso \longrightarrow \lso$, in (\ref{B-2.21}), we obtain
	\begin{equation}\label{B-2.35}
		(A_{o,q})^* = A^*_{o,q} \ \text{(for short)} \ = P_{q'} (L_e)^*: W^{-1,q'}(\Omega) \longrightarrow \lo{q'}
	\end{equation} 
	\noindent where the expression of $(L_e)^*$, which is not needed, is given in \cite[p 55]{BLT1:2006}, \cite[below (54)]{LT1:2015}, \cite{F.1}.\\
	
	\noindent (v) As a consequence of (ii), (iii) we have $\ds (\calA_q)^* = \calA_q^* = - (\nu_0 A^*_q + A^*_{o,q}), \ \calD(\calA^*_q) = \calD(A^*_q)$.
	
	\subsection{The operator $\ds D^* \calA_q^*$}
	\begin{thm}\label{B-Thm-2.3}
		Let $1 < q < \infty$. Let $v \in \calD(\calA^*_q) = \calD(A^*_q) = W^{2,q'}(\Omega) \cap W^{1,q'}_0(\Omega) \cap \lo{q'}$, by (\ref{B-2.31}), $\ds \frac{1}{q} + \frac{1}{q'} = 1$ so that $\ds \frac{\partial v}{\partial \nu} \bigg|_{\Gamma} \in W^{1-\rfrac{1}{q'}, q'}(\Gamma) \subset L^{q'}(\Gamma)$. Let $g \in L^q(\Gamma), g \cdot \nu = 0$ on $\Gamma$. Then 
		\begin{equation}\label{B-2.36}
		\big< D^* \calA^*_q v, g \big>_{L^{q'}(\Gamma),L^{q}(\Gamma)} = \nu_o \bigg< \frac{\partial v}{\partial \nu}, g \bigg>_{L^{q'}(\Gamma),L^{q}(\Gamma)}
		\end{equation}
		\noindent where: $q > 3$ for $d = 3$; and $q > 2$ for $d = 2$.
	\end{thm}
	
	\begin{proof} We shall first prove (\ref{B-2.36}) with $g \in W^{1-\rfrac{1}{q},q}(\Gamma), \ g \cdot \nu  = 0 $ on $\Gamma$; and then extend the validity of (\ref{B-2.36}) to $g \in L^q(\Gamma), \ g \cdot \nu = 0$ on $\Gamma$ by density.
		\noindent By (\ref{B-2.22}), $\ds \calA_q = -(\nu_o A_q + A_{o,q})$. Accordingly, we consider $\ds D^* A^*_q$ in Step 1 and $\ds D^* A^*_{o,q}$ in Step 2.\\
		
		\noindent \underline{Step 1:}
		Let $v \in \calD(A^*_q) = W^{2,q'}(\Omega) \cap W^{1,q'}_0(\Omega) \cap \lo{q'}$, so that $\ds A^*_q v \in \lo{q'}$, and let initially $g \in W^{1 - \rfrac{1}{q}, q}(\Gamma) \subset L^q(\Gamma), g \cdot \nu = 0$ on $\Gamma$, so that $\ds Dg \in W^{1,q}(\Omega) \cap \lso$ by (\ref{B-2.12}) with $\theta = 0$. Our first step is to show
		\begin{equation}\label{B-2.37}
		-\big< D^* A^*_q v, g \big>_{L^{q'}(\Gamma),L^{q}(\Gamma)} = \int_{\Omega} v \Delta (Dg) \ d \Omega + \int_{\Gamma} \frac{\partial v}{\partial \nu} g \ d \Gamma,
		\end{equation}
		\noindent where the integral term under $\Omega$ is well-defined as a duality pairing with $v \in W^{1,q'}_0(\Omega)$ and $\Delta (Dg) \in W^{-1,q}(\Omega)$; while the integral term under $\Gamma$ is well-defined as a duality paring with $\ds \frac{\partial v}{\partial \nu} \bigg|_{\Gamma}$ in $L^{q'}(\Gamma)$ and $\ds g \in L^q(\Gamma)$.\\
		
		\noindent In fact, we compute - and the computations in (\ref{B-2.38}) through (\ref{B-2.40}) below actually work even for $\ds g \in W^{\rfrac{-1}{q},q}(\Gamma)$ so that $\ds Dg \in \lso$ by (\ref{B-2.12}) with $\theta = 1$, and hence $\ds P_q Dg = Dg$
		\begin{align}
		-\big< D^* A^*_q v, g \big>_{L^{q'}(\Gamma),L^{q}(\Gamma)} &= -\big< A^*_q v, Dg \big>_{\lo{q'},\lso} \label{B-2.38}\\
		\text{(by (\ref{B-2.31}))} \hspace{1cm}&= \big< P_{q'} \Delta v, Dg \big>_{L^{q'},L^q} = \big< \Delta v, P_{q'}^* Dg \big>_{L^{q'},L^q} \label{B-2.39}\\
		&= \big< \Delta v, P_q Dg \big>_{L^{q'},L^q} = \big< \Delta v, Dg \big>_{L^{q'},L^q} \label{B-2.40}
		\end{align}
		\noindent where in going from (\ref{B-2.39}) to (\ref{B-2.40}) we have recalled $\ds P^*_{q'} = P_q$ by the statement above \eqref{B-1.7}. Next, we apply Green's theorem in (\ref{B-2.40}) and get
		\begin{equation}
		- \big< D^* A^*_qv, g \big>_{L^{q'}(\Gamma),L^q(\Gamma)}= \int_{\Omega} \Delta v Dg \ d \Omega = \int_{\Omega} v \Delta(Dg) \ d \Omega + \int_{\Gamma} \frac{\partial v}{\partial \nu} g \ d \Gamma - \int_{\Gamma} \cancelto{}{v \frac{\partial Dg}{\partial \nu}} \ d \Gamma \label{B-2.41}
		\end{equation}
		\noindent where we have used $Dg|_{\Gamma} = g$ by definition of $D$, and $v|_{\Gamma} = 0$ as $v \in W^{1,q'}_0(\Omega)$. Then (\ref{B-2.41}) proves (\ref{B-2.37}).\\
		
		\noindent \underline{Step 2:} Let $\ds v \in \calD(A_{o,q}^*) = \calD \big((A_q^*)^{\rfrac{1}{2}} \big) = W^{1,q'}_0(\Omega) \cap \lo{q'}$ by (\ref{B-2.21}) and let $\ds g \in W^{1 - \rfrac{1}{q},q}(\Gamma)$ and $g \cdot \nu = 0$ on $\Gamma$ so that  $\ds Dg \in W^{1,q}(\Omega) \cap \lso$ by (\ref{B-2.12}) with $\theta = 0$. Recall from Theorem \ref{B-Thm-1.1} that $\ds y_e \in W^{2,q}(\Omega) \cap W^{1,q}_0(\Omega)$.\\
		
		\noindent Our second step is to show that 
		\begin{equation}\label{B-2.42}
		\big< D^* A^*_{o,q} v, g \big>_{L^{q'}(\Gamma),L^{q}(\Gamma)} = \big< (y_e \cdot \nabla)(Dg) + ((Dg) \cdot \nabla)y_e, v\big>_{L^{q}(\Omega),L^{q'}(\Omega)}
		\end{equation}
		\noindent Proof of (\ref{B-2.42}).\\
		
		\noindent \textit{Step (2a)}: Let initially $\ds h \in \calD(A_{o,q}) = \calD(A^{\rfrac{1}{2}}_q) = W^{1,q}_0(\Omega) \cap \lso$ by (\ref{B-2.21}). Recalling (\ref{B-2.21}) compute 
		\begin{align}
		\big< A_{o,q} h, v \big>_{\lso,\lo{q'}} &= \big< P_q [(y_e \cdot \nabla) h], v \big>_{\lso,\lo{q'}} + \big< P_q [(h \cdot \nabla) y_e], v \big>_{\lso,\lo{q'}}\\
		&= \big<  [(y_e \cdot \nabla) h], P_q^* v \big>_{L^{q},L^{q'}} + \big< [(h \cdot \nabla) y_e], P_q^* v \big>_{L^{q},L^{q'}}\\
		&= \big<  [(y_e \cdot \nabla) h], P_{q'} v \big>_{L^{q},L^{q'}} + \big< [(h \cdot \nabla) y_e], P_{q'} v \big>_{L^{q},L^{q'}}\\
		&= \big<  [(y_e \cdot \nabla) h], v \big>_{L^{q},L^{q'}} + \big< [(h \cdot \nabla) y_e], v \big>_{L^{q},L^{q'}} \label{B-2.46}
		\end{align}
		\noindent where we have recalled $\ds P_q^* = P_{q'}$ from the statement above \eqref{B-1.7} and $\ds P_{q'} v = v$, as $v \in \lo{q'}$.\\
		
		\noindent \textit{Step (2b)}: In the next lemma, we show that the terms in (\ref{B-2.46}) are well-defined in an appropriate range of $q$, at any rate for $q > d$, which is our goal, $d = 2, d = 3$ .
	
		\begin{lemma}\label{B-Lem-2.4}
			With reference to (\ref{B-2.46}) we have \\
				\noindent (i) \begin{equation}
				 \ (y_e \cdot \nabla)h \in L^q(\Omega) = W^{0,q}(\Omega) \text{ for }
				\begin{cases}
				d = 3, \ q > \rfrac{3}{2} \\
				d = 2, \ q > 1 \\
				\end{cases}
				\end{equation} 
				\noindent (ii)\begin{equation}
				(h \cdot \nabla)y_e \in W^{1,q}(\Omega) \text{ for }
				\begin{cases}
				d = 3, \ q > 3 \\
				d = 2, \ q > 2. \\
				\end{cases}
				\end{equation}				
		\end{lemma}	
	\begin{proof}
		\noindent \underline{First way:} We may use multiplier theory \cite[Theorem 3, p 252]{MS:1985}. We have by Theorem \ref{B-Thm-1.1} on $y_e$ and the assumption on $\ds h \in \calD(A_{o,q}) = W^{1,q}_0(\Omega) \cap \lso$:
		\begin{enumerate}[(i)]
			\item 
			\begin{equation}
			y_e \in W^{2,q}(\Omega),\ \abs{\nabla h} \in L^q(\Omega) = W^{0,q}(\Omega).
			\end{equation}
			\noindent Then \cite[Theorem 3, p 252 with $m=2 > \ell = 0$]{MS:1985} yields the multiplier space 
			\begin{equation}
				\calM \big(W^{2,q} \longrightarrow W^{0,q} \big) = W^{0,q}(\Omega).
			\end{equation}
			\noindent for $ mq = 2q > d$ or $q > \rfrac{3}{2}$ for $d = 3; q > 1$ for $d = 2;$ and part (i) of Lemma \ref{B-Lem-2.4} established.
			\item We start with \begin{equation}
			h \in W^{1,q}_0(\Omega), \ \abs{\nabla y_e} \in W^{1,q}(\Omega).
			\end{equation}
			\noindent Then \cite[Theorem 3, p 252; with $m = \ell = 1$]{MS:1985} yields the multiplier space 
			\begin{equation}
			\calM (W^{1,q} \longrightarrow W^{1,q}) = W^{1,q}(\Omega)
			\end{equation}
			\noindent for $mq = 1.q > d$ or $q > 3$ for $d = 3, q > 2$ for $d = 2$ and part (ii) of Lemma \ref{B-Lem-2.4} is established.
		\end{enumerate}
	\noindent \underline{Second way:} We use embedding theory \cite[p 79]{SK:1989}
	\begin{equation}\label{B-2.53}
	W^{m,q}(\Omega) \hookrightarrow C^k(\overline{\Omega}), \ m > \frac{d}{q}, \ k \text{ integer part of } \bigg[ m - \frac{d}{q} \bigg]
	\end{equation}
	\noindent Thus (i) \[ y_e \in W^{2,q}(\Omega) \hookrightarrow y_e \in C^0(\overline{\Omega}) \text{ for } 
		\begin{cases}
		m = 2, d = 3, q > \rfrac{3}{2}, k = 0 \\
		m = 2, d = 2, q > 1, k = 0 \\
		\end{cases}\]	
		\noindent and since $| \nabla h | \in L^q(\Omega)$, then 
		\begin{equation}
		(y_e \cdot \nabla) h \in L^q(\Omega), d = 3, q > \rfrac{3}{2}; \text{ or } d = 2, q > 1
		\end{equation}	
		\noindent and (i) of Lemma \ref{B-Lem-2.4} is reproved.\\
		\noindent (ii) Similarly, (\ref{B-2.53}) gives for $m = 1$
		\begin{equation}
		h \in W^{1,q}(\Omega) \hookrightarrow h \in C^0(\overline{\Omega}) \text{ for }
		\begin{cases}
			d = 3, q > 3, k = 0 \\
			d = 2, q > 2, k = 0 \\
		\end{cases}
	\end{equation} 
	\noindent and since $\ds \abs{\nabla y_e} \in W^{1,q}(\Omega)$, then 
	\begin{equation}
	(h \cdot \nabla) y_e \in W^{1,q}(\Omega), \ d = 3, q > 3; d = 2, q > 2
 	\end{equation}
 	\noindent and (ii) of Lemma \ref{B-Lem-2.4} is reproved.\\
	\noindent Lemma \ref{B-Lem-2.4} is proved.
	\end{proof}
	\noindent \textit{Step (2c)}: Using Lemma \ref{B-Lem-2.4} in (\ref{B-2.46}) we see that the two terms are well-defined with $v \in \lo{q'}$. We rewrite (\ref{B-2.46}) as
	\begin{equation}\label{B-2.57}
	\big< h, A^*_{o,q} v \big>_{L^q_{\sigma},L^{q'}_{\sigma}} = \big< (y_e \cdot \nabla)h + (h \cdot \nabla) y_e, v \big>_{L^{q},L^{q'}},
	\end{equation}
	\noindent which shows that it can be extended to all $\ds h \in W^{1,q}(\Omega) \cap \lso$: the condition $h|_{\Gamma} = 0$ is not used. With $\ds g \in W^{1-\rfrac{1}{q},q}(\Gamma), \ g \cdot \nu = 0$ on $\Gamma$, so that $Dg \in W^{1,q}(\Omega) \cap \lso$ by (\ref{B-2.9}), we may apply such extended version (\ref{B-2.57}) to $Dg$ and obtain
	\begin{align}
	\big< Dg, A^*_{o,q}v \big> &= \big< g, D^* A^*_{o,q}v \big>_{L^q(\Gamma),L^{q'}(\Gamma)} \nonumber \\
	&= \big< (y_e \cdot \nabla)(Dg) + ((Dg) \cdot \nabla) y_e, v \big>_{L^{q}(\Omega),L^{q'}(\Omega)}
	\end{align}
	\noindent and (\ref{B-2.42}) is established.\\
	
	\noindent \underline{Step 3:}
	\noindent In view of $\ds \calA_q = - (\nu_o A_q + A_{o,q})$ by (\ref{B-2.22}), we now combine (\ref{B-2.37}) of Step 1, with (\ref{B-2.42}) of Step 2. Let again $\ds g \in W^{1-\rfrac{1}{q},q}(\Gamma), \ g \cdot \nu = 0$ on $\Gamma$, so that $\ds Dg \in W^{1,q}(\Gamma) \cap \lso$ by (\ref{B-2.12}) with $\theta = 0$, and $v \in \calD(A_q^*) = \calD(\calA_q^*) = W^{2,q'}(\Omega) \cap W^{1,q'}_0(\Omega) \cap \lo{q'}$ via (\ref{B-2.31}). We shall establish the following final relation
	
	\begin{equation}
		\big< D^* \calA_q^* v, g\big> = \nu_0 \int_{\Gamma} \frac{\partial v}{\partial \nu}g \ d \Gamma = \nu_0 \bigg< \frac{\partial v}{\partial \nu},g\bigg>_{L^{q'}(\Gamma),L^{q}(\Gamma)}.
	\end{equation}
	\noindent In fact, we start from $\ds - \calA^*_q = \nu_0 A^*_q + A^*_{o,q}$ via part (iv) of Section \ref{B-Sec-2.4} and next recall (\ref{B-2.37}) and (\ref{B-2.42}) to obtain 
	\begin{align}
	-\big< D^* \calA_q^* v, g\big>_{L^{q'}(\Gamma),L^{q}(\Gamma)} & = \nu_o \big< D^* A_q^* v, g \big>_{L^{q'}(\Gamma),L^{q}(\Gamma)} + \big< D^* A_{o,q}^* v, g \big>_{L^{q'}(\Gamma),L^{q}(\Gamma)}
	\end{align}  
	\vspace{-0.6cm}
	\begin{align}
	&= - \nu_o \int_{\Omega} v \Delta (Dg) \ d \Omega  - \nu_o \int_{\Gamma} \frac{\partial v}{\partial \nu} g \ d \Gamma + \big< v, (y_e \cdot \nabla)(Dg) + ((Dg) \cdot \nabla)y_e \big>_{L^{q'}(\Omega),L^{q}(\Omega)} \label{B-2.61}\\
	&=\big< v, - \nu_o \Delta (Dg) + (y_e \cdot \nabla)(Dg) + ((Dg) \cdot \nabla)y_e \big>_{L^{q'}(\Omega),L^{q}(\Omega)} - \nu_o \bigg< \frac{\partial v}{\partial \nu},g\bigg>_{L^{q'}(\Gamma),L^{q}(\Gamma)} \label{B-2.62}\\
	&= \big< v, -\nu_o \Delta (Dg) + L_e(Dg)\big>_{L^{q'}(\Omega),L^{q}(\Omega)} - \nu_o \bigg< \frac{\partial v}{\partial \nu},g\bigg>_{L^{q'}(\Gamma),L^{q}(\Gamma)}, \label{B-2.63} 
	\end{align}
	\noindent recalling $\ds L_e(\psi = Dg) = (y_e \cdot \nabla) (Dg) + ((Dg) \cdot \nabla) y_e$ by (\ref{B-1.9}). We next invoke the definition $\ds \psi = Dg$ in Eq (\ref{B-2.1a}) of the Stationary Oseen Equation (\ref{B-2.1}). This way we rewrite (\ref{B-2.63}) as 
	\begin{align}
	\big< D^* \calA_q^* v, g \big>_{L^{q'}(\Gamma),L^{q}(\Gamma)} &= \big< v, \nabla \pi^* \big>_{L^{q'}(\Omega),L^{q}(\Omega)} + \nu_o \bigg< \frac{ \partial v}{\partial \nu}, g\bigg>_{L^{q'}(\Gamma),L^{q}(\Gamma)} \label{B-2.64}\\
	&= \nu_o \bigg< \frac{ \partial v}{\partial \nu}, g\bigg>_{L^{q'}(\Gamma),L^{q}(\Gamma)},\label{B-2.65}
	\end{align}	
	\noindent since 
	\begin{equation}
	\int_{\Omega} v \cdot \nabla \pi^* = \int_{\Gamma} \pi^* v \cdot \nu \ d \Gamma - \int_{\Omega} \pi^* \text{div } v \ d \Omega \equiv 0 \label{B-2.66}
	\end{equation}
	\noindent where $v|_{\Gamma} = 0$ as $v \in W^{1,q'}_0{(\Omega)}$ and div $v \equiv 0$ since $v \in \lso$, recall (\ref{B-1.4}). Thus, (\ref{B-2.65}) shows (\ref{B-2.36}) so far for $\ds g \in W^{1-\rfrac{1}{q},q}(\Gamma), \ g \cdot \nu = 0 $ on $\Gamma$.\\
	
	\noindent By density, we extend the validity of (\ref{B-2.36}) to $g \in L^q(\Gamma), \ g \cdot \nu  = 0$ on $\Gamma$. Theorem \ref{B-Thm-2.3} is proved.	
	\end{proof}

	\begin{prop}\cite[Lemma 3.3.1 p35]{BLT1:2006}\label{B-Prop-2.5}
		Let $\varphi \in C^1(\Omega)$ be a $d$-function satisfying the following properties:
		\begin{enumerate}[(i)]
			\item $\varphi|_{\Gamma} = 0$;
			\item div $\varphi = 0$ in $\overline{\Omega}$ (actually only on an interior collar of $\Gamma$)
		\end{enumerate}
	\noindent Then we have that 
	\begin{equation}
	\begin{cases}
	\text{the boundary vector }	\nabla \varphi \cdot \nu = \frac{\partial \varphi}{\partial \nu} \text{ is tangential to } \Gamma \\[2mm]
	\text{i.e} \ (\nabla \phi \cdot \nu) \cdot \nu \equiv 0 \text{ on } \Gamma.
	\end{cases}	
	\end{equation}
	\end{prop}
	\noindent For $\ds v \in \calD(\calA^*_q) = W^{2,q'}(\Omega) \cap W^{1,q'}_0{\Omega} \cap \lo{q'}$, we have $v|_{\Gamma} = 0$ and div $v = 0$ in $\Omega$. Thus extending Proposition \ref{B-Prop-2.5} to $v \in W^{2,q'}(\Omega)$, we have $\ds \frac{\partial v}{\partial \nu} \bigg|_{\Gamma} = $ tangential on $\Gamma$. Returning to Theorem \ref{B-Thm-2.3}, and recalling that $g$ is tangential. $g \cdot \nu = 0$ on $\Gamma$, we then obtain from (\ref{B-2.36}) the following
	\begin{clr}\label{B-Clr-2.6}
	With reference to Theorem \ref{B-Thm-2.3} we have
	\begin{equation}\label{B-2.68}
	\begin{Bmatrix}
	\text{tangential}\\
	\text{component of }  D^* \calA^* v 
	\end{Bmatrix}
	= (D^* \calA_q^* v) \tau = \nu_o \frac{\partial v}{\partial \nu}, \ v \in \calD(\calA_q^*) = W^{2,q'}(\Omega) \cap W^{1,q'}_0{\Omega} \cap \lo{q'}
	\end{equation}
	\noindent $q > 3$ for $d = 3$; $q > 2$ for $d = 2$. \qedsymbol
	\end{clr}
	
	\noindent We return to the Dirichlet map $D$ introduced in \eqref{B-2.13}, and extract an important result [to be used e.g. in \eqref{B-2.13} to claim that the feedback generator $\BA_F$ generates a s.c. analytic semigroup in $\lso$]. All this is a perfect counterpart of results in Hilbert spaces $(q = 2)$, which have been used in \cite{LT2:2015} etc. We first quote a known result.

	\begin{prop}
		With reference to the Stokes operator $A_q$ in (\ref{B-2.20}) on $\lso$, we have, for $1 < q < \infty$.\\
		\begin{subequations}\label{B-2.69}
				(i)\begin{equation}
				\ W^{s,q}(\Omega) \equiv W^{s,q}_0(\Omega), \quad 0 \leq s \leq \frac{1}{q} \label{B-2.69a}
				\end{equation}
				(ii)
				\begin{equation}
				W^{2s,q}_0(\Omega) \cap \lso \subset \calD \big( A^{\gamma}_q \big), \ 0 \leq \gamma < s, \ 0 \leq s \leq 1, \ q \geq 2, 2s \neq \frac{1}{q}, \ 2s \neq \frac{1}{q} + 1 \label{B-2.69b}
				\end{equation}
				(iii) In particular, for $\varepsilon > 0$ arbitrary, $q \geq 2$, via (\hyperref[B-3.40a]{i}): 
				\begin{equation}
				W^{\rfrac{1}{q},q}(\Omega) \cap \lso = W^{\rfrac{1}{q},q}_0(\Omega) \cap \lso \subset \calD \Big( A_q^{\rfrac{1}{2q} - \varepsilon} \Big) \qed \label{B-2.69c}
				\end{equation}		
		\end{subequations}
	\end{prop}
	\noindent Indeed, for (\hyperref[B-3.40a]{i}) we invoke \cite[(0.2.17) p XX1]{W:1985}. For (\hyperref[B-3.40b]{ii}), we quote \cite[Theorem III.2.3 p 91]{W:1985} where, in this reference, the space $\ds H_q(\Omega)$ is our space $\lso$, and the space $\ds \mathring{H}^{2s,q}(\Omega)$ can be replaced (see proof) by the space $\ds W^{2s,q}_0(\Omega)$ in our notation. For (\hyperref[B-3.40c]{iii}), we apply (\hyperref[B-3.40a]{i}) and (\hyperref[B-3.40b]{ii}) with $\ds 2s = \rfrac{1}{q}$, hence $\ds \gamma < \rfrac{1}{2q}$. \qedsymbol

	\begin{clr}\label{B-Clr-2.8}
	\noindent For the Dirichlet map $\ds D: g \longrightarrow \psi $ defined in reference to problem (\ref{B-2.1}) and the paragraph below it, we have complementing (\ref{B-2.13}) = (\ref{B-2.28})
	\begin{subequations}\label{B-3.41}
	\begin{multline}\label{B-2.70a}
	g \in U_q = \big\{ g \in L^q(\Gamma); \ g \cdot \nu = 0 \text{ on } \Gamma \big\} \longrightarrow Dg \in W^{\rfrac{1}{q},q}(\Omega) \cap \lso \subset \calD \Big( A_q^{\rfrac{1}{2q} - \varepsilon} \Big)
	\end{multline}
	\begin{equation}\label{B-2.70b}
	\text{or } A_q^{\rfrac{1}{2q} - \varepsilon} D \in \calL(U_q, \lso)
	\end{equation}
	\end{subequations}		
	\end{clr}
	\noindent We shall invoke this property in Theorem \ref{B-Thm-5.1}, \eqref{B-5.7} and Proposition \ref{B-Prop-6.2}, \eqref{B-6.14b}.

	\begin{rmk}\label{B-rmk-2.2}
	As noted in Remark \ref{B-rmk-1.2}, The literature reports physical situations where the volumetric force $f$ in (\ref{B-1.1a}) or (\ref{B-1.2a}), is actually replaced by $\nabla g(x)$; that is, $f$ is a conservative vector field. In this case, a solution to the stationary problem (\ref{B-1.2}) is: $y_e \equiv 0, \pi_e = g$. Taking $y_e \equiv 0 $ (hence $L_e(\phi) = 0$) and returning to Eq (\ref{B-1.1a}) with $f(x)$ replaced now by $\nabla g(x)$ and applying to the resulting equation the projection operator $P_q$, one obtains in this case the projected equation 
	\begin{equation}\label{B-2.71}
		y_t - \nu_o P_q \Delta y + P_q \big[ (y \cdot \nabla) y \big] = P_q (mu) \quad \text{in } Q.
	\end{equation}
	\noindent This, along with the solenoidal and boundary conditions (\ref{B-1.1b}), (\ref{B-1.1c}), yields the corresponding abstract form 
	\begin{equation}\label{B-2.72}
		y_t + \nu_o A_q(y-Dv) + \calN_q y = P_q (mu) \quad \text{in } \lso.
	\end{equation}
	\noindent Then $y$-problem (\ref{B-2.72}) is the same as the $z$-problem (\ref{B-2.24}) or (\ref{B-2.25}), except without the Oseen term $A_{o,q}$ see (\ref{B-2.22}). The linearized version of problem (\ref{B-2.72}) is then 
	\begin{equation}\label{B-2.73}
		\eta_t + \nu_o A_q (\eta - Dv) = P_q (mu) \quad \text{in } \lso,
	\end{equation}
	\noindent which is the same as the $w$-problem (\ref{B-2.26}), except without the Oseen term $A_{o,q}$ in \eqref{B-1.10}. The s.c. analytic semigroup $\ds e^{-\nu_o A_q t}$ driving the linear equation (\ref{B-2.73}) is uniformly stable in $\lso$, see \eqref{B-A.4} of Appendix \ref{B-app-A}, as well as in $\ds \Bto$, \eqref{B-1.15b} with decay rate $-\delta$, \eqref{B-A.9} of Appendix \ref{B-app-A}. Then, in the case of the present Remark and as anticipated in the Orientation, the present paper may be used to \underline{enhance at will the uniform stability} of the corresponding problem by use \uline{only} of the tangential boundary feedback finite dimensional control $v$, as acting on the entire boundary $\Gamma$. It is of the \underline{form} given by (\ref{B-5.6a}), with boundary vectors $f_k$ now acting on the entire boundary $\Gamma$. Thus one can take the interior tangential-like control $u \equiv 0$, or vectors $q_k \equiv 0$ in (\ref{B-5.3}). Given the original decay rate $-\delta$ of the Stokes semigroup in $\ds \Bto$ in (\ref{B-A.9}) of Appendix \ref{B-app-A}, and preassigned a desirable decay rate $-k^2$ (arbitrary), the procedure of the present paper can be adopted to construct such a \uline{tangential} boundary \uline{finite} dimensional feedback control $v$ \uline{on all} of $\Gamma$ that yields the decay rate $-k^2$. Thus there is no need to perform the translation $y \longrightarrow z$ of Section \ref{B-Sec-1.5}, when $f$ in (\ref{B-1.2a}) is replaced by $\nabla g(x)$; i.e. $y_e = 0$ in this case. The corresponding required ``unique continuation property" holds true for the Stokes problem ($y_e = 0$), see Problem \#3 of Appendix \ref{B-app-C}, \cite{RT:2008}, \cite{RT1:2009}.
	\end{rmk}		
	\section{Introducing the Problem of Feedback Stabilization of the Linearized $w$-Problem (\ref{B-2.26}) on the Complexified $\lso$-space.}\label{B-Sec-3}
	
	\noindent \textbf{Preliminaries:} In this subsection we take $q$ fixed, $1 < q < \infty$ throughout. Accordingly, to streamline the notation in the preceding setting of Section \ref{B-Sec-2}, we shall drop the dependence on $q$ of all relevant quantities and thus write $P, A, A_o, \calA$ instead of $P_q, A_q, A_{o,q}, \calA_q$. We return to the linearized system (\ref{B-2.26}). Moreover, as in \cite{BT:2004}, \cite{BLT1:2006}, we shall henceforth let $\lso$ denote the complexified space $\lso + i\lso$, whereby then we consider the extension of the linearized problem (\ref{B-2.26}) to such complexified space. Thus, henceforth, $w$ will mean $w + i \wti{w}$, $u$ will mean $u + i \wti{u}$, $v$ will mean $v + i \wti{v}$, $w_0$ will mean $w_0 + i \wti{w}_0$. Thus, henceforth, the abstract model (\ref{B-2.26}) is rewritten with the same symbols as 	
	\begin{equation}\label{B-3.1}
		w_t - \calA w = - \calA Dv + P((mu)\tau) \in [\calD(\calA^*)]', \quad w(0) \in \lso, \quad v \cdot \nu = 0 \text{ on } \Sigma
	\end{equation}	
	\noindent to mean however the complexified version of (\ref{B-2.26}). As noted in Theorem \hyperref[B-A.1]{A.1(iii)}, the Oseen operator $\calA$ has compact resolvent on $\lso$. It follows that $\calA$ has a discrete point spectrum $\sigma(\calA) = \sigma_p(\calA)$ consisting of isolated eigenvalues $\{ \lambda_j\}_{j = 1}^{\infty}$, which are repeated according to their (finite) geometric multiplicity $\ell_j$. However, since $\calA$ generates a $C_0$ analytic semigroup on $\lso$, Theorem \hyperref[B-Thm-A.1]{A.1(ii)}, its eigenvalues $\{ \lambda_j\}_{j = 1}^{\infty}$ lie in a triangular sector of a well-known type. We recall the underlying assumption \eqref{B-1.12} of instability of the equilibrium solution $y_e$ under consideration, which is the prerequisite for investigating the present uniform stabilization problem. This means that the corresponding Oseen operator $\calA = \calA_q$ has a finite number, say $N$, of eigenvalues $\lambda_1, \lambda_2 ,\lambda_3 , \dots, \lambda_N$ on the complex half plane $\{ \lambda \in \mathbb{C} : Re~\lambda \geq 0 \}$ which we then order according to their real parts, so that \eqref{B-1.12} holds, repeated here for the convenience,	
	\begin{equation}\label{B-3.2}
	\ldots \leq Re~\lambda_{N+1} < 0 \leq Re~\lambda_N \leq \ldots \leq Re~\lambda_1,
	\end{equation}	
	\noindent each $\lambda_i, \ i=1,\dots,N$ being an unstable eigenvalue repeated according to its geometric multiplicity $\ell_i$. Let $M$ denote the number of distinct unstable eigenvalues $\lambda_j$ of $\calA$. Denote by $P_N$ and $P_N^*$ the projections given explicitly by \cite[p 178]{TK:1966}, \cite{BT:2004}, \cite{BLT1:2006}
	\begin{subequations}\label{B-3.3}
		\begin{align}
		P_N &= -\frac{1}{2 \pi i}\int_{\Gamma}\left( \lambda I - \calA \right)^{-1}d \lambda : \lso \text{ onto } W^u_N \subset \lso \label{B-3.3a} \\
		P_N^* &= -\frac{1}{2 \pi i}\int_{\bar{\Gamma}}\left( \lambda I - \calA^* \right)^{-1}d \lambda : (\lso)^* = \lo{q'} \text{ onto } (W^u_N)^* \subset \lo{q'}, \label{B-3.3b} 
		\end{align}
	\end{subequations}
	
	\noindent $\ds \rfrac{1}{q} + \rfrac{1}{q'} = 1$, recall (\ref{B-1.7}), where $\Gamma$ (respectively, its conjugate counterpart $\bar{\Gamma}$) is a smooth closed curve that separates the unstable spectrum from the stable spectrum of $\calA$ (respectively, $\calA^*$).\\
	
	\noindent As in \cite[Sect 3.4, p 37]{BLT1:2006}, following \cite{RT:1975},we decompose the space $\lso$ into the sum of two complementary subspaces (not necessarily orthogonal):
	
	\begin{equation}\label{B-3.4}
	\lso = W^u_N \oplus W^s_N; \quad W^u_N \equiv P_N \lso ;\quad W^s_N \equiv (I - P_N) \lso; \quad \text{ dim } W^u_N = N, 
	\end{equation}
	
	\noindent where each of the spaces $W^u_N$ and $W^s_N$ is invariant under $\calA$, and let	
	\begin{equation}\label{B-3.5}
	\calA^u_N = P_N \calA = \calA |_{_{W^u_N}} ; \quad \calA^s_N = (I - P_N) \calA = \calA |_{W^s_N}
	\end{equation}
	
	\noindent be the restrictions of $\calA$ to $W^u_N$ and $W^s_N$ respectively. The original point spectrum (eigenvalues) $\{ \lambda_j \}_{j=1}^{\infty} $ of $\calA$ is then split into two sets	
	\begin{equation}\label{B-3.6}
	\sigma (\calA_N^u) = \{ \lambda_j \}_{j=1}^{N}; \quad  \sigma (\calA_N^s) = \{ \lambda_j \}_{j=N+1}^{\infty},
	\end{equation}	
	\noindent and $W_N^u$ is the generalized eigenspace of ($\calA$, hence of) $\calA^u_N$, corresponding to its unstable eigenvalues. The system (\ref{B-3.1}) on $\lso$ with $v \cdot \nu = 0$ on $\Sigma$ can accordingly be decomposed as	
	\begin{equation}\label{B-3.7}
	w = w_N + \zeta_N, \quad w_N = P_N w, \quad \zeta_N = (I-P_N)w. 
	\end{equation}	
	\noindent After applying $P_N$ and $(I-P_N)$ (which commute with $\calA$) to (\ref{B-3.1}), we obtain via (\ref{B-3.5})	
	\begin{subequations}\label{B-3.8} 
		\begin{align}
		\text{on } W_N^u: w'_N - \calA^u_N w_N &= -P_N(\calA Dv) + P_N P ((mu)\tau) \label{B-3.8a}\\
		&= -\calA^u_N P_N Dv + P_N P ((mu)\tau); \  w_N(0) = P_N w_0 \label{B-3.8b}
		\end{align}
	\end{subequations}
	\begin{subequations}\label{B-3.9}
		\begin{align}
		\text{on } W_N^s: \zeta'_N - \calA^s_N \zeta_N &= - (I-P_N)(\calA Dv) + (I-P_N) P ((mu) \tau ) \label{B-3.9a}\\
		&= - \calA^s_N(I-P_N) Dv + (I-P_N) P ((mu) \tau ); \  \zeta_N(0) = (I - P_N) w_0 \label{B-3.9b}
		\end{align}
	\end{subequations}
	\noindent respectively.  [In (\ref{B-3.8a}), (\ref{B-3.8b}), actually $P_N$ is the extension from  original $\lso$ to $[\calD(\calA^*)]'$	\cite[Appendix A.1]{BLT1:2006}]. For each distinct $\lambda_i, i=1, \ldots, M$, let $P_{N,i}, P_{N,i}^*$ be the projection corresponding to $\lambda_i$ and $\bar{\lambda}_i$, respectively, given by a similar integral of   $(\lambda I-\calA)^{-1}$, or $(\lambda I-\calA^*)^{-1}$, respectively, as in (\hyperref[B-3.3]{3.3a-b}), this time over a curve that encircles only $\lambda_i$, or $\bar{\lambda}_i$, respectively, and no other eigenvalue. Let $(W_N^u)_i=P_{N,i} \lso$, and $(A_N^u)_i=\calA^u\big|_{(W_N^u)_i}$.\\
	
	\noindent We have that, for $1 < p,q < \infty$:	
	\begin{equation}\label{B-3.10}
	W^u_N =
	\begin{Bmatrix}
	\text{space of generalized}\\
	\text{eigenfunctions of } \calA_q (= \calA^u_N)\\
	\text{corresponding to its distinct}\\
	\text{unstable eigenvalues}
	\end{Bmatrix}
	\subset
	\begin{Bmatrix}
	\begin{aligned}
	&\lqaq\\
	&\big[ \calD(A_q), \lso \big]_{1-\alpha} = \calD(A^{\alpha}_q), \ 0 < \alpha < 1
	\end{aligned}
	\end{Bmatrix}
	\subset \lso. 
	\end{equation}
	
	\section{Uniform stabilization with arbitrary decay rate of the finite dimensional $w_N$-dynamics (\ref{B-3.8}) by suitable finite-dimensional tangent-like pair $\{ v_N, u_N \} $ on $\{ \wti{\Gamma}, \omega \}$. Constructive proof with $q \geq 2$.}\label{B-Sec-4}
	
	\noindent All the main results of this paper, Theorems \ref{B-Thm-4.1} through \ref{B-Thm-9.2}, are stated (at first) in the complex state space setting $\ds \lso + i \lso$. Thus, the finitely many stabilizing feedback vectors $p_k \in (W^u_{N})^* \subset \lo{q'}, \ u_k \in W^u_{N} \subset \lso$ constructed in the subsequent proofs are related to the complex finite dimensional unstable subspace $W^u_N$. The question then arises as to transfer back these results into the original real setting. This issue was resolved in \cite{BT:2004}. Here, the translation, taken from \cite{BT:2004}, from the results in the complex setting (Theorems \ref{B-Thm-4.1} through \ref{B-Thm-9.2}) into corresponding results in the original real setting is given in Section \ref{B-Sec-10}. The following is the key desired control theoretic result of the dynamic $w_N$ in (\ref{B-3.8}) over the finite dimensional space $\ds W^u_N \subset \lso$. We shall henceforth impose the condition $q \geq 2$, due to requirement (\ref{B-B.5}) in Appendix \ref{B-app-B}. 
	
	\begin{thm}\label{B-Thm-4.1}
		Let $\lambda_1,\ldots,\lambda_M$ be the unstable distinct eigenvalues of the Oseen operator $\calA \ (= \calA_q)$ as in \eqref{B-1.12} = (\ref{B-3.2}), with geometric multiplicity $\ell_i, \ i = 1, \dots, M$, and set $K = \sup \{ \ell_i; i = 1, \dots, M \}$. Let $\wti{\Gamma} $ be an open connected subset of the boundary $\Gamma$ of positive surface measure and $\omega$ be a localized collar supported by $\wti{\Gamma}$ (Fig.~2). Let $q \geq 2$. Given $\gamma_1 > 0$ arbitrarily large, we can construct two $K$-dimensional controllers: a boundary tangential control $v = v_N$ acting with support on $\wti{\Gamma}$, of the form given by 				
		\begin{equation}\label{B-4.1}
		v= v_N = \sum^K_{k=1} \nu_k(t) f_k, \ f_k \in \calF \subset W^{2-\frac{1}{q},q}(\Gamma), \ q \geq 2,  \mbox{ so that } f_k \cdot \nu = 0, \ \text{hence } v_N \cdot \nu = 0 \mbox{ on } \Gamma
		\end{equation}
		
		\noindent $\calF$ defined in (\ref{B-1.25}), $q \geq 2$, $f_k$ supported on $\wti{\Gamma}$, and an interior tangential-like control $u = u_N$ acting on $\omega$, of the form given by		
		\begin{equation}\label{B-4.2}
		u = u_N = \sum_{k=1}^K \mu_k(t)u_k, \quad u_k \in W^u_N \subset \lso, \quad \mu_k(t) = \text{scalar,}
		\end{equation}
		\noindent thus with interior vectors $[u_1, \dots, u_K]$ in the smooth subspace $W^u_N$ of $\lso, \ 2 \leq q < \infty$, supported on $\omega$, such that, once inserted in the finite dimensional projected $w_N$-system in (\ref{B-3.8}), yields the system
		\begin{equation}\label{B-4.3}
		w'_N - \calA^u_N w_N = - \calA^u_N P_N D \Bigg ( \sum_{k=1}^{K} \nu_k(t) f_k \Bigg) + P_N P \Bigg ( m \Bigg( \sum_{k=1}^{K} \mu_k(t) u_k \Bigg) \tau \Bigg),
		\end{equation}
		\noindent whose solution then satisfies the estimate
		\begin{multline}\label{B-4.4}
		\| w_N(t) \|_{\lso} + \| v_N(t) \|_{L^q(\wti{\Gamma})} + \| v'_N(t) \|_{L^q(\wti{\Gamma})} + \\ \|u_N(t)\|_{L^q_{\sigma}(\omega)} + \|u'_N(t)\|_{L^q_{\sigma}(\omega)} \leq C_{\gamma_{1}} e^{-\gamma_1 t} \|P_Nw_0\|_{\lso}, \quad t \geq 0.
		\end{multline}
		
		\noindent In (\ref{B-4.4}) we may replace the $\lso$-norm, $2 \leq q < \infty$, alternatively either with the $\ds \big(\lso, \calD(A_q) \big)_{1-\frac{1}{p},p}$ norm, $2 \leq q < \infty$; or else with the $\ds \big[ \calD(A_q), \lso \big]_{1-\alpha} = \calD(A^{\alpha}_q)$-norm in \eqref{B-A.5}, $ 0 \leq \alpha \leq 1,\ 2 \leq q < \infty$. In particular, we also have 
		\begin{multline}\label{B-4.5}
		\norm{w_N(t)}_{\Bto} + \norm{v_N(t)}_{L^q(\wti{\Gamma})} + \norm{v'_N(t)}_{L^q(\wti{\Gamma})} + \\ \norm{u_N(t)}_{\Bt(\omega)} + \norm{u'_N(t)}_{\Bt(\omega)} \leq C_{\gamma_{1}} e^{-\gamma_1 t} \norm{P_N w_0}_{\Bto}, \quad t \geq 0,
		\end{multline}
		
		\noindent in the $\Bto$-norm, $2 \leq q < \infty, \ p < \rfrac{2q}{2q - 1}$.[Estimate (\ref{B-4.4}) (in the weaker form \eqref{B-5.20}, i.e. without the derivative terms) will be invoked in the stabilization proof of Section \ref{B-Sec-5.3}.\\
		
		\noindent Moreover, such controllers $v=v_N$ and $u=u_N$ may be chosen in feedback form: that is, with references to the explicit expressions (\ref{B-4.1}) for $v$ and (\ref{B-4.2}) for $u$, of the form $\nu_k(t) = \big< w_N(t),p_k \big>_{_{W^u_N}}$ and $\mu_k(t) = \big<w_N(t),q_k \big>_{_{W^u_N}}$ for suitable vectors $p_k \in (W^u_N)^* \subset \lo{q'}$, $q_k \in (W^u_N)^* \subset \lo{q'}$ depending on $\gamma_1$, where $\big< \ , \ \big>$ denotes the duality pairing $\ds W^u_N \times (W^u_N)^*$.\\
		
		\noindent In conclusion, $w_N$ in (\ref{B-4.5}) is the solution of the equation (\ref{B-4.3}) on $W^u_N$ rewritten explicitly as
		\begin{equation}\label{B-4.6}
		w'_N-\calA^u_N w_N = - \calA^u_N P_N D \Bigg(\sum^K_{k=1}\big<w_N(t),p_k\big>_{_{W^u_N}} f_k \Bigg)
		+ P_NP \Bigg(m \Bigg(\sum^K_{i=1} \big< w_N(t),q_k \big>_{_{W^u_N}} u_k \Bigg) \tau \Bigg),
		\end{equation}
		\noindent $f_k$ supported on $\wti{\Gamma}$, $u_k$ supported on $\omega$, rewritten in turn as
		\begin{equation}\label{B-4.7}
		w'_N = \overline{A}^u w_N, \ w_N(t) = e^{\overline{A}^{u}t} P_N w_0, \ w_N(0) = P_N w_0 \quad \text{on } W^u_N.
		\end{equation}
	\end{thm}
	\begin{proof}
		A (lengthy, technical) proof is given in \cite{LT2:2015} for $q = 2$. As the present Theorem \ref{B-Thm-4.1} is the preliminary pillar upon which the present paper rests, we need to provide more insight. In the present general case $W^u_N$ is the space of generalized eigenfunctions of $\calA_q \ ( = \calA_N^u$) corresponding to its unstable eigenvalues, see \eqref{B-3.10}. For $i = 1,\ldots,M$, we now denote by $\{\varphi_{ij}\}^{\ell_i}_{j=1}$, $\{\varphi^*_{ij}\}^{\ell_i}_{j=1}$ the (normalized) linearly independent (on $\lso$) eigenfunctions corresponding to the (assumed unstable) {\em distinct} eigenvalues $\lambda_1,\ldots,\lambda_M$ of $\calA \ (= \calA_q)$ and $\overline{\lambda}_1,\ldots\overline{\lambda}_M$ of $\calA^* \ ( = \calA_q^*)$, respectively:
		\begin{subequations}{\label{B-4.8}}
			\begin{equation}\label{B-4.8a}
			\begin{aligned}
			\calA_q \varphi_{ij} &= \lambda_i \varphi_{ij} \in \calD(\calA_q) = W^{2,q}(\Omega) \cap W^{1,q}_0(\Omega) \cap \lso \subset \lso, \\
			\calA^* \varphi^*_{ij} &= \overline{\lambda}_i \varphi^*_{ij} \in \calD(\calA^*_q) = W^{2,q'}(\Omega) \cap W^{1,q'}_0(\Omega) \cap \lo{q'} \subset \lo{q'}.
			\end{aligned}		
			\end{equation}
			\noindent The eigenvectors $\varphi_{ij}$ and $\varphi^*_{ij}$ are in $\ds \calD(\calA_q^n)$ and $\ds \calD((\calA^*_q)^n)$, for any $n$, hence they are arbitrarily smooth in $\lso$ and $\lo{q'}$, respectively. For our purposes, it will suffice to take $q \geq 2$, hence $\ds q' \leq 2, \ \rfrac{1}{q} + \rfrac{1}{q'} = 1$, and view eigenvectors henceforth as follows, see Appendix \ref{B-app-B}, Eq (\ref{B-B.5}) 
			\begin{equation}\label{B-4.8b}
			\varphi_{ij} \in W^{3,q}(\Omega) \cap \lso; \quad \varphi^*_{ij} \in W^{3,q}(\Omega) \cap \lso. 
			\end{equation}		
		\end{subequations}
		\noindent Hence, for $i = 1, \dots, M$, we may view $\varphi_{ij}$ and $\varphi^*_{ij}$ as elements of the generalized eigenspace $\ds W^u_N$ in (\ref{B-3.10}) and its dual $\ds (W^u_N)^*$, corresponding to the unstable eigenvalues as in (\ref{B-3.2}). For $h_1 \in W^u_N, h_2 \in (W^u_N)^*$, we set $\ds \big< h_1, h_2 \big>_{_{W^u_N}} = \int_{\Omega} h_1 h_2 \ d \Omega$, as a duality pairing.\\ 
		
		\noindent \uline{Step 1:} The challenging key step of the proof in \cite{LT2:2015} consists in showing that the $N$-dimensional  $w_N$-problem (\ref{B-3.8}) is controllable in $W^u_N$ by using a finite dimensional pair $\{v,u\}$ of localized tangential controllers, in particular in feedback form. To this end, we introduce the $\ell_i \times K$ matrix $W_i$, $i = 1,\ldots,M$:		
		\begin{equation}\label{B-4.9}
		W_i = \left[
		\begin{array}{lcl}
		(f_1,\partial_\nu \varphi^*_{i1}|_\Gamma)_{\wti{\Gamma}},
		& \cdots, &
		(f_K,\partial_\nu \varphi^*_{i1}|_\Gamma)_{\wti{\Gamma}}
		\\[3mm]
		(f_1,\partial_\nu \varphi^*_{i2}|_\Gamma)_{\wti{\Gamma}},
		& \cdots, &
		(f_K,\partial_\nu \varphi^*_{i2}|_\Gamma)_{\wti{\Gamma}}
		\\
		\hspace*{40pt} \vdots &  & \hspace*{40pt} \vdots
		\\
		(f_1,\partial_\nu \varphi^*_{i\ell_{i}}|_\Gamma)_{\wti{\Gamma}},
		& \cdots, &
		(f_K,\partial_\nu \varphi^*_{i\ell_{i}}|_\Gamma)_{\wti{\Gamma}}
		\end{array}
		\right]: \ell_i \times K;  \ \partial_\nu = \frac{\partial}{\partial\nu}, \ (\ ,\ )_{\wti{\Gamma}} = (\ ,\ )_{L^q(\wti{\Gamma}),L^{q'}(\wti{\Gamma})}.
		\end{equation}
		\noindent as well as the $\ell_i \times K$ matrix $U_i$, $i = 1,\ldots,M$;  $K \geq \ell_i$, $i = 1,\ldots,M$:
		\begin{equation}\label{B-4.10}
		U_i = \left[
		\begin{array}{lcl}
		\big<u_1, \varphi^*_{i1}  \cdot \tau\big>_\omega  ,
		& \cdots, &
		\big<u_K,\varphi^*_{i1} \cdot \tau\big>_\omega
		\\[3mm]
		\big<u_1, \varphi^*_{i2} \cdot \tau\big>_\omega
		& \cdots, &
		\big<u_K, \varphi^*_{i2} \cdot \tau\big>_\omega
		\\
		\hspace*{30pt} \vdots &  & \hspace*{30pt} \vdots
		\\
		\big<u_1, \varphi^*_{i\ell_{i}} \cdot \tau \big>_\omega
		& \cdots, &
		\big<u_K,\varphi^*_{i\ell_{i}} \cdot \tau \big>_\omega
		\end{array}
		\right]: \ \ell_i \times K; \
		\quad \big<\cd,\cd\big>_\omega = \big<\cd,\cd\big>_{\ls(\omega)}.
		\end{equation}
		\noindent The following is the main result of the present section - verification of the corresponding Kalman controllability criterion.
		
		\begin{thm}\label{B-Thm-4.2}
			With reference to (\ref{B-4.9}), (\ref{B-4.10}), it is possible to select boundary vectors $f_1,\ldots,f_K$ in $\calF \subset  W^{2-\rfrac{1}{q},q}(\Gamma)$, $\calF$ defined in (\ref{B-1.25}) with support on $\wti{\Gamma}$,  and interior vectors $u_1,\ldots,u_K \in L^q(\omega)$,$K = sup\{\ell_i, i =1 \ldots M\}$, such that for the matrix $[ - \nu_0 W_i \
			\begin{picture}(0,0)\multiput(0,9)(0,-5){3}{\line(0,-1){3}}\end{picture}\ U_i]$ of size $\ell_i \times 2 \ell_i$, we have
			\begin{subequations}\label{B-4.11}			
				\begin{equation}			
				\mbox{rank }[ - \nu_0 W_i \ \begin{picture}(0,0)\multiput(0,9)(0,-5){3}{\line(0,-1){3}}\end{picture}\ U_i] = \mbox{full} = \ell_i, \quad i = 1,\ldots,M. \label{B-4.11a}
				\end{equation}
				In fact, explicitly and more precisely, for each $i = 1,\ldots,M$, we have via (\ref{B-4.9}), (\ref{B-4.10}):
				\begin{equation}
				\mbox{rank}
				\left[
				\begin{array}{ccc}
				\big(f_1,\partial_\nu\varphi^*_{i1}\big)_{\wti{\Gamma}} \cdots \big(f_{\ell_{i}},\partial_\nu\varphi^*_{i1}\big)_{\wti{\Gamma}}
				&
				\begin{picture}(0,0)\multiput(0,12.5)(0,-7.5){9}{\line(0,-1){6}}\end{picture}
				&
				\big<u_1,\varphi^*_{i1} \cdot \tau\big>_\omega \cdots \big<u_{\ell_{i}},\varphi^*_{i1}\cdot \tau\big>_\omega
				\\
				\big(f_1,\partial_\nu\varphi^*_{i2}\big)_{\wti{\Gamma}} \cdots \big(f_{\ell_{i}},\partial_\nu\varphi^*_{i2})_{\wti{\Gamma}}
				&
				&
				\big<u_1,\varphi^*_{i2} \cdot \tau\big>_\omega \cdots \big<u_{\ell_{i}},\varphi^*_{i2}\cdot \tau\big>_\omega
				\\
				\vdots \hspace*{90pt} && \vdots \hspace*{90pt}
				\\
				\big(f_1,\partial_\nu\varphi^*_{i\ell_{i}}\big)_{\wti{\Gamma}} \cdots \big(f_{\ell_{i}},\partial_\nu\varphi^*_{i\ell_{i}}\big)_{\wti{\Gamma}}
				&
				&
				\big<u_1,\varphi^*_{i\ell_{i}} \cdot \tau\big>_\omega \cdots \big<u_{\ell_{i}},\varphi^*_{i\ell_{i}}\cdot \tau\big>_\omega
				\end{array}\right] = \ell_i, \label{B-4.11b}
				\end{equation}
			\end{subequations}
			where the matrix in (\hyperref[B-4.11b]{4.11b}) is $\ell_i \times 2 \ell_i$ and the boundary \underline{terms are only evaluated on $\wti{\Gamma}$}.
		\end{thm}		
		\noindent \uline{Step 2:} Verification of the above algebraic rank conditions of Kalman and Hautus style rests critically on the following unique  continuation property for the adjoint of the Oseen eigenvalue problem \cite{LT2:2015}.
		
		\begin{lemma}\label{B-lem-4.3}
			Let $\bar{\lambda}$ be an unstable eigenvalue of the adjoint Oseen operator, as in \ref{B-4.8a}. Let $\{ \varphi^*, p^* \} \in W^{2,q}(\Omega) \cap W^{1,q}(\Omega)$ where
			\begin{subequations}\label{B-4.12}
				\begin{align}
				-\nu_o \Delta \varphi^*-(L_e)^*(\varphi^*) + \nabla p^* & =  \overline{\lambda} \varphi^* & \text{in } \Omega; \label{B-4.12a}\\
				\begin{picture}(0,0)
				\put(-130,-5){ $\left\{\rule{0pt}{33pt}\right.$}\end{picture}
				\div \ \varphi^* & \equiv  0  & \text{in } \Omega; \label{B-4.12b}\\
				\varphi^*|_{\wti{\Gamma}} = 0; \ \frac{\partial \varphi^*}{\partial\nu} \bigg|_{\wti{\Gamma}} = 0; \ \varphi^* \cdot \tau & =  0 \text{ in } \omega; \label{B-4.12c}
				\end{align}
			\end{subequations}
		\begin{equation}\label{B-4.13}
		(L_e)^*(\varphi^*) = (y_e \cdot \nabla) \varphi^* + (\varphi^* \cdot \nabla)^* y_e,
		\end{equation}
		Then
		\begin{equation}\label{B-4.14}
			\varphi^* \equiv 0, \text{and } p^* \equiv \text{const} \quad \text{in } \Omega \ \qedsymbol
		\end{equation}
		\end{lemma}
	\noindent This result is equivalent to the UCP in the Problem \#2, Appendix \ref{B-app-C}: (\hyperref[B-C.6a]{C.6a-b-c}) $\implies$ \eqref{B-C.7}, of which it is the adjoint version. It is proved in \cite[Lemma 6.2]{LT2:2015}. If one omits the over-determination $\varphi^* \cdot \tau \equiv 0$ in $\omega$ in \eqref{B-4.12c}, then the corresponding UCP is \uline{false}. This is documented in Appendix \ref{B-app-C}, Problem \#1, in view of the counterexample in \cite{FL:1996}. Thus, it is Lemma \ref{B-lem-4.3} that justifies the necessity of using localized, tangential-like interior control $u$ on $\omega$, in addition to the localized boundary, even tangential, boundary control $v$ on $\wti{\Gamma}$. Relying only on $v$ (i.e. $u \equiv 0$) would not establish the required controllability of the $N$-dimensional $w_N$-problem \eqref{B-3.8}.\\
	
	\noindent \uline{Step 2:} Having established the controllability condition for the $N$-dimensional $w_N$-problem \eqref{B-3.8}, then by the well-known Popov's criterion in the finite-dimensional theory [\cite[p44]{Za} under the name of complete stabilization] allows us to obtain the stabilizing controls in feedback form, thus completing the proof of Theorem \ref{B-4.1}. Details are in \cite{LT2:2015}.
	\end{proof}
	
	\section{Global well-posedness and Uniform Exponential Stabilization of the Linearized $w$-problem (\ref{B-1.28}) in $\lso, q \geq 2$, and $\Bto$  by means of the same feedback controls $\{v,u\}$ obtained for the $w_N$-problem in Section \ref{B-Sec-4}.}\label{B-Sec-5}
	
	\subsection{The operator $\BA_{_{F,q}}$ defining the linearized $w$-problem in feedback form.}\label{B-Sec-5.1}	
	
		Let $q \geq 2$. Consider the same K-dimensional feedback controllers constructed in Theorem \ref{B-Thm-4.1} and yielding estimate (\ref{B-4.4}), (\ref{B-4.5}) for the finite-dimensional projected $w_N$-system (\ref{B-3.8}) in feedback form (\ref{B-4.6}); that is, the tangential boundary controller $v = v_N$ supported on $\wti{\Gamma}$, and the tangential-like interior controller $u = u_N$ supported on $\omega$		
		\begin{align}
		v = v_N &= \sum_{k=1}^K \nu_k(t) f_k = \sum_{k=1}^K \big<w_N(t),p_k\big>_{_{W^u_N}} f_k, \quad f_k \in \calF \subset W^{2-\rfrac{1}{q},q}(\Gamma), \ p_k \in (W^u_N)^* \subset \lo{q'}, \ q \geq 2 \nonumber \\ & \hspace{6cm} \ f_k \cdot \nu |_{\Gamma} = 0;  \text{ hence } v \cdot \nu|_{\Gamma} = 0, \  f_k \text{ supported on } \wti{\Gamma} \label{B-5.1}\\
		u = u_N &= \sum_{k=1}^K \mu_k(t) u_k = \sum_{k=1}^K \big<w_N(t),q_k\big>_{_{W^u_N}} u_k, \quad q_k \in (W^u_N)^* \subset \lo{q'},\ u_k \text{ supported on } \omega. \label{B-5.2}
		\end{align}
		\noindent 
			 Once inserted, this time, in the full linear $w$-problem (\ref{B-1.28}) or (\ref{B-2.26}) = (\ref{B-3.1}), such $v$ and $u$ in (\ref{B-5.1}), (\ref{B-5.2}) yield the linearized feedback dynamics ($w_N = P_N w$) driven by the dynamical feedback stabilizing operator $\ds \BA_{_{F,q}}$ below		
			\begin{equation}\label{B-5.3}
			\frac{dw}{dt} = \calA_q w - \calA_q D \Bigg( \sum_{k=1}^{K} \big<P_N w,p_k\big>_{_{W^u_N}} f_k \Bigg) + P_q \Bigg ( m \Bigg( \sum_{k=1}^{K} \big<P_N w,q_k\big>_{_{W^u_N}} u_k \Bigg) \tau \Bigg) \equiv \BA_{_{F,q}} w.
			\end{equation}			
			\noindent More specifically $\ds \BA_{_{F,q}}$ is rewritten as in the subsequent Section \ref{B-Sec-6}, Eqts \eqref{B-6.2}, \eqref{B-6.9a} as			
			\begin{equation}\label{B-5.4}
			\BA_{_{F,q}} = A_{_{F,q}} + G: \lso \supset \calD \big( \BA_{_{F,q}} \big) \longrightarrow \lso, \ q \geq 2 
			\end{equation}
			\begin{subequations}\label{B-5.5}
				\begin{align}
				A_{_{F,q}} &= \calA_q (I - DF) \ : \ \lso \supset \calD(A_{_{F,q}}) \longrightarrow \lso, \ q \geq 2 \label{B-5.5a}\\
				\begin{picture}(0,0)
				\put(-15,9){$\left\{\rule{0pt}{21pt}\right.$}\end{picture}
				\calD(A_{_{F,q}}) &= \big\{ h \in \lso: h - DFh \in \calD(\calA_q) = W^{2,q}(\Omega) \cap W^{1,q}_0(\Omega) \cap \lso \big\} = \calD(\BA_{_{F,q}}) \label{B-5.5b}
				\end{align}
			\end{subequations}
			\begin{subequations}\label{B-5.6}
				\begin{equation}
				F(\cdot)  = \sum_{k=1}^K \big< P_N \ \cdot, p_k \big>_{_{W^u_N}} f_k \in W^{2-\rfrac{1}{q},q}(\wti{\Gamma}); \ G(\cdot)  = P_q \bigg( m \bigg( \sum_{k=1}^K \big< P_N \ \cdot, q_k \big>_{_{W^u_N}} u_k \bigg) \tau \bigg) \in \lso \label{B-5.6a}
				\end{equation}	
				\begin{equation}\label{B-5.6b}
				F \in \calL(\lso, L^q(\wti{\Gamma})); \quad G \in \calL(\lso), \ q \geq 2 .
				\end{equation}				
			\end{subequations}
		
		\subsection{The feedback operator $\BA_{_{F,q}}$ in (\ref{B-5.3}) generates a s.c analytic semigroup in $\lso, \ 2 \leq q < \infty$ or in $\ds \Bto, \ 1 < p < \rfrac{2q}{2q-1}, \ q > d, \ d = 2,3$.}\label{B-Sec-5.2}
		
		\begin{thm}\label{B-Thm-5.1}
			Let $q \geq 2$. With reference to the feedback operator $\BA_{_{F,q}}$ in (\ref{B-5.4}) describing the feedback $w$-system in (\ref{B-5.3}), with $\ds \calD \big( \BA_{_{F,q}} \big) = \calD \big( A_{_{F,q}} \big) = \Big\{ \varphi \in W^{2,q}(\Omega) \cap \lso: \ \varphi \big|_{\Gamma} = F \varphi \Big\}$ we have: $\ds \BA_{_{F,q}}$ generates a s.c. analytic semigroup $\ds e^{\BA_{_{F,q}}t}$ on $\ds \lso, \ t > 0, \ q \geq 2$; $\ds \BA_{_{F,q}}$ has a compact resolvent on $\ds \lso, q \geq 2$.
		\end{thm} 
	
	\begin{proof}
		For $q \geq 2$ the finite dimensional feedback operators $\ds F: \lso \longrightarrow L^q(\Gamma)$ and $G$ on $\lso$ are bounded. This in turn, is due to Appendix \ref{B-app-B}, in particular Eq (\ref{B-B.5}): $\ds \varphi^*_{ij} \in W^{3,q}(\Omega), \ q \geq 2$ and		
		\begin{equation*}
		\frac{\partial \varphi_{ij}^*}{\partial \nu} \bigg|_{\Gamma} \in W^{2-\rfrac{1}{q},q}(\Gamma) \subset L^q(\Gamma), \ \text{ and } \calF \subset W^{2-\rfrac{1}{q},q}(\Gamma).
		\end{equation*}		
		\noindent Thus, it suffices to consider the operator $\ds A_{_{F,q}}$, below in (\ref{B-5.5a}), which differs from $\ds \BA_{_{F,q}}$ by the bounded operator $G$. We give two proofs.
		
		\noindent \uline{First proof (after \cite{BLT1:2006} $q = 2$):} We shall critically use property (\ref{B-3.41}) for the Dirichlet map $D$ in the $L^q$-setting, $1 < q < \infty$, just as it was done in these references in the Hilbert setting; namely that, with $\varepsilon > 0$, recalling (\ref{B-2.70b}) of Corollary \ref{B-Clr-2.8}, we have:
		\begin{multline}\label{B-5.7}
		D: \text{ continuous } U_q = \big\{ g \in L^q(\Omega) = W^{0,q}(\Omega), \ g \cdot \nu = 0 \ \text{on } \Gamma \big\} \longrightarrow \\ W^{\rfrac{1}{q},q}(\Omega) \cap \lso \subset \calD \Big(A_q^{\rfrac{1}{2q}-\varepsilon}\Big),
		\end{multline}
		\noindent $1 < q < \infty$, where $A_q$ is the Stokes operator in \eqref{B-1.8}. We next transfer relation (\ref{B-5.7}) to the Oseen operator $\calA_q$ in (\ref{B-1.11}). To do this, we just translate it. Let $k>0$ be suitably large, then, via (\ref{B-5.7})
		\begin{equation}\label{B-5.8}
		\hat{\calA}_q^{\rfrac{1}{2q}-\varepsilon} D = \big( kI - \calA_q \big)^{\rfrac{1}{2q}-\varepsilon} D: \text{ continuous } \big\{ g \in L^q(\Gamma), \ g \cdot \nu = 0 \ \text{on } \Gamma \big\} \longrightarrow \lso
		\end{equation}
		\noindent Both elements - that $\ds A_{_{F,q}}$ generates a s.c. analytic semigroup on $\lso$ and has compact resolvent on it - rely on the perturbation formula \cite{P:1983} written for $\ds A_{_{F,q}}$ in \eqref{B-5.5a}.
		\begin{equation}\label{B-5.9}
		R \big(\lambda, A_{_{F,q}} \big) = [I + R (\lambda, \calA_q) \calA_qDF]^{-1} R(\lambda, \calA_q)
		\end{equation}
		\noindent where by property (\ref{B-5.8}) $\ds \hat{\calA}_q^{\rfrac{1}{2q} - \varepsilon} DF \in \calL(\lso)$. Moreover, since $\calA_q$ generates a s.c. analytic semigroup in $\lso$ (Theorem \hyperref[B-Thm-A.1]{A.1.ii} of Appendix \ref{B-app-A}), a well-known formula \cite{P:1983} gives for $\ds \varepsilon > 0, \ \theta = 1 - \rfrac{1}{2q} - \varepsilon$
		\begin{equation}\label{B-5.10}
		\norm{R(\lambda, \hat{\calA}_q)\hat{\calA}_q^{\theta}}_{\calL(\lso)} \leq \frac{C}{|\lambda|^{1-\theta}} = \frac{C}{|\lambda|^{\rfrac{1}{2q} + \varepsilon}} \rightarrow 0 \text{ as } |\lambda| \rightarrow \infty.
		\end{equation}
		\noindent Then, (\ref{B-5.10}) in (\ref{B-5.9}) yields 
		\begin{equation}\label{B-5.11}
		\norm{R \big(\lambda, A_{_{F,q}} \big)}_{\calL(\lso)} \leq C_{\rho_o} \norm{R(\lambda, \calA_q)}_{\calL(\lso)}, \quad \forall \ \lambda, \ |\lambda| \geq \text{ some }\rho_o > 0
		\end{equation}
		\noindent and hence, via (\ref{B-5.11}) the properties of $R(\lambda, \calA_q)$ of Theorem \ref{B-Thm-A.1} [generation of s.c. analytic semigroup on $\lso$ and, respectively, compact resolvent] transfer into corresponding properties for $R(\lambda, A_{_{F,q}})$. Theorem \ref{B-Thm-5.1} is proved.\\
		
		\noindent \uline{Analyticity: second proof:} One may provide a second proof that $\ds A_{_{F,q}}$ generates a s.c. analytic semigroup on $\lso, q \geq 2$. This is still a perturbation argument, however perturbation of an original analytic generator, not of the resolvent. In fact, the present perturbation argument applies to the adjoint operator $\ds A_{_{F,q}}^*$ on $\lo{q'}$, not to $\ds A_{_{F,q}}$ on $\lso, 1 < q' \leq 2, 2 \leq q$. Eqts (\ref{B-6.13}), (\ref{B-6.14}) in the argument below of Proposition \ref{B-Prop-6.2} dealing with $\ds \BA_{_{F,q}} = A_{_{F,q}} + G$ show that $\ds \BA_{_{F,q}}^*$ can be written as $\ds \BA_{_{F,q}}^* = -A^*_q + \Pi$, where the perturbation $\Pi$ is $\ds \big( A^*_q \big)^{\theta}$-bounded, with $\ds \theta = 1 - \rfrac{1}{2q} + \varepsilon < 1$, see (\ref{B-6.17}). Thus, since $-A^*_q$ generates s.c. analytic semigroup on $\lo{q'}$ (by adjointness on Theorem \hyperref[B-Thm-A.1]{A.1(i)} on $-A_q$), then a standard semigroup result \cite{P:1983} implies that the perturbed operator $\ds \BA_{_{F,q}}^*$ is an analytic semigroup generator on $\lo{q'}, 1 < q' \leq 2$. But this is equivalent ($\lso$ being reflexive, $1 < q < \infty$) to the original operator $\ds \BA_{_{F,q}}$ being an analytic semigroup generator on $\lso, q \geq 2$, as desired. The quoted Proposition \ref{B-Prop-6.2} shows more. In fact: that is, that $\ds A_{_{F,q}}$ (actually $\ds \BA_{_{F,q}} = A_{_{F,q}} + G$) has $L^p$-maximal regularity on $\lso$, in symbols, $\ds A_{_{F,q}} \in MReg \big( L^p (0, \infty; \lso) \big), q \geq 2$ (\ref{B-6.18}). And maximal regularity implies analyticity \cite{Dore:2000}, The argument of Proposition \ref{B-Prop-6.2} is of the perturbation type described above, however tuned to the notion of maximal regularity, which is stronger than analyticity. 		
	\end{proof}
		
	\noindent We next extend Theorem \ref{B-Thm-5.1} to the Besov space $\ds \Bto$ in (\ref{B-1.15b}) $\ds 1 < p < \rfrac{2q}{2q-1}, \ q \geq 2$, of interest. To this end, we need the following result.
	\begin{prop}\label{B-Prop-9.3}
		Let $\ds 1 < p < \frac{2q}{2q - 1}, q \geq 2$. Then
		\begin{align}
		\lqafq &= \Bto \label{B-5.12}\\
		&= \Big\{ g \in \Bso: \text{ div } g \equiv 0, \ g \cdot \nu|_{\Gamma} = 0 \Big\}. \label{B-5.13}
		\end{align}
		\end{prop}
	\begin{rmk}\label{B-Rmk-5.1}
		This formula should be compared with the original definition of $\ds \Bto$ in \eqref{B-1.15b}. 
	\end{rmk}
		\begin{proof}
		\noindent \underline{Step 1:} From the characterization of $\ds \calD \big( \BA_{_{F,q}} \big) = \calD \big( A_{_{F,q}} \big)$ in (\ref{B-5.5b}) we obtain for $0 < \theta < 1, \ p > 1,\ q \geq 2$		
		\begin{equation}\label{B-5.14}
		\big(\lso, \calD(\BA_{_{F,q}} )\big)_{\theta, p} \subset \big( \lso, W^{2,q}(\Omega) \cap \lso \big)_{\theta, p} = B^{2 \theta}_{q,p}(\Omega) \cap \lso			
		\end{equation}	
		\noindent recalling the definition/characterization (\ref{B-1.13a}) of $\ds B^s_{q,p}(\Omega)$ with $\ds m = 2, \ \rfrac{s}{2} = \theta$. Next we take $\ds 1 < p < \rfrac{2q}{2q-1},\ q \geq 2, \ \theta = 1 - \rfrac{1}{p}$, so that - for these parameters - (\ref{B-5.14}) specializes to 
		\begin{equation}
		\lqafq \subset \Bto = \text{ defined in \eqref{B-1.15b}}  \label{B-5.15}
		\end{equation} 
		\noindent \underline{Step 2:} But $\ds \Bto$ does not recognize boundary conditions [the conditions div $g = 0, \ g \cdot \nu|_{\Gamma} = 0$ are included in the definition of the underlying space $\ds \lso$, see Remark \ref{B-rmk-1.4}].\\			
		\noindent Hence, the space in the LHS of (\ref{B-5.15}) does not recognize boundary conditions. Thus, for the indexes $\ds \big\{ \theta = 1 - \rfrac{1}{p}, p\big\}$, with $\ds 1 < p < \rfrac{2q}{2q-1}, q \geq 2$, we have recalling (\ref{B-1.15b})
		\begin{align}
		\lqafq &= \lqaq = \Bto \label{B-5.16}
		\end{align}
		\noindent as $\ds \calD \big( \BA_{_{F,q}} \big)$ and $\ds \calD(A_q)$ both consist of $\ds W^{2,q}(\Omega) \cap \lso$ functions, subject only to different boundary conditions. Thus (\ref{B-5.16}) proves the desired conclusion (\ref{B-5.12}), (\ref{B-5.13}).
		\end{proof}			
	
	\begin{thm}\label{B-Thm-9.4}
		The operator $\ds \BA_{_{F,q}}$ in (\ref{B-5.4}), where the bounded operators $F$ and $G$ are defined by (\ref{B-5.6}), generates a s.c. analytic semigroup $e^{\BA_{_{F,q}}t}$ on the Besov space $\ds \Bto, \ 1 < p < \rfrac{2q}{2q-1}, \ q \geq 2$ defined in (\ref{B-5.13}).  
	\end{thm}
	
	\begin{proof}
		The operator $\ds \BA_{_{F,q}}$ generates a s.c. analytic semigroup $\ds e^{\BA_{_{F,q}}t}$ on $\lso$ by Theorem \ref{B-Thm-5.1} for $q \geq 2$. Then, it generates a s.c. analytic semigroup on $\ds \calD \big( \BA_{_{F,q}} \big)$. Hence the conclusion follows by (\ref{B-5.12}).
	\end{proof}

	\subsection{The analytic semigroup $\ds e^{\BA_{_{F,q}}t}$ is uniformly stable on $\lso$ and $\Bto$.}\label{B-Sec-5.3}
	
	\begin{thm}\label{B-Thm-5.4}
		The s.c. analytic semigroup $\ds e^{\BA_{_{F,q}}t}$ defining the feedback $w$-dynamics in \eqref{B-5.3} is uniformly stable with decay rate $\gamma_0 > 0$ in both the space $\lso, \ 2 \leq q < \infty$, as well as in the space $\lqaq, \ 2 \leq q < \infty$; in particular, in the space $\Bto, \ 2 \leq q < \infty, \ 1 < p < \rfrac{2q}{2q-1}$: there exists $C_{\gamma_0} > 0$ such that 
	\end{thm}

	\begin{equation}\label{B-5.17}
	\norm{e^{\BA_{_{F,q}}t} w_0}_{\lso} = \norm{w(t;w_0)}_{\lso} \leq C_{\gamma_0} e^{-\gamma_0 t}\norm{w_0}_{\lso}, \quad t \geq 0, \ q \geq 2
	\end{equation} 
	\noindent or for $0 < \theta < 1, \delta > 0$ arbitrarily small, $q \geq 2$
	\begin{subequations}\label{B-5.18}		
		\begin{numcases}{\norm{ A^{\theta}_q \  e^{\BA_{_{F,q}}t} w_0}_{\lso} = \norm{ A^{\theta}_q \ w(t;w_0)}_{\lso}=}
		C_{\gamma_0, \theta} e^{-\gamma_0 t}\norm{A^{\theta}_q \ w_0}_{\lso},  t \geq 0, w_0 \in \calD(A^{\theta}_q). \label{B-5.18a}\\
		C_{\gamma_0, \theta, \delta} e^{-\gamma_0 t}\norm{w_0}_{\lso}, \quad t \geq \delta > 0. \label{B-5.18b}
		\end{numcases}
	\end{subequations}
	\begin{multline}\label{B-5.19}
	\norm{e^{\BA_{_{F,q}}t} w_0}_{\Bto} = \norm{w(t;w_0)}_{\Bto} \leq C_{\gamma_0} e^{-\gamma_0 t}\norm{w_0}_{\Bto}, \ t  \geq 0\\
	2 \leq q < \infty, \ 1 < p < \frac{2q}{2q-1}
	\end{multline}
	
	\begin{proof}
		\noindent \uline{On $\lso$)} A proof of \eqref{B-5.17}-\eqref{B-5.18} on the Hilbert space $H$ in \eqref{B-1.6b} (i.e. $q = 2$) is given in \cite[Lemma 2.3]{LT2:2015}. Essentially the same proof works on $\lso$, using of course the estimate \eqref{B-4.4} now:
		\begin{equation}\label{B-5.20}
		\norm{w_N(t)}_{\lso} + \norm{v_N(t)}_{L^q(\wti{\Gamma})} + \norm{u_N(t)}_{L^q(\omega)} \leq C_{\gamma_{1}} e^{-\gamma_1t}\norm{P_N w_0}_{\lso}, t \geq 0, \ q \geq 2
		\end{equation}
		((\ref{B-4.4}) includes also $v'_N$ and $u'_N$). We just sketch the strategy.\\
		
		\noindent Next one examines the impact of such constructive feedback control pair $\{v_N,u_N \cdot \tau\}$ on the $\zeta_N$-dynamics (\ref{B-3.9}), whose explicit solution is given by the variation of parameter formula
		\begin{align}
		\zeta_N(t) &= e^{\calA^{s}_Nt} \zeta_N(0) + (I_{\rm int})(t) + (I_{\rm bry})(t); \label{B-5.21}\\
		\norm{e^{\calA^{s}_N t}}_{\calL(\lso)} &\leq  C_{\gamma_{0}} e^{-\gamma_0t}, \ 0 \leq t, \ 0 < \gamma_0 < |\mbox{Re } \lambda_{N+1}|; \label{B-5.22}\\
		\big(I_{\rm int} \big)(t) & = -\int^t_0 e^{\calA^s_N(t-r)} (I-P_N)P(m(u_N(r) \cdot \tau(r))dr; \label{B-5.23}\\
		\big(I_{\rm bry} \big)(t) & = -\int^t_0 e^{\calA^s_N(t-r)} \calA^s_N(I-P_N)Dv_N(r)dr; \label{B-5.24}
		\end{align}
		Here, $I_{\rm int}$ is the integral term driven by the interior control $u_N$, while $I_{\rm bry}$ is the integral term driven by the tangential boundary control $v_N$. We omit the details. See \cite{LT2:2015}.\\
		
		\noindent \uline{On $\Bto$.} The proof is similar using the fact that, by \eqref{B-A.3} of Appendix \ref{B-app-A}, the s.c. analytic semigroup of the Oseen operator $\ds e^{\calA_q t}$, once restricted on the stable subspace $W^s_N$, has the property 
		\begin{equation}\label{B-5.25}
		e^{\calA^s_{N,q}t}: \text{ continuous } \Bto \longrightarrow \xipqs, \quad \norm{e^{\calA^s_{N,q}t}}_{\calL \big( \Bto \big)} \leq C e^{-\gamma_0 t} , \ t \geq 0
		\end{equation}
		\noindent counterpart of (\ref{B-5.21}). We now repeat the above proof, except on the space $\ds \Bto$ rather than $\lso$., by using (\ref{B-5.24}) instead of (\ref{B-5.21}).			
	\end{proof}		

	\section{Maximal $L^p$-regularity on $\lso, \ q \geq 2$ and $\Bto$ up to $T = \infty$ of the s.c. analytic semigroup $e^{\BA_{_{F,q}}t}$ yielding uniform decay of the linearized feedback $w$-problem (\ref{B-5.3}) of Theorem \ref{B-5.1}.}\label{B-Sec-6}

\noindent \textbf{Preliminaries}\\
1. \noindent We recall the tangential boundary feedback operator $\ds F \in \calL \big( \lso, L^q(\wti{\Gamma}) \big)$, for $q \geq 2$ and the interior tangential-like feedback operator $\ds G \in \calL \big(\lso \big)$ from (\ref{B-5.6}), $q \geq 2$,
\begin{multline}\label{B-6.1}
F(\cdot)  = \sum_{k=1}^K \big< P_N \ \cdot, p_k \big>_{_{W^u_N}} f_k \in W^{2-\rfrac{1}{q},q}(\wti{\Gamma}); \ G(\cdot)  = P_q \bigg( m \bigg( \sum_{k=1}^K \big< P_N \ \cdot, q_k \big>_{_{W^u_N}} u_k \bigg) \tau \bigg) \in \lso,
\end{multline}

\noindent so that we rewrite the feedback $w$-equation (\ref{B-5.3}) as 	
\begin{subequations}\label{B-6.2}
	\begin{align}
	\frac{dw}{dt} &= \calA_q(I - DF)w + Gw = \BA_{_{F,q}}w \label{B-6.2a}\\[3mm]
	\BA_{_{F,q}}&: \lso \supset \calD \big( \BA_{_{F,q}} \big) \longrightarrow \lso, \quad q \geq 2 \nonumber \\
	\begin{picture}(0,0)\put(-13,8){$\left\{\rule{0pt}{20pt} \right.$}\end{picture}
	\calD \big( \BA_{_{F,q}} \big) &= \big\{ h \in \lso : (h - DFh) \in \calD(\calA_q) = \calD(A_q) \big\}, \label{B-6.2b}
	\end{align}
	\noindent since $G$ is a bounded operator $G \in \calL \big( \lso \big)$. Recall from (\ref{B-1.25}), (\ref{B-4.1}), ultimately Appendix \ref{B-app-B}, \eqref{B-B.7} that the boundary vectors $f_k$ (= linear combinations of normal traces of eigenfunctions of $\calA^* = \calA_q^*$, the adjoint of the Oseen operator \eqref{B-1.11}, have the regularity $\ds f_k \in W^{2-\rfrac{1}{q},q}(\Gamma)$, so that $DFh \in W^{2,q}(\Omega) \cap \lso$ for $h \in \lso, \ q \geq 2$, in light of Corollary \hyperref[B-B.2]{B.2(v)} in Appendix \ref{B-app-B}. Thus, we can more specifically describe $\ds \calD(\BA_{_{F,q}})$ as follows.		
	\begin{align}\label{B-6.2c}
	\BA_{_{F,q}}&: \lso \supset \calD \big( \BA_{_{F,q}} \big) \longrightarrow \lso, \nonumber \\
	\begin{picture}(0,0)\put(-13,8){$\left\{\rule{0pt}{20pt} \right.$}\end{picture}		
	\calD \big( \BA_{_{F,q}} \big) &= \big\{ \varphi \in W^{2,q}(\Omega) \cap \lso: \varphi|_{\Gamma} = F \varphi \big\}, \ q \geq 2
	\end{align}
\end{subequations}
	\noindent see (\ref{B-5.4}). Such characterization of $\ds \calD \big( \BA_{_{F,q}} \big)$ will be critical in using maximal $L^p$ regularity of $\ds \BA_{_{F,q}}$ in the analysis of the non-linear problem in Sections \ref{B-Sec-7} and \ref{B-Sec-8}.\\

	\noindent We also recall that $\ds e^{\BA_{_{F,q}}t}$ is a s.c. analytic semigroup in $\lso$ (Theorem \ref{B-Thm-5.1}), which moreover is uniformly stable here (Theorem \ref{B-Thm-5.4}, Eq (\ref{B-5.17})):

	\begin{equation}\label{B-6.3}
		\norm{e^{\BA_{_{F,q}}t}}_{\calL \big( \lso \big)} \leq C_{\gamma_0} e^{-\gamma_0 t}, \quad t \geq 0, \ q \geq 2. 
	\end{equation}
	\noindent 2. We consider the system
	\begin{align}
		\frac{d \eta}{dt} = \BA_{_{F,q}} \eta + f; \quad \eta(0) = \eta_0 \text{ in } \lso, \ q \geq 2 \label{B-6.4}\\
		\eta(t) = e^{\BA_{_{F,q}}t} \eta_0 + \int_{0}^{t} e^{\BA_{_{F,q}}(t -\tau)}f(\tau) \ d\tau. \label{B-6.5}
	\end{align}
	\noindent \textbf{Goal}: The goal of the present section is to establish maximal $L^p$ regularity on $\lso$ and for $T = \infty$ of the feedback analytic generator $\BA_{_{F,q}}$ in (\ref{B-6.2}), as described in the following result. 

\begin{thm}\label{B-Thm-6.1}
	Let $q \geq 2$. With reference to the dynamics (\ref{B-6.4}), (\ref{B-6.5}) with $\eta_o = 0$, we have: the map
	\begin{subequations}\label{B-6.6}
		\begin{align}
		f \longrightarrow \eta(t) = \int_{0}^{t} e^{\BA_{_{F,q}}(t-\tau)}f(\tau) \ d\tau: &\text{ continuous }  \label{B-6.6a}\\
		L^p(0,\infty; \lso) \longrightarrow &L^p \big( 0,\infty; \calD(\BA_{_{F,q}}) \big), \ 1 < p < \infty, \label{B-6.6b}\\
		L^p(0,\infty; \lso) \longrightarrow & \xipqs \big( \BA_{_{F,q}} \big) \equiv L^p \big(0,\infty; \calD(\BA_{_{F,q}}) \big) \cap W^{1,p}(0,\infty;\lso), \label{B-6.6c}
		\end{align}
	\end{subequations}
	\noindent by (\ref{B-6.2c}), so that, there exists a constant $C = C_{p,q} > 0$ such that
	\begin{subequations}\label{B-6.7}
		\begin{equation}\label{B-6.7a}
		\norm{\eta_t}_{\lplsq} + \norm{\BA_{_{F,q}} \eta}_{\lplsq} \leq C \norm{f}_{\lplsq}.
		\end{equation}
		\noindent In short:
		\begin{equation}\label{B-6.7b}
		\BA_{_{F,q}} \in MReg (\lplsq)
		\end{equation}
	\end{subequations}
	
	\noindent If we introduce the space of maximal regularity for $\{ \eta, \eta_t \}$, with $\eta_0 = 0$, as
	\begin{subequations}\label{B-6.8}
		\begin{align}
		\xipqs \big( \BA_{_{F,q}} \big) &\equiv L^p \big(0,\infty; \calD(\BA_{_{F,q}}) \big) \cap W^{1,p}(0,\infty;\lso) \label{B-6.8a}\\
		&\hspace{5cm}\subset \xipq \equiv  L^p(0, \infty; W^{2,q}(\Omega)) \cap W^{1,p}(0, \infty; L^q(\Omega)), \label{B-6.8b}
		\end{align}
		\noindent we rewrite (\ref{B-6.7}) as
		\begin{equation}\label{B-6.8c}
		f \in \lplsq \longrightarrow  \eta  \in \xipqs\big( \BA_{_{F,q}} \big) \hookrightarrow C \big( [0,\infty); \Bso \big)
		\end{equation}
		where to justify the continuous embedding in (\ref{B-6.8c}), we recall \cite[Theorem 4.10.2; p180]{HA:2000} and the characterization (\ref{B-6.2c}) for $\calD \big( \BA_{_{F,q}} \big)$
	\end{subequations}
\end{thm}
\begin{proof}
	\noindent \underline{Step 1:} Because of the intrinsic presence of the operator $DF$ (boundary feedback $F$ followed by the Dirichlet map $D$) as a right factor in 
	\begin{subequations}\label{B-6.9}
		\begin{align}
		\BA_{_{F,q}} &= \calA_q (I - DF) + G = (-\nu_o A_q - A_{o,q})(I - DF) + G \label{B-6.9a}\\
		&: \lso \supset \calD(\BA_{_{F,q}}) \text{ in (\ref{B-6.2b})} \longrightarrow \lso, \ 2 \leq q < \infty \label{B-6.9b}
		\end{align}
	\end{subequations}
	\noindent recall \eqref{B-1.11} = (\ref{B-A.2}), we find it is necessary to consider instead the more amenable adjoint/dual operator (with $\nu_o = 1$ wlog)
	\begin{subequations}\label{B-6.10}
		\begin{align}
		\BA_{_{F,q}}^* &= (I - DF)^* \calA_q^* + G^* = -(I - DF)^*A_q^* - (I - DF)^*A_{o,q}^* + G^* \label{B-6.10a}\\
		\calD \big( \BA_{_{F,q}}^* \big) &= \calD \big( \calA_q^* \big) = \calD \big( A^*_q \big) = \big\{ h \in W^{2,q'}(\Omega) \cap W^{1,q'}_0(\Omega) \cap \lo{q'} \big\} \label{B-6.10b}\\
		\BA_{_{F,q}}^* &: \lo{q'} \supset \calD \big( \BA_{_{F,q}}^* \big) \text{ in (\ref{B-6.10b})} \longrightarrow \lo{q'}, \ 1 < q' \leq 2 \label{B-6.10c}
		\end{align}
	\end{subequations}	
	\noindent Here $\ds \rfrac{1}{q} + \rfrac{1}{q'} = 1$, where $q \geq 2$ for $\ds \BA_{_{F,q}}$. In order to have $F$ bounded $\ds \lso \longrightarrow L^q(\Gamma)$, we need to impose $1 < q' \leq 2$, in which case $(I-DF)^* \in \calL(\lo{q'}), \ 1 < q' \leq 2$, see Appendix \ref{B-app-B}, Eqt (\ref{B-B.11}). We rewrite $\ds \BA_{_{F,q}}^*$ in (\ref{B-6.10a}) as 
	\begin{align}
	\BA_{_{F,q}}^* &= -A_q^* + \big[F^*D^*{A_q^*}^{\rfrac{1}{2q}-\varepsilon} \big] {A_q^*}^{1-\rfrac{1}{2q}+\varepsilon} - \big[ (I-DF)^* \big( A_q^{\rfrac{-1}{2}} A_{o,q} \big)^* \big]{A_q^*}^{\rfrac{1}{2}} + G^* \label{B-6.11}\\
	&: \lo{q'} \supset \calD(\BA_{_{F,q}}^*) \longrightarrow \lo{q'}, \quad \frac{1}{q} + \frac{1}{q'} = 1, \ q \geq 2, \ 1 < q' \leq 2 \label{B-6.12}
	\end{align}
	\noindent whereby the adjoint of the right factor becomes now a left factor. In obtaining in (\ref{B-6.10a}) the form of $\ds \BA_{_{F,q}}^*$ from that of $\ds \BA_{_{F,q}}$ in (\ref{B-6.9a}), we have used that $\ds (I-DF) \in \calL(\lso) \ q \geq 2$ \cite[p 14]{Fat.1}. Moreover, to go from (\ref{B-6.9}) to (\ref{B-6.11}), we use $\ds A_{o,q} = A_q^{\rfrac{1}{2}} \big(A_q^{\rfrac{-1}{2}} A_{o,q} \big)$, hence $\ds A_{o,q}^* = \big(A_q^{\rfrac{-1}{2}} A_{o,q} \big)^*{A_q^*}^{\rfrac{1}{2}}$, where the ( )--term is bounded by (\ref{B-1.10}).\\
	
	\noindent \underline{Step 2:} By duality on Theorem \ref{B-Thm-5.1} on a reflexive Banach space, the operator $\ds \BA_{_{F,q}}^*$ in (\ref{B-6.10}) generates a s.c. analytic semigroup $\ds e^{\BA_{_{F,q}}^*t}$ on $\lo{q'}$, which moreover is uniformly stable by Theorem \ref{B-Thm-5.4} in $\ds \calL \big(\lo{q'} \big), \ 1 < q' \leq 2$, with the same decay rate $\gamma_0 > 0$ in (\ref{B-6.3}) = (\ref{B-5.17}) as $\ds e^{\BA_{_{F,q}}t}$ in $\ds \calL \big(\lso \big), \ q \geq 2$ in Theorem \ref{B-Thm-5.4}.\\
	
	\noindent \underline{Step 3:}
	\begin{prop}\label{B-Prop-6.2}
		For the generator $\BA_{_{F,q}}^*$ in (\ref{B-6.10}) of a s.c. analytic, uniformly bounded semigroup $\ds e^{\BA_{_{F,q}}^*t}$ on $\ds \lo{q'}$, we have: $\ds \BA_{_{F,q}}^* \in MReg \big( L^p(0,\infty;\lo{q'}) \big), \ 1 < q' \leq 2$.
	\end{prop}
	\begin{proof}
		\noindent The proof is based on a perturbation argument. For $q \geq 2$, rewrite (\ref{B-6.11}) as
		\begin{equation}\label{B-6.13}
		\BA_{_{F,q}}^* = -A_q^* + \Pi \hspace{10cm}
		\end{equation}
		\vspace{-0.5cm}
		\begin{subequations}\label{B-6.14}
			\begin{equation}\label{B-6.14a}
				\Pi = \big[F^*D^*{A_q^*}^{\rfrac{1}{2q}-\varepsilon} \big] {A_q^*}^{1-\rfrac{1}{2q}+\varepsilon} - \big[ (I-DF)^* \big( A_q^{\rfrac{-1}{2}} A_{o,q} \big)^* \big]{A_q^*}^{\rfrac{1}{2}} + G^*. 
			\end{equation}
		
		\noindent In (\ref{B-6.14}), both terms in square brackets [ ~ ] are bounded in $\ds \lo{q'}$, and so is $G^*, \ 1 < q' \leq 2$. To this end we use critically and recall (\ref{B-2.70b}):		
		\begin{equation}\label{B-6.14b}
		A_q^{\rfrac{1}{2q} - \varepsilon} D \in \calL (U_q, \lso), \ \text{so } D^*{A_q^*}^{\rfrac{1}{2q} - \varepsilon} \in \calL(\lo{q'},L^{q'}(\Gamma)), \ 1 \leq q' \leq 2:
		\end{equation}
	\end{subequations}
		\noindent while $\ds A_q^{-\rfrac{1}{2}}A_{o,q} \in \calL(\lso)$ by \eqref{B-1.10} or (\ref{B-2.21}).\\
		
		\noindent The following estimates then hold, $q \geq 2, 1 < q' \leq 2$:
		\begin{align}
		i. & \ \norm{\big[F^*D^*{A_q^*}^{\rfrac{1}{2q}-\varepsilon} \big] {A_q^*}^{1-\rfrac{1}{2q}+\varepsilon} x}_{\lo{q'}} \leq C_q \norm{{A_q^*}^{1-\rfrac{1}{2q}+\varepsilon} x}_{\lo{q'}}, \quad \forall \ x \in \calD \big( {A_q^*}^{1-\rfrac{1}{2q}+\varepsilon} \big) \label{B-6.15}\\[2mm] 		
		ii. & \ \norm{\big[ (I-DF)^* \big( A_q^{\rfrac{-1}{2}} A_{o,q} \big)^* \big]{A_q^*}^{\rfrac{1}{2}}x}_{\lo{q'}} \leq C_q \norm{{A_q^*}^{\rfrac{1}{2}} x}_{\lo{q'}}, \nonumber \\
		& \hspace{3.8cm} = C_q \norm{ \Big({A_q^*}^{-\rfrac{1}{2} + \rfrac{1}{2q} - \varepsilon} \Big) {A_q^*}^{1-\rfrac{1}{2q} + \varepsilon}x}_{\lo{q'}} \leq \wti{C}_q \norm{{A_q^*}^{1-\rfrac{1}{2q} + \varepsilon}x}_{\lo{q'}} \label{B-6.16}
		\end{align} 
		\noindent Hence, the perturbation $\Pi$ in (\ref{B-6.14}) satisfies $q \geq 2, \ 1 < q' \leq 2$:
		\begin{equation}\label{B-6.17}
		\norm{\Pi x}_{\lo{q'}} \leq C_q \norm{{A_q^*}^{1-\rfrac{1}{2q}+\varepsilon} x}_{\lo{q'}}, \quad x \in \calD \big( {A_q^*}^{1-\rfrac{1}{2q}+\varepsilon} \big)
		\end{equation}
		\noindent $\rfrac{1}{q} + \rfrac{1}{q'} = 1, \ \varepsilon > 0$. We draw now some consequences from (\ref{B-6.13}), (\ref{B-6.17}):
		\begin{enumerate}[(a)]
			\item The perturbation operator $\Pi$ is ${A_q^*}^{\theta}$-bounded on $\lo{q'}$ with $\theta = 1 - \rfrac{1}{2q} + \varepsilon < 1, \ 1 < q' \leq 2 \leq q$.
			\item On the other hand $\ds A_q^* \in  MReg \big( L^p(0,\infty; \lo{q'})\big)$, from Appendix \hyperref[B-A.d]{A(d)}, in particular Theorem \ref{B-A.4}. In fact, while $A_q$ is the Stokes operator on $\lso, \ 1 < q < \infty, \ A_q^*$ is the Stokes operator on $\lo{q'}$ by (\ref{B-2.31}), $\ds \rfrac{1}{q} + \rfrac{1}{q'} = 1$. Then, properties (a), (b) imply -by the abstract perturbation theorem in the \cite[Appendix B]{LPT.1}, see also \cite[Theorem 6.2. p 311]{Dore:2000}, \cite{KW:2001} and \cite[SNP Remark 1i, p 426 for $\beta = 1$]{KW:2004}, for related results, that $\ds \BA_{_{F,q}}^* \in MReg \big( L^p(0,\infty; \lo{q'})\big), \ 1 < q' \leq 2$ and Proposition \ref{B-Prop-6.2} is proved. 
		\end{enumerate}
	\end{proof}
	\noindent \underline{Step 4:} We now prove Theorem \ref{B-Thm-6.1} that $\ds \BA_{F,q}$ satisfies the maximal $L^p$ regularity on $\lso$:
	\begin{equation}\label{B-6.18}
	\BA_{_{F,q}} \in MReg \Big( \lplsq \Big), \quad 2 \leq q < \infty.
	\end{equation}
	\noindent \underline{Step 4.i:} We invoke the fundamental result of L. Weis \cite[Theorem 1.11 p 76]{KW:2004}, \cite[Theorem p 198]{We:2001}. Since $\ds \BA_{_{F,q}}$ generates a bounded analytic semigroup $\ds e^{\BA_{_{F,q}}t}$ on $\ds \lso, \ 2 \leq q < \infty,$ on a UMD-space \cite[p 75]{KW:2004}, then the sought-after property that $\ds \BA_{_{F,q}} \in MReg \big( \lplsq \big)$ is equivalent to the property that the family $\ds \tau \in \calL \big(\lso \big)$
	\begin{equation}\label{B-6.19}
	\tau = \Big\{ tR \big(it,\BA_{_{F,q}} \big), t \in \BR \backslash \{0\} \Big\} \quad \text{be } \text{$R$-bounded}
	\end{equation}
	\noindent \underline{Step 4.ii:} By the complete duality for $R$-boundedness on $L^q(\Omega), \ 2 \leq q < \infty,$ we have \cite[Corollary 2.11 p90]{KW:2004} that the family $\tau$ in (\ref{B-6.19}) is $R$-bounded if and only if the corresponding dual family $\tau'$ in $\ds \calL \big( \lo{q'} \big)$, $\ds (\lso)^* = \lo{q'})$ by (\ref{B-1.7})
	\begin{equation}\label{B-6.20}
	\tau' = \Big\{ tR \big(it,\BA_{_{F,q}}^* \big), t \in \BR \backslash \{0\} \Big\} \quad \text{is } \text{$R$-bounded}
	\end{equation}
	\noindent \underline{Step 4.iii:} But the $R$-boundedness property in (\ref{B-6.20}) is equivalent, by the same result \cite{KW:2004} to the property that $\ds \BA_{_{F,q}}^* \in MReg \big( L^p (0,\infty; \lo{q'}) \big), 1 < q' \leq 2, \ \rfrac{1}{q} + \rfrac{1}{q'} = 1$, and this is true by Proposition \ref{B-Prop-6.2}. In conclusion: $\ds \BA_{_{F,q}} \in MReg \big( \lplsq \big)$, and Theorem \ref{B-Thm-6.1} is proved.
\end{proof}

\noindent We next examine the regularity of the term $\ds e^{\BA_{_{F,q}}t} \eta_0$ due to the initial condition $\eta_0$ in (\ref{B-6.5}). For the same reasons noted in the Theorem \ref{B-Thm-6.1}, Eqts (\ref{B-6.10}) through (\ref{B-6.12}), we shall equivalently examine the regularity of the adjoint semigroup $\ds e^{\BA_{_{F,q}}^*t}$.  To this end, we need the counterpart of Proposition \ref{B-Prop-9.3} this time for the adjoint/dual operator $\ds \BA_{_{F,q}}^*$ on $\lo{q'}, \ 1 < q' \leq 2, \ \rfrac{1}{q} + \rfrac{1}{q'} = 1$.

\begin{prop}\label{B-Prop-6.3}
	Let $\ds 1 < p < \frac{2q'}{2q'-1}, \ 1 < q' \leq 2, \ q \geq 2, \ \frac{1}{q} + \frac{1}{q'} = 1 $. Then
	\begin{align}
	\lqafqs &= \Btos \label{B-6.21}\\
	&= \Big\{ g \in \Bsos: \div \ g \equiv 0, \ g \cdot \nu|_{\Gamma} = 0 \Big\} \label{B-6.22}
	\end{align}
\end{prop}
\begin{proof}
	\noindent From (\ref{B-6.10b}) we have: $\ds \calD \big( \BA_{_{F,q}}^* \big) \subset W^{2,q'}(\Omega) \cap \lo{q'}, \ 1 < q' \leq 2$. Thus, as in (\ref{B-5.14})
	\begin{equation}\label{B-6.23}
	\big(\lo{q'}, \calD(\BA_{_{F,q}}^* )\big)_{\theta, p} \subset \big( \lo{q'} , W^{2,q'}(\Omega) \cap \lo{q'} \big)_{\theta, p} = B^{2 \theta}_{q',p}(\Omega) \cap \lo{q'}
	\end{equation}
	\noindent recalling the definition (\ref{B-1.14}). Next we take $\ds 1 < p < \rfrac{2q'}{2q'-1}, \ \theta = 1 - \rfrac{1}{p}$, so that, for these parameters, (\ref{B-6.23}) specializes to 
	\begin{equation}\label{B-6.24}
	\lqafqs \subset \Btos = \text{ defined in (\ref{B-1.15b})}.
	\end{equation} 
	\noindent But $\ds \Btos$ does not recognize the boundary conditions so neither does the space on the LHS of (\ref{B-6.24}). Thus for these parameters $\ds \theta = 1 - \rfrac{1}{p}$, with $\ds 1 < p < \rfrac{2q'}{2q'-1}$, we have 
	\begin{align}
	\lqafqs &= \lqaqs \label{B-6.25}\\
	&= \Btos \label{B-6.26}
	\end{align}
	\noindent recalling (\ref{B-2.31}) ($A^*_q$ is the Stokes operator on $\lo{q'}$) and (\ref{B-1.15b}) as $\ds \calD \big( \BA_{_{F,q}}^* \big)$ and $\ds \calD \big( A_q^* \big)$ both consist of $\ds W^{2,q'}(\Omega) \cap \lo{q'}$ functions, subject only to possibly, different boundary conditions. Thus (\ref{B-6.26}) proves the desired conclusion.
\end{proof}
\noindent We conclude this section with results for the semigroup $\ds e^{\BA_{_{F,q}}t}$ on $\Bto$ that yield the solution in the space $\xtpq$ of maximal regularity. This is the companion result of Theorem \ref{B-Thm-6.1}. It is done by duality on the adjoint semigroup $\ds e^{\BA^*_{_{F,q}}t}$ as in the proof of Theorem \ref{B-Thm-6.1}.

\begin{thm}\label{B-Thm-6.4}
	\noindent (i) Let $\ds 1 < p < \rfrac{2q'}{2q'-1}, \ 1 < q' \leq 2, \ q \geq 2, \ \rfrac{1}{q} + \rfrac{1}{q'} = 1$. Consider the adjoint s.c. analytic semigroup $\ds e^{\BA_{_{F,q}}^*t}$ on $\ds \lo{q'}$, which is uniformly stable here, by duality or (\ref{B-5.17}) of Theorem \ref{B-Thm-5.4}. Then (see (\ref{B-6.26}))
	\begin{align}
	e^{\BA_{_{F,q}}^*t} &: \text{ continuous } \Btos = \lqafqs = \lqaqs\\
	& \longrightarrow X^{\infty}_{p,q'} \equiv L^p \big(0, \infty ; W^{2,q'}(\Omega) \big) \cap W^{1,p} \big( 0, \infty; \lo{q'} \big).
	\end{align}
	\noindent (ii) Consider now the original s.c. analytic  feedback semigroup $e^{\BA_{_{F,q}}t}$ on $\lso$, which is uniformly stable here by (\ref{B-5.7}). Let $\ds 1 < p < \rfrac{2q}{2q-1}, \ q \geq 2$. Then, see (\ref{B-5.16})
	\begin{align}
	e^{\BA_{_{F,q}}t}&: \text{ continuous } \Bto = \lqafq = \lqaq \label{B-6.29}\\
	&\longrightarrow \xipq = L^p \big(0, \infty ; W^{2,q}(\Omega) \big) \cap W^{1,p} \big( 0, \infty; \lso \big). \label{B-6.30}
	\end{align}
\end{thm}
\begin{proof}
	\noindent We shall prove (i) and then (ii) will follow by duality.\\
	
	\noindent \underline{Step 1:} Thus, consider $\ds \BA_{_{F,q}}^*$ in $\lo{q'}, \ 1 < q' \leq 2$. Write 
	\begin{equation}\label{B-6.31}
	\chi(t) = e^{\BA_{_{F,q}}^*t}\chi_{_o}, \quad \chi_{_t} = \BA_{_{F,q}}^*\chi, \ \chi(0) = \chi_{_o}. 
	\end{equation}
	\noindent Recalling $\ds \BA_{_{F,q}}^* = -A_q^* + B_1 {A^*_q}^{1 - \rfrac{1}{2q} + \varepsilon} + B_2 {A^*_q}^{\rfrac{1}{2}} + G^*$ from (\ref{B-6.11}), where $B_1,B_2$ are in $\ds \calL \big( \lo{q'} \big)$, we rewrite the equation in (\ref{B-6.31}) as 
	\begin{equation}\label{B-6.32}
	\chi_{_t} = -A_q^*\chi + B_1 {A^*_q}^{1 - \rfrac{1}{2q} + \varepsilon}\chi + B_2 {A^*_q}^{\rfrac{1}{2}}\chi + G^*\chi
	\end{equation}
	\noindent whose solution is 
	\begin{multline}\label{B-6.33}
	\chi(t) = e^{-A^*_qt} \chi_{_o} + \int_{0}^{t} e^{-A^*_q(t-\tau)}B_1{A^*_q}^{1 - \rfrac{1}{2q} + \varepsilon} \chi(\tau) \ d \tau + \int_{0}^{t}e^{-A^*_q(t-\tau)}B_2 {A^*_q}^{\rfrac{1}{2}} \chi(\tau) \ d \tau \\ +\int_{0}^{t}e^{-A^*_q(t-\tau)} G^* \chi(\tau) \ d \tau.
	\end{multline}
	\noindent Hence apply $\ds A^*_q$ throughout,
	\begin{multline}\label{B-6.34}
	A^*_q\chi(t) = A^*_qe^{-A^*_qt} \chi_{_o} + A^*_q\int_{0}^{t} e^{-A^*_q(t-\tau)}B_1{A^*_q}^{1 - \rfrac{1}{2q} + \varepsilon} \chi(\tau) \ d \tau + A^*_q\int_{0}^{t}e^{-A^*_q(t-\tau)}B_2 {A^*_q}^{\rfrac{1}{2}} \chi(\tau) \ d \tau \\
	+ A^*_q\int_{0}^{t}e^{-A^*_q(t-\tau)} G^* \chi(\tau) \ d \tau.
	\end{multline}
	\noindent \underline{Step 2:} We now recall from (\ref{B-2.31}) that $A_q^*$ is nothing but the Stokes operator on the space $\lo{q'}$. Thus, $A^*_q$ enjoys the maximal regularity properties stated for $\ds A_q$ in Appendix \hyperref[B-A.d]{A(d)}, in particular Theorem \ref{B-Thm-A.4} except on $\ds (\lso)' = \lo{q'}$, see (\ref{B-1.7}). We shall use these for each of the four terms of the RHS of (\ref{B-6.34}).\\
	
	\noindent \underline{First term:} By use of estimate (\ref{B-A.20b}), or (\ref{B-A.17}) in Appendix \ref{B-app-A}, we obtain changing $q$ into $q'$
	\begin{equation}\label{B-6.35}
	\norm{A^*_qe^{-A^*_q \cdot} \chi_{_o}}_{\lplqd} \leq C \norm{\chi_{_o}}_{\Btos}.
	\end{equation}
	\noindent \underline{Second Term:} Again by the maximal regularity property of $A^*_q$ in (\ref{B-A.19}), except in $\lo{q'}$, we estimate since $\ds B_1 \in \calL \big( \lo{q'} \big)$
	\begin{align}
	\norm{A^*_q\int_{0}^{\ \cdot} e^{-A^*_q( \ \cdot \ -\tau)}B_1{A^*_q}^{1 - \rfrac{1}{2q} + \varepsilon} \chi(\tau) \ d \tau}_{\lplqd} &\leq C \norm{B_1 {A^*_q}^{1 - \rfrac{1}{2q} + \varepsilon} \chi}_{\lplqd} \label{B-6.36}\\
	\leq \wti{C} \norm{{A^*_q}^{1 - \rfrac{1}{2q} + \varepsilon} \chi}_{\lplqd}
	\leq \varepsilon_1 &\norm{A^*_q \chi}_{\lplqd} + C_{\varepsilon_1}\norm{\chi}_{\lplqd} \label{B-6.38}
	\end{align}
	\noindent after using an interpolation inequality \cite[Thm 5.3, Eq (3)]{HT:1980}, with $\ds \theta = \rfrac{1}{2}$ as in \cite[Appendix B]{LPT.1}, to go across (\ref{B-6.38}).\\
	
	\noindent \underline{Third term:} Similarly, since $\ds B_2 \in \calL \big( \lo{q'} \big)$, via (\ref{B-A.17}):
	\begin{align}
	\norm{A^*_q \int_{0}^{\ \cdot}e^{-A^*_q(\ \cdot \ -\tau)}B_2 {A^*_q}^{\rfrac{1}{2}} \chi(\tau) \ d \tau}_{\lplqd} &\leq C \norm{B_2 {A^*_q}^{\rfrac{1}{2}}}_{\lplqd} \nonumber\\ 
	\leq \wti{C}\norm{{A^*_q}^{\rfrac{1}{2}}}_{\lplqd}	\leq \varepsilon_2 &\norm{A^*_q \chi}_{\lplqd} + C_{\varepsilon_2} \norm{\chi}_{\lplqd}. \label{B-6.40}
	\end{align}
	\noindent \underline{Fourth term:} Finally, since $\ds G^* \in \calL \big( \lo{q'} \big)$, via (\ref{B-A.17})
	\begin{equation}\label{B-6.41}
	\norm{ A^*_q \int_{0}^{ \ \cdot}e^{-A^*_q(\ \cdot \ -\tau)} G^* \chi(\tau) \ d \tau}_{\lplqd} \leq C \norm{\chi}_{\lplqd}. 
	\end{equation}
	\noindent Invoking (\ref{B-6.35}), (\ref{B-6.38}), (\ref{B-6.40}), (\ref{B-6.41}) in (\ref{B-6.34}), we obtain 
	\begin{equation}\label{B-6.42}
	\norm{A^*_q \chi}_{\lplqd} \leq C \norm{\chi_{_o}}_{\Btos} + (\varepsilon_1 + \varepsilon_2)\norm{A^*_q \chi}_{\lplqd} + \wti{C} \norm{\chi}_{\lplqd}
	\end{equation}
	\noindent from which we obtain $\ds 1 < q' \leq 2, \ q \geq 2, \ 1 < p < \rfrac{2q'}{2q'-1}$, with $\varepsilon_1 + \varepsilon_2$ small,
	\begin{equation}\label{B-6.43}
	\norm{A^*_q \chi}_{\lplqd} \leq C \norm{\chi_{_o}}_{\Btos} + \wti{C} \norm{\chi}_{\lplqd}
	\end{equation}
	\noindent \underline{Step 3:} By returning to (\ref{B-6.31}) with $\chi_{_o} \in \Btos$, since $\ds e^{\BA_{_{F,q}}^*t}$ is a s.c. semigroup, uniformly stable in such space $\ds \Btos \subset \lo{q'}$, see (\ref{B-5.19}) with $q$ replaced by $q'$, we obtain a-fortiori		
	\begin{subequations}\label{B-6.44}
		\begin{align}
		\chi_{_o} \in \Btos & \ \xrightarrow{e^{\BA_{_{F,q}}^*t}} \ \chi \in \lplqd \label{B-6.44a}\\[3mm]
		\begin{picture}(5,0)
		\put(-10,14){$\left\{\rule{0pt}{30pt}\right.$}\end{picture}
		\norm{\chi}_{\lplqd} &\leq C \norm{\chi_{_o}}_{\Btos}. \label{B-6.44b}
		\end{align}
	\end{subequations}
	\noindent Substituting (\ref{B-6.44b}) in (\ref{B-6.43}) yields the desired estimate.
	\begin{align}
	\norm{A^*_q \chi}_{\lplqd} &\leq C \norm{\chi_{_o}}_{\Btos} \label{B-6.45}\\
	\chi_{_o} \in \Btos & \ \xrightarrow{e^{\BA_{_{F,q}}^*t}} \ \chi \in L^p \big( 0, \infty; \calD (A^*_q) \big)  = L^p \big( 0, \infty; W^{2,q'}(\Omega) \cap W^{1,q'}_{0}(\Omega) \cap \lo{q'} \big)  \label{B-6.46}
	\end{align}
	\noindent continuously see (\ref{B-6.10b}). Consequently, (\ref{B-6.46}) gives
	\begin{equation}\label{B-6.47}
	e^{\BA_{_{F,q}}^*t}: \text{ continuous } \Btos \rightarrow X^{\infty}_{q',p} = L^p \big( 0, \infty; W^{2,q'}(\Omega) \big) \cap W^{1,p} \big( 0, \infty; L^{q'}(\Omega) \big)
	\end{equation}
	\noindent Thus (\ref{B-6.47}) shows part (i) for $\ds e^{\BA_{_{F,q}}^*t}$, based on $\ds \lo{q'}$. As noted, part (ii) then follows by duality. Theorem \ref{B-Thm-6.4} is proved.
\end{proof}

	\section{Local well-posedness of the translated nonlinear $z$-problem (\ref{B-1.27}) or (\ref{B-2.25}) by means of a finite dimensional, tangential-like feedback control pair $\{v,u\}$ on $\{ \wti{\Gamma}, \omega \}$. Case $d = 3, q > 3$.}\label{B-Sec-7}

	Starting with the present section, the nonlinearity of problem (\ref{B-1.1}) will impose for $d = 3$ the requirement $q > 3$, while $q > 2$ for $d = 2$, see (\ref{B-7.19}) below. As our deliberate goal is to obtain the stabilization result in the space $\ds \Bto$ defined by \eqref{B-1.15b}, which does not recognize boundary conditions (Remark \ref{B-rmk-1.4}), then the limitation $p < \rfrac{2q}{2q - 1}$, of this space applies. In conclusion, our well-posedness and stabilization results will hold under the restriction $q > 3, 1 < p < \rfrac{6}{5}$ for $d = 3$; $q > 2, 1 < p < \rfrac{4}{3}$ for $d = 2$. As throughout this paper, $\wti{\Gamma}$ is an open connected subset of the boundary $\Gamma$ of positive surface measure and $\omega$ in a localized collar, supported by $\wti{\Gamma}$ (Fig. 2).\\ 

	\noindent Consider the nonlinear $z$-problem (\ref{B-1.27}) or (\ref{B-2.25}) in the following feedback form in the notation of Theorem \ref{B-Thm-5.1}:
	\begin{subequations}\label{B-7.1}
	\begin{equation}\label{B-7.1a}
	\frac{dz}{dt}  - \calA_q ( I - DF) z + \calN_q z - G_qz = 0; \ z_0 = z(0)
	\end{equation}
	explicitly
	\begin{equation}\label{B-7.1b}
	\frac{dz}{dt}  - \calA_q \Bigg[ z - D \Bigg(  \sum_{k=1}^{K}\big<P_Nz, p_k\big>_{_{W^u_N}}f_k \Bigg)\Bigg] + \calN_q z = P_q \Bigg( m \Bigg( \sum_{k = 1}^{K}\big<P_Nz,q_k\big>_{_{W^u_N}}u_k \Bigg) \tau \Bigg); \ z_0 = z(0)
	\end{equation}
	\end{subequations}	
	\noindent i.e. subject to a feedback controls of the same structure as in the linear $w$-dynamics (\ref{B-5.3}) of Theorem \ref{B-Thm-5.1}, 
	\begin{equation}\label{B-7.2}
		v = Fz = \sum_{k=1}^{K}\big<P_Nz, p_k\big>_{_{W^u_N}}f_k, \quad u = G_q z = P_q \Bigg( m \Bigg( \sum_{k = 1}^{K}\big<P_Nz,q_k\big>_{_{W^u_N}}u_k \Bigg) \tau \Bigg).
	\end{equation}	
	Here $p_k,q_k,f_k,u_k$ are the same vectors as those constructed in Theorem \ref{B-Thm-4.1}, and appearing in (\ref{B-4.6}), (\ref{B-5.1})-(\ref{B-5.3}); $f_k$ supported on $\wti{\Gamma}$, $u_k$ supported on $\omega$.
	Recalling from (\ref{B-5.4}), \eqref{B-5.5} the feedback generator $\BA_{F,q}$, we can rewrite (\ref{B-7.1a}) as
	\begin{equation}\label{B-7.3}
	z_t = \BA_{F,q}z - \calN_q z; \quad z(0) = z_0 
	\end{equation}
	\noindent whose variation of parameters formula is 
	\begin{equation}\label{B-7.4}
	z(t) = e^{\BA_{F,q}t}z_0 - \int_{0}^{t} e^{\BA_{F,q}(t - \tau)}\calN_q z(\tau) d \tau.
	\end{equation}

	\begin{thm}\label{B-Thm-7.1}
	(Well-posedness) Let $\ds d = 3, \ 1 < p < \frac{6}{5}$ and $q > 3$ (in order to satisfy the requirement $\ds p < \rfrac{2q}{2q-1}$). There exists a positive constant $r_1 > 0$ (identified in the proof below in \eqref{B-7.27}), such that if the initial condition $z_0$ satisfies
	\begin{equation}\label{B-7.5}
		\norm{z_0}_{\Bto} < r_1,
	\end{equation}
	then problem (\ref{B-7.3}) defines a unique solution $z$ in the space (see \eqref{B-6.8a})
	\begin{equation}\label{B-7.6}
		\xipqs\big( \BA_{_{F,q}} \big) \equiv L^p(0,\infty ; \calD\big( \BA_{_{F,q}} \big)) \cap W^{1,p}(0, \infty; \lso) \hookrightarrow C([0, \infty); \Bto)
	\end{equation}
	of maximal regularity of the operator $\BA_{_{F,q}}$. More precisely the operator $\BF_q$ introduced below in \eqref{B-7.9} has a unique fixed point on $\xipqs \big( \BA_{_{F,q}} \big)$
	\begin{equation}\label{B-7.7}
	\BF_q(z_0,z) = z, \text{ or } z(t) = e^{\BA_{F,q}t}z_0 - \int_{0}^{t} e^{\BA_{F,q}(t - \tau)}\calN_q z(\tau) d \tau
	\end{equation}
	which therefore is the unique (nonlinear semigroup) solution of problem \eqref{B-7.4} = \eqref{B-7.1} in $\xipqs(\BA_{_{F,q}})$.
	\end{thm}
\begin{proof}
	\noindent The proof will be critically based on the maximal regularity property of $\ds \BA_{_{F,q}}$ Section \ref{B-Sec-6}. We already know from (\ref{B-5.19}) of Theorem \ref{B-Thm-5.4} that for $\ds z_0 \in \Bto, 1 < q < \infty, \ 1 < p < \rfrac{2q}{2q-1}$ we have
	\begin{equation}\label{B-7.8}
	\norm{e^{\BA_{F,q}t} z_0}_{\Bto} \leq M_{\gamma_0} e ^{-\gamma_0 t} \norm{z_0}_{\Bto}, \quad t \geq 0,
	\end{equation}
	\noindent with $M_{\gamma_0}$ possibly depending on $p,q$. Maximal regularity properties corresponding to the solution operator formula in (\ref{B-7.4}) were established in section \ref{B-Sec-6}. Accordingly, for $z_0 \in \Bto$ and $f \in \xipqs \equiv L^p(0, \infty; \calD(\BA_{F,q})) \cap W^{1,p}(0, \infty;\lso)$, $\calD(\BA_{F,q})$ given by (\ref{B-6.2b}), we define the operator $\BF_q$ by
	\begin{equation}\label{B-7.9}
	\BF_q(z_0,f)(t) = e^{\BA_{F,q}t}z_0 - \int_{0}^{t} e^{\BA_{F,q}(t - \tau)}\calN_q f(\tau) d \tau.
	\end{equation}
	\noindent \uline{Claim:} We need to show, equivalently, that under the assumptions of Theorem \ref{B-Thm-7.1}, in particular \eqref{B-7.5}, the operator $\BF_q$ in (\ref{B-7.9}) has a unique fixed point on $\xipqs$ in \eqref{B-7.6}, as stated in \eqref{B-7.7}. The proof of the Claim is accomplished in two steps.\\	
	\noindent \underline{Step 1}:
	\begin{prop}\label{B-Prop-7.2}
		Let $d = 3, \ q>3$ and $\ds 1 < p < \rfrac{6}{5}$. There exists a positive constant $r_1 > 0$ (identified below in (\ref{B-7.27})) and a subsequent constant $r>0$ (identified below in (\ref{B-7.25})) depending on $r_1 > 0$ and the constant $C>0$ in (\ref{B-7.23}), such that with $\ds \norm{z_0}_{\Bto} < r_1$, the operator $\BF_q(z_0,f)$ maps the ball $B(0,r)$ in $\xipqs$ into itself.
	\end{prop}
	\noindent The above Claim will then follow from Proposition \ref{B-Prop-7.2} after establishing that\\	
	\noindent \underline{Step 2}:
	\begin{prop}\label{B-Prop-7.3}
		Let $d = 3, \ q>3$ and $\ds 1 < p < \rfrac{6}{5}$. There exists a positive constant $r_1 > 0$, such that if $\ds \norm{z_0}_{\Bto} < r_1$ as in (\ref{B-7.5}), then there exists a constant $0 < \rho_0 < 1$, depending on the constant $r$ of Proposition \ref{B-Prop-7.2}, such that the operator $\BF_q(z_0,f)$ in \eqref{B-7.9} defines a contraction in the ball $B(0,\rho_0)$ of $\xipqs$ in \eqref{B-7.6}.
	\end{prop}
	
	\noindent The Banach contraction principle then establishes the Claim, once we prove Propositions \ref{B-Prop-7.2} and \ref{B-Prop-7.3}.\\
	
	\noindent \textbf{Proof of Proposition \ref{B-Prop-7.2}}.  \textit{Step 1}: We start from definition (\ref{B-7.9}) of $\BF_q$ and invoke the maximal regularity properties (\ref{B-6.29}), (\ref{B-6.30}) for $e^{\BA_{F,q}t}$ and (\ref{B-6.8c}) for $\ds \int_{0}^{t} e^{\BA_{F,q}(t - \tau)} \calN_q f(\tau) d \tau$. We obtain 
	
	\begin{align}
	\norm{\BF_q(z_0,f)(t)}_{\xipqs} &\leq \norm{e^{\BA_{F,q}t}z_0} _{\xipqs}+ \norm{\int_{0}^{t} e^{\BA_{F,q}(t - \tau)}\calN_q f(\tau) d \tau}_{\xipqs} \label{B-7.10}\\
	&\leq C \Big[ \norm{z_0}_{\Bto} + \norm{\calN_q f}_{L^p(0, \infty; \lso)} \Big]. \label{B-7.11} 
	\end{align}
	
	\noindent \textit{Step 2}: By the definition $\calN_q f = P_q [(f.\nabla)f]$ in (\ref{B-2.23}), we estimate ignoring $\ds \norm{P_q}$ and using, $\ds \sup_{\cdot} \ \big[ \abs{g(\cdot)} \big]^r = [\sup_{\cdot} \ (\abs{g(\cdot)})]^r$	
	\begin{align}
	\norm{\calN_q f}^p_{L^p(0,\infty;\lso)} &\leq \int_{0}^{\infty} \norm{P_q [(f.\nabla)f]}^p_{\lso} dt \leq \int_{0}^{\infty} \bigg\{ \int_{\Omega} \abs{f(t,x)}^q \abs{ \nabla f(t,x)}^q d \Omega \bigg\}^{\rfrac{p}{q}} dt\\
	&\leq \int_{0}^{\infty} \bigg\{ \bigg[ \sup_{\Omega} \abs{ \nabla f(t, \cdot)}^q  \bigg]^{\rfrac{1}{q}} \bigg[ \int_{\Omega} \abs{f(t,x)}^{q} d \Omega \bigg]^{\rfrac{1}{q}}  \bigg\}^p dt\\
	&\leq \int_{0}^{\infty} \norm{\nabla f(t,\cdot)}^p_{L^{\infty}(\Omega)} \norm{f(t,\cdot)}^p_{\lso} dt \label{B-7.15}\\
	&\leq \sup_{0\leq t \leq \infty} \norm{f(t,\cdot)}^p_{\lso} \int_{0}^{\infty} \norm{\nabla f(t,\cdot)}^p_{\lo{\infty}} dt\\
	&= \norm{f}^p_{L^{\infty}(0,\infty; \lso)} \norm{\nabla f}^p_{L^p(0,\infty; \lso)}\label{B-7.17}
	\end{align}
	
	\noindent \textit{Step 3}: The following embeddings hold true:
	\begin{enumerate}[(i)]
		\item \cite[Proposition 4.3, p 1406 with $\mu = 0, s = \infty, r = q$]{GGH:2012}  so that the required formula reduces to $1 \geq \rfrac{1}{p}$, as desired
		\begin{subequations}\label{B-7.18}
			\begin{align}
			f \in \xipqs\big( \BA_{_{F,q}} \big) \hookrightarrow f &\in L^{\infty}(0,\infty; \lso) \label{B-7.18a}\\
			\text{ so that, } \norm{f}_{L^{\infty}(0,\infty; \lso)} &\leq C\norm{f}_{\xipqs} \label{B-7.18b}
			\end{align}
		\end{subequations}
		\item \cite[Theorem 2.4.4, p 74 requiring $C^1$-boundary]{SK:1989}
		\begin{equation}\label{B-7.19}
		W^{1,q}(\Omega) \subset L^{\infty}(\Omega) \text{ for q}>\text{dim }\Omega=d, \ d = 2,3,
		\end{equation}
	\end{enumerate}
	
	\noindent so that, with $p>1, q>3$, in case $d = 3$:
	\begin{align}
	\norm{\nabla f}^p_{L^p(0,\infty; \lso)} &\leq C \norm{ \nabla f}^p_{L^p(0,\infty; W^{1,q}(\Omega))} \leq C \norm{f}^p_{L^p(0,\infty; W^{2,q}(\Omega))} \label{B-7.20}\\
	&\leq C \norm{f}^p_{\xipqs} \label{B-7.21}
	\end{align}
	
	\noindent In going from (\ref{B-7.20}) to (\ref{B-7.21}) we have recalled the definition of $f \in \xipqs = \xipqs\big( \BA_{_{F,q}} \big)$ in (\ref{B-6.8a}). Then, the sought-after final estimate of the non-linear term $\calN_q f, f \in \xipqs$, is obtained from substituting (\ref{B-7.18b}) and (\ref{B-7.21}) into the RHS of (\ref{B-7.17}). We obtain 
	
	\begin{equation}\label{B-7.22}
	\norm{\calN_q f}_{L^p(0,\infty; \lso)} \leq C \norm{f}^2_{\xipqs}, \quad f \in \xipqs.
	\end{equation}
	
	\noindent Returning to (\ref{B-7.11}), we finally, obtain by (\ref{B-7.22})
	
	\begin{equation}\label{B-7.23}
	\norm{\BF_q (z_0, f)}_{\xipqs} \leq C \Big\{ \norm{z_0}_{\Bto} + \norm{f}^2_{\xipqs} \Big\}.
	\end{equation}
	
	\noindent \textit{Step 4}: We now impose the restrictions on the data on the RHS of (\ref{B-7.23}): $z_0$ is in a ball of radius $r_1 > 0$ in $\Bto$ and $f$ is in a ball of radius $r>0$ in $\xipqs$. We further demand that the final result $\BF_q(z_0,f)$ shall lie in a ball of radius $r$ in $\xipqs$. Thus we obtain from (\ref{B-7.23})
	
	\begin{equation}\label{B-7.24}
	\norm{\BF_q(z_0,f)}_{\xipqs} \leq C \Big\{ \norm{z_0}_{\Bto} + \norm{f}^2_{\xipqs} \Big\} \leq C(r_1 + r^2) \leq r
	\end{equation}
	
	\noindent This implies 
	\begin{equation}\label{B-7.25}
	Cr^2 - r + Cr_1 \leq 0 \quad \text{or} \quad \frac{1 - \sqrt{1-4C^2r_1}}{2C} \leq r \leq \frac{1 + \sqrt{1-4C^2r_1}}{2C}
	\end{equation}
	\noindent whereby 
	\begin{equation}\label{B-7.26}
	\begin{Bmatrix}
	\text{ range of values of r }
	\end{Bmatrix}
	\longrightarrow \text{ interval } \Big[ 0, \frac{1}{C} \Big], \text{ as } r_1 \searrow 0
	\end{equation}
	\noindent a constraint which is guaranteed by taking 
	\begin{equation}\label{B-7.27}
	r_1 \leq \frac{1}{4C^2},\ C \text{ being the constant in } (\ref{B-7.23}).
	\end{equation}
	\noindent We have thus established that by taking $r_1$ as in (\ref{B-7.27})  and subsequently $r$ as in (\ref{B-7.25}), then the map
	\begin{multline}\label{B-7.28}
	\BF_q(z_0, f) \text{ takes: }
	\begin{Bmatrix}
	\text{ ball in } \Bto \\
	\text{of radius } r_1
	\end{Bmatrix}
	\times 
	\begin{Bmatrix}
	\text{ ball in } \xipqs \\
	\text{of radius } r
	\end{Bmatrix}
	\text{ into }
	\begin{Bmatrix}
	\text{ ball in } \xipqs \\
	\text{of radius } r
	\end{Bmatrix},\\ \ 3 < q, \ 1 < p < \frac{2q}{2q-1}
	\end{multline}
	\noindent This establishes Proposition \ref{B-Prop-7.2}. \qedsymbol \\
	
	\noindent \textbf{Proof of Proposition \ref{B-Prop-7.3}} \underline{Step 1}: For $f_1,f_2$ both in the ball of $\xipqs$ of radius $r$ obtained in (\ref{B-7.25}) of the proof of Proposition \ref{B-Prop-7.2}, subject to $r_1$ chosen as in \eqref{B-7.27}, we estimate from (\ref{B-7.9}):
	\begin{align}
	\norm{ \BF_q(z_0,f_1) - \BF_q(z_0,f_2)}_{\xipqs} &= \norm{\int_{0}^{t} e^{\BA_{F,q}(t-\tau)} \big[ \calN_q f_1(\tau) - \calN_q f_2(\tau) \big] d \tau}_{\xipqs} \label{B-7.29}\\
	&\leq  \widetilde{m} \norm{\calN_q f_1 - \calN_q f_2}_{L^p(0,\infty;\lso)} \label{B-7.30}
	\end{align}
	\noindent after invoking the maximal regularity property \eqref{B-6.6}.\\
	
	\noindent \underline{Step 2}: Next recalling $\calN_qf = P_q [(f \cdot \nabla )f]$ from (\ref{B-2.23}), we estimate the RHS of (\ref{B-7.30}). In doing so, we add and subtract $(f_2 \cdot \nabla) f_1$, set $ \ds A = (f_1 \cdot \nabla) f_1 - (f_2 \cdot \nabla) f_1, \ B = (f_2 \cdot \nabla) f_1 - (f_2  \cdot \nabla) f_2,$ and use $\abs{A+B}^q \leq 2^q[\abs{A}^q + \abs{B}^q]$. \cite[p 12]{TL:1980} We obtain, again ignoring $\norm{P_q}$:
	
	\begin{align}
	\norm{\calN_q f_1 - \calN_q f_2}_{L^p(0,\infty;\lso)} &\leq \int_{0}^{\infty} \bigg\{ \bigg[ \int_{\Omega} \abs{(f_1 \cdot \nabla) f_1 - (f_2 \cdot \nabla) f_2}^q d \Omega \bigg]^{\rfrac{1}{q}}\bigg\}^p dt\\
	&= \int_{0}^{\infty} \bigg[ \int_{\Omega} \abs{A+B}^q d \Omega \bigg]^{\rfrac{p}{q}}dt \leq 2^q \int_{0}^{\infty} \bigg\{ \int_{\Omega} \big[\abs{A}^q + \abs{B}^q \big] d \Omega \bigg\}^{\rfrac{p}{q}}dt\\
	&= 2^q \int_{0}^{\infty} \bigg\{ \Big[ \int_{\Omega} \abs{A}^q d \Omega + \int_{\Omega} \abs{B}^q d \Omega   \Big]^{\rfrac{1}{q}} \bigg\}^p dt\\
	&= 2^q \int_{0}^{\infty} \bigg\{ \Big[ \norm{A}^q_{\lso} + \norm{B}^q_{\lso} \Big]^{\rfrac{1}{q}} \bigg\}^p dt\\
	&\leq 2^q \cdot 2^{\rfrac{1}{q}}\int_{0}^{\infty} \Big\{ \norm{A}_{\lso} + \norm{B}_{\lso} \Big\}^p dt\\
	&\leq 2^{p+q+\rfrac{1}{q}}\int_{0}^{\infty} \Big[ \norm{A}^p_{\lso} + \norm{B}^p_{\lso} \Big] dt\\
	= 2^{p+q+\rfrac{1}{q}}\int_{0}^{\infty} &\Big[ \norm{((f_1 - f_2)\cdot \nabla)f_1}^p_{\lso} + \norm{(f_2 \cdot \nabla)(f_1 - f_2 )}^p_{\lso} \Big] dt\\
	= 2^{p+q+\rfrac{1}{q}}\int_{0}^{\infty} &\Big\{ \norm{f_1 - f_2}^p_{\lso} \norm{\nabla f_1}^p_{\lso} + \norm{f_2}^p_{L^q(\Omega)} \norm{\nabla(f_1 - f_2)}^p_{L^q(\Omega)} \Big\} dt \label{B-7.39}
	\end{align}
	
	\noindent \underline{Step 3}: We now notice that regarding each of the integral term in the RHS of (\ref{B-7.39}) we are structurally and topologically as in the RHS of (\ref{B-7.15}), except that in (\ref{B-7.39}) the gradient terms $\nabla f_1, \nabla(f_1 - f_2)$ are penalized in the $\lso$-norm which is dominated by the $L^{\infty}(\Omega)$-norm, as it occurs for the gradient term $\nabla f$ in (\ref{B-7.15}). Thus we can apply to each integral term on the RHS of (\ref{B-7.39}) the same argument as in going from (\ref{B-7.15}) to the estimates (\ref{B-7.18b}) and (\ref{B-7.21}) with $q >$ dim $\Omega=3$. We obtain
	\begin{align}
	\norm{\calN_q f_1 - \calN_q f_2}^p_{L^p(0,\infty;\lso)} &\leq \text{RHS of (\ref{B-7.39})} \nonumber \\
	\text{(see (\ref{B-7.17}))} \hspace{3cm}&\leq C_{p,q} \Big\{ \norm{f_1 - f_2}^p_{L^{\infty}(0, \infty;L^q(\Omega))} \norm{\nabla f_1}^p_{L^p(0, \infty;L^{\infty}(\Omega))}\nonumber \\ &\hspace{2cm} + \norm{f_2}^p_{L^{\infty}(0, \infty;L^q(\Omega))} \norm{\nabla(f_1 - f_2)}^p_{L^p(0, \infty;L^{\infty}(\Omega)} \Big\}\\
	\text{(see (\ref{B-7.18b}) and (\ref{B-7.21}))} \qquad &\leq C_{p,q} \Big\{ \norm{f_1 - f_2}^p_{\xipqs} \norm{f_1}^p_{\xipqs} + \norm{f_2}^p_{\xipqs} \norm{f_1 - f_2}^p_{\xipqs} \Big\}\\
	&= C_{p,q} \Big\{ \norm{f_1 - f_2}^p_{\xipqs} \big( \norm{f_1}^p_{\xipqs} + \norm{f_2}^p_{\xipqs}\big) \Big\} \label{B-7.43}
	\end{align}
	\noindent with $\ds C_{p,q} = 2^{p+q+\rfrac{1}{q}}$, finally (\ref{B-7.43}) yields
	\begin{align}
	\norm{\calN_q f_1 - \calN_q f_2}_{L^p(0,\infty;\lso)} &\leq C_{p,q} \norm{f_1 - f_2}_{\xipqs} \Big( \norm{f_1}^p_{\xipqs} + \norm{f_2}^p_{\xipqs} \Big)^{\rfrac{1}{p}}\\
	&\leq 2^{\rfrac{1}{p}} C_{p,q} \norm{f_1 - f_2}_{\xipqs} \Big( \norm{f_1}_{\xipqs} + \norm{f_2}_{\xipqs} \Big) \label{B-7.45}
	\end{align}
	\noindent \underline{Step 4}: Using estimate (\ref{B-7.45}) on the RHS of estimate (\ref{B-7.30}) yields
	\begin{equation}\label{B-12.42}
	\norm{\BF_q(z_0,f_1) - \BF_q(z_0,f_2)}_{\xipqs} \leq K_{p,q} \norm{f_1 - f_2}_{\xipqs} \Big( \norm{f_1}_{\xipqs} + \norm{f_2}_{\xipqs} \Big)
	\end{equation}
	\noindent $K_{p,q} = \widetilde{m}C_{p,q} = \wti{m} 2^{p + \rfrac{1}{p} + q + \rfrac{1}{q}}$ ($\widetilde{m}$ as in (\ref{B-7.30})). Next, recall that $f_1,f_2$ are in the ball of $\xipqs$ of radius $r$ obtained in \eqref{B-7.25}:
	\begin{equation}
	\norm{f_1}_{\xipqs},\norm{f_2}_{\xipqs} \leq r.
	\end{equation}
	\noindent Then 
	\begin{equation}
	\norm{\BF_q(z_0,f_1) - \BF_q(z_0,f_2)}_{\xipqs} \leq \rho_0 \norm{f_1 - f_2}_{\xipqs}
	\end{equation}
	\noindent and $\BF_q(z_0,f)$ is a contraction on the space $\xipqs$ as soon as 
	\begin{equation}
	\rho_0 \equiv 2K_{p,q}r < 1 \text{ or } r < \rfrac{1}{2 K_{p,q}}, \ K_{p,q} = \widetilde{m} 2^{p + \rfrac{1}{p} + q + \rfrac{1}{q}} 
	\end{equation}
	\noindent where $\rho_0$ depends on $r$, hence on $r_1$ in \eqref{B-7.27}. In this case, the map $\BF_q(z_0,f)$ defined in (\ref{B-7.9}) has a fixed point $z$ in $\xipqs$
	\begin{equation}
	\BF_q(z_0,z) = z, \text{ or } z = e^{\BA_{F,q}t}z_0 - \int_{0}^{t}e^{\BA_{F,q}(t - \tau)}\calN_q z(\tau) d \tau
	\end{equation}
	\noindent and such fixed point $z \in \xipqs = \xipqs\big( \BA_{_{F,q}} \big)$ is the unique solution of the translated non-linear equation (\ref{B-7.1}), or (\ref{B-7.6}) with finite dimensional control $u$ in feedback form, as described by the RHS of (\ref{B-7.1}). The claim is proved. \qedsymbol 	
\end{proof}

	\section{Local exponential decay of the non-linear translated $z$-dynamics (\ref{B-7.1}) with finite dimensional, localized, tangential-like, feedback control $\{v,u\}$ on $\{ \wti{\Gamma}, \omega \}$. Case $d = 3$.}\label{B-Sec-8}

\begin{thm}\label{B-Thm-8.1}(Uniform Stabilization) Let $d = 3, 1 < p < \rfrac{6}{5}, q > 3$. Consider the setting of Theorem \ref{B-Thm-7.1}, which provides the solution of the $z$-problem \eqref{B-7.1} in the space $\xipqs\big( \BA_{_{F,q}} \big)$ in \eqref{B-7.6} provided the initial condition $z_0$ satisfies the smallness condition \eqref{B-7.5} with $r_1$ given by \eqref{B-7.27}. If $r_1$ is, possibly, further smaller to satisfy condition \eqref{B-8.18} below, then $z(t;z_0)$ is uniformly stable on the space $\Bto$: there exist constants $\widetilde{\gamma} > 0, M_{\widetilde{\gamma}} \geq 1$, such that said solution satisfies 
	\begin{equation}\label{B-8.1}
	\begin{aligned}
	\norm{z(t;z_0)}_{\Bto} &\leq M_{\widetilde{\gamma}} e^{-\widetilde{\gamma} t}\norm{z_0}_{\Bto}.
	\end{aligned}
	\end{equation}
\end{thm}
\noindent Remark \ref{B-rmk-8.1} will provide insight on the relationship between $\wti{\gamma}$ in the nonlinear case in \eqref{B-8.1} and $\gamma_0  \approx \abs{Re \lambda_{N+1}}$ in the corresponding linear case in \eqref{B-5.13}.
\begin{proof}
	We return to the feedback problem \eqref{B-7.1} rewritten equivalently as in \eqref{B-7.4}
	\begin{equation}\label{B-8.2}
	z(t) = e^{\BA_{F,q}t}z_0 - \int_{0}^{t}e^{\BA_{F,q}(t - \tau)} \calN_q z(\tau) d \tau.
	\end{equation} 
	\noindent For $z_0$ in a small ball of $\Bto$, Theorem \ref{B-Thm-7.1} provides a unique solution in a ball of $\xipqs$ in \eqref{B-7.6}. We recall from (\ref{B-5.17}) = (\ref{B-7.8})
	\begin{equation}\label{B-8.3}
	\norm{e^{\BA_{F,q}t} z_0}_{\Bto} \leq M_{\gamma_0} e^{-\gamma_0 t} \norm{z_0}_{\Bto}, \ t \geq 0.
	\end{equation}
	\noindent Our goal now is to show that for $z_0$ in a small ball of $\Bto$, problem (\ref{B-8.2}) satisfies the exponential decay 
	\begin{equation}
	\norm{z(t)}_{\Bto} \leq C_a e^{-at} \norm{z_0}_{\Bto}, \ t \geq 0, \ \text{for some constants, } a>0, C_a \geq 1. \nonumber
	\end{equation}
	\noindent \underline{Step 1}: Starting from (\ref{B-8.2}) and using (\ref{B-8.3}), we estimate
	\begin{align}
	\norm{z(t)}_{\Bto} &\leq M e^{-\gamma_0 t} \norm{z_0}_{\Bto} + \sup_{0\leq t \leq \infty} \norm{\int_{0}^{t}e^{\BA_{F,q}(t - \tau)} \calN_q z(\tau) d \tau}_{\Bto} \label{B-8.4}\\
	&\leq M e^{-\gamma_0 t} \norm{z_0}_{\Bto} + C \norm{\int_{0}^{t}e^{\BA_{F,q}(t - \tau)} \calN_q z(\tau) d \tau}_{\xipqs} \label{B-8.5}\\
	&\leq M e^{-\gamma_0 t} \norm{z_0}_{\Bto} + C \norm{\calN_q z}_{L^p(0,\infty;\lso)} \label{B-8.6}\\
	\norm{z(t)}_{\Bto} &\leq M e^{-\gamma_0 t} \norm{z_0}_{\Bto} + C_1 \norm{z}^2_{\xipqs}. \label{B-8.7}
	\end{align}
	\noindent In going from (\ref{B-8.4}) to (\ref{B-8.5}) we have recalled the embedding $\ds \xipqs\big( \BA_{_{F,q}} \big) \hookrightarrow L^{\infty}\big(0,\infty;\Bto \big)$ from (\ref{B-6.8c}) with $T = \infty$. Next, in going from (\ref{B-8.5}) to (\ref{B-8.6}) we have used the maximal regularity property (\hyperref[B-6.6a]{6.6}). Finally, to go from (\ref{B-8.6}) to (\ref{B-8.7}) we have invoked estimate (\ref{B-7.22}).\\
	
	\noindent \underline{Step 2}: We shall next establish that
	\begin{equation}\label{B-8.8}
	\norm{z}_{\xipqs} \leq M \norm{z_0}_{\Bto} + K \norm{z}^2_{\xipqs}, \text{ hence } \norm{z}_{\xipqs} \big(1-K\norm{z}_{\xipqs} \big) \leq M \norm{z_0}_{\Bto}
	\end{equation}
	\noindent In fact, to this end, we take the $\xipqs$ estimate of equation (\ref{B-8.2}). We obtain 
	\begin{equation}
	\norm{z}_{\xipqs} \leq \norm{e^{\BA_{F,q}t}z_0}_{\xipqs} + \norm{\int_{0}^{t} e^{\BA_{F,q}(t-\tau)}\calN_q z(\tau) d \tau}_{\xipqs}
	\end{equation}
	\noindent from which then (\ref{B-8.8}) follows by invoking the maximal regularity property (\ref{B-6.29}), (\ref{B-6.30}) on $e^{\BA_{F,q}t}$ as well as the maximal regularity estimate \eqref{B-6.6} followed by use of (\ref{B-7.22}), as in going from (\ref{B-8.5}) to (\ref{B-8.7})
	\begin{align}
	\norm{\int_{0}^{t} e^{\BA_{F,q}(t-\tau)}\calN_qz(\tau) d \tau}_{\xipqs} &\leq \widetilde{m} \norm{\calN_q z}_{L^p(0,\infty;\lso)} \label{B-8.10}\\
	&\leq \widetilde{m} C \norm{z}^2_{\xipqs}. \label{B-8.11}
	\end{align}
	\noindent Thus (\ref{B-8.8}) is proved with $K = \widetilde{m}C$ where $C$ is the same constant occurring in (\ref{B-7.22}) or \eqref{B-7.25}, hence in (\ref{B-7.24}), (\ref{B-7.25}).\\
	
	\noindent \underline{Step 3}: The well-posedness Theorem \ref{B-Thm-7.1} says that
	\begin{equation}\label{B-8.12}
	\begin{Bmatrix}
	\text{ If } \norm{z_0}_{\Bto} \leq r_1 \\
	\text{for } r_1 \text{ sufficiently small as in \eqref{B-7.27}}
	\end{Bmatrix}
	\implies
	\begin{Bmatrix}
	\text{ The solution } z \text{ satisfies} \\
	\norm{z}_{\xipqs} \leq r
	\end{Bmatrix}
	\end{equation}
	where $r$ satisfies the constraint (\ref{B-7.25}) in terms of $r_1$ satisfying \eqref{B-7.27} and some constant $C$ that occurs for $K = \widetilde{m}C$ in (\ref{B-8.11}). We seek to guarantee that we can obtain
	\begin{equation}\label{B-8.13}
	\begin{cases}
	\norm{z}_{\xipqs} \leq r < \frac{1}{2K} = \frac{1}{2 \widetilde{m} C} \ \Big( < \frac{1}{2C}\Big)\\
	\ \\
	\text{hence } \frac{1}{2} < 1 - K \norm{z}_{\xipqs},
	\end{cases}
	\end{equation}
	\noindent where w.l.o.g. we can take the maximal regularity constant $\widetilde{m}$ in (\ref{B-7.30}) to satisfy $\widetilde{m} \geq 1$. Again, the constant $C$ arises from application of estimate (\ref{B-7.22}). This is indeed possible by choosing $r_1 > 0$ sufficiently small. In fact, as $r_1 \searrow 0$, (\ref{B-7.26}) shows that the interval $r_{min} \leq r \leq r_{max}$ of corresponding values of $r$ tends to the interval $\ds \bigg[ 0, \frac{1}{C} \bigg]$. Thus (\ref{B-8.3}) can be achieved as $r_{min} \searrow 0$: $\ds 0 < r_{min} < r < \frac{1}{2 \widetilde{m} C}$. Next, (\ref{B-8.3}) implies that (\ref{B-8.8}) holds true and yields then 
	\begin{equation}\label{B-8.14}
	\norm{z}_{\xipqs} \leq 2M \norm{z_0}_{\Bto} \leq 2Mr_1.
	\end{equation}
	\noindent Substituting (\ref{B-8.14}) in estimate (\ref{B-8.7}) then yields
	\begin{align}
	\norm{z(t)}_{\Bto} &\leq M e^{-\gamma_0 t} \norm{z_0}_{\Bto} + 4C_1M^2 \norm{z_0}^2_{\Bto}\\
	&= M \bigg[ e^{-\gamma_0 t} + 4 M C_1 \norm{z_0}_{\Bto} \bigg] \norm{z_0}_{\Bto} \label{B-8.16}\\
	\norm{z(t)}_{\Bto} &\leq M \big[ e^{-\gamma_0 t} + 4 M C_1 r_1 \big] \norm{z_0}_{\Bto} \label{B-8.17}
	\end{align}
	\noindent recalling the constant $r_1 > 0$ in (\ref{B-8.12}).\\
	
	\noindent \underline{Step 4}: Now take $T$ sufficiently large and $r_1 > 0$ sufficiently small such that 
	\begin{equation}\label{B-8.18}
	\beta \equiv M e^{-\gamma_0 T} + 4M^2C_1r_1 < 1
	\end{equation}
	\noindent Then (\ref{B-8.16}) implies by (\ref{B-8.18})
	\begin{subequations}\label{B-8.19}
		\begin{align}
		\norm{z(T)}_{\Bto} &\leq \beta \norm{z_0}_{\Bto} \text{ and hence } \label{B-8.19a}\\
		\norm{z(nT)}_{\Bto} &\leq \beta \norm{z((n-1)T)}_{\Bto} \leq \beta^n \norm{z_0}_{\Bto}. \label{B-8.19b}
		\end{align}	
	\end{subequations}
	\noindent Since $\beta < 1$, the semigroup property of the evolution implies that there are constants $\widetilde{M} \geq 1, \widetilde{\gamma} > 0$ such that
	\begin{equation}\label{B-8.20}
	\norm{z(t)}_{\Bto} \leq \widetilde{M} e^{-\widetilde{\gamma} t} \norm{z_0}_{\Bto}, \quad t \geq 0
	\end{equation}
	\noindent This proves Theorem \ref{B-Thm-8.1}.	
\end{proof}

\begin{rmk}\label{B-rmk-8.1}
	The above computations - (\ref{B-8.18}) through (\ref{B-8.20}) - can be used to support qualitatively the intuitive expectation that ``the larger the decay rate $\gamma_0$ in (\ref{B-5.13}) of the linearized feedback $w$-dynamics (\ref{B-5.3}), the larger the decay rate $\wti{\gamma}$ in (\ref{B-8.20}) of the nonlinear feedback $z$-dynamics (\ref{B-1.29}) = (\ref{B-7.1}); hence the larger the rate $\wti{\gamma}$ in (\ref{B-1.22}) of the original $y$-dynamics in feedback form as in (\ref{B-1.18})".\\
	
	\noindent The following considerations are somewhat qualitative. Let $S(t)$ denote the non-linear semigroup in the space $\ds \Bto$, with infinitesimal generator $\ds \big[ \BA_{_{F,q}} - \calN_q \big]$ describing the feedback $z$-dynamics (\ref{B-1.29}) in the abstract form (\ref{B-7.1}), as guaranteed by the well posedness Theorem \hyperref[B-Thm-B]{B.(i)} = Theorem \ref{B-Thm-7.1}. Thus, $\ds z(t;z_0) = S(t)z_0$ on $\ds \Bto$. By (\ref{B-8.18}), we can rewrite (\ref{B-8.19a}) as:
	\begin{equation}\label{B-8.21}
	\norm{S(T)}_{\calL \big(\Bto \big)} \leq \beta < 1.
	\end{equation}
	\noindent It follows from \cite[p 178]{Bal:1981} via the semigroup property that
	\begin{equation}\label{B-8.22}
	- \wti{\gamma} \ \ \text{is just below} \ \ \frac{\ln \beta}{T} < 0.
	\end{equation}
	\noindent Pick $r_1 > 0$ in (\ref{B-8.18}) so small that $4M^2 C_1 r_1$ is negligible, so that $\beta$ is just above $\ds M e^{- \gamma_0 T}$, so $\ln \beta$ is just above $\ds \big[ \ln M - \gamma_0 T \big]$, hence
	\begin{equation}\label{B-8.23}
	\frac{\ln \beta}{T} \text{ is just above } (-\gamma_0) + \frac{\ln M}{T}.
	\end{equation}
	\noindent Hence, by (\ref{B-8.22}), (\ref{B-8.23}),
	\begin{equation}\label{B-8.24}
	\wti{\gamma} \sim \gamma_0 - \frac{\ln M}{T}
	\end{equation}
	\noindent and the larger $\gamma_0$, the larger is $\wti{\gamma}$, as desired.
\end{rmk}

\section{Well-posedness of the pressure $\chi$ for the $z$-problem (\ref{B-1.29}) in feedback form, and of the pressure $\pi$ for the $y$-problem (\ref{B-1.18}) in feedback form.}\label{B-Sec-9}

\noindent \underline{The $z$-problem in feedback form:} We return to the translated $z$ problem (\ref{B-1.29}) = (\ref{B-7.1}), with $\L_e(z)$ given by (\ref{B-1.9})	
\begin{subequations}\label{B-9.1}
	\begin{align}
	z_t - \nu \Delta z + L_e(z) + (z \cdot \nabla) z + \nabla \chi &= m(\wti{G}z)\tau   &\text{ in } Q \label{B-9.1a}\\ 
	\div \ z &= 0   &\text{ in } Q \\
	z &= Fz &\text{ on } \Sigma\\
	z(0,x) &= y_0(x) - y_e(x) &\text{ on } \Omega
	\end{align}	
	\noindent with $Fz$ and $m(\wti{G}z)\tau$ given in the feedback form as in (\hyperref[B-5.6a]{5.6a-b})
	\begin{equation}\label{B-9.1e}
	m(\wti{G}z)\tau = m \Bigg( \sum_{k = 1}^{K}(P_Nz,q_k)_{_{W^u_N}}u_k \Bigg) \tau, \quad Fz = \sum_{k = 1}^K \big< P_Nz, p_k \big>_{\Gamma} f_k,
	\end{equation}
\end{subequations}
\noindent for which Theorem \hyperref[B-Thm-B]{B(i)} = Theorem \ref{B-Thm-7.1} provides a local well-posedness result in (\ref{B-7.6}), (\ref{B-7.7}) for the $z$ variable. We now complement such well-posedness for $z$ with a corresponding local well-posedness result for the pressure $\chi$.\\

\noindent Here we recall maximal regularity result of the Stokes operator in Appendix (\ref{B-A.17}) for problem (\hyperref[B-A.10a]{A.10a-b-c-d}) which accounts for inhomogeneous no-slip Dirichlet boundary conditions \cite{PS:2016}.	We present it for convenience
\begin{subequations}\label{B-9.2}
	\begin{align}
	\psi_t - \nu_o \Delta \psi + \nabla \pi &= F &\text{ in } (0, T] \times \Omega \equiv Q \label{B-9.2a}\\		
	div \ \psi &\equiv 0 &\text{ in } Q\\
	\begin{picture}(0,0)
	\put(-75,5){$\left\{\rule{0pt}{35pt}\right.$}\end{picture}
	\left. \psi \right \rvert_{\Sigma} &\equiv h_0 &\text{ in } (0, T] \times \Gamma \equiv \Sigma\\
	\left. \psi \right \rvert_{t = 0} &= \psi_0 &\text{ in } \Omega,
	\end{align}
\end{subequations}
Then there exists a unique solution $\varphi \in \xtpqs, \pi \in \ytpq$ to the dynamic Stokes problem (\ref{B-9.2}) or Appendix (\ref{B-A.17}), continuously on the data: there exist constants $C_0, C_1$ independent of $T, \Fs = P_q F, \varphi_0$ such that via Appendix (\ref{B-A.14}), $0 < T \leq \infty$:
\begin{equation}\label{B-9.3}
\begin{aligned}
C_0 \norm{\varphi}_{C \big([0,T]; \Bso \big)} &\leq \norm{\varphi}_{\xtpqs} +  \norm{\pi}_{\ytpq}\\ &\equiv \norm{\varphi'}_{L^p(0,T;\lso)} + \norm{A_q \varphi}_{L^p(0,T;\lso)} +  \norm{\pi}_{\ytpq}\\
&\leq C_1 \bigg \{ \norm{\Fs}_{L^p(0,T;\lso)}  + \norm{\varphi_0}_{\big( \lso, \calD(A_q)\big)_{1-\frac{1}{p},p}} + \norm{h_0}_{L^p(0, \infty; W^{1-\rfrac{1}{q},q}(\Gamma))} \bigg \}.
\end{aligned}
\end{equation}

\begin{thm}
	Consider the setting of Theorem \hyperref[B-Thm-A]{A} for problem (\ref{B-1.18}). Then the following well-posedness result for the pressure $\chi$ holds true, where we recall the spaces $\ds \yipq$ for $T = \infty$ and $\ds \widehat{W}^{1,q}(\Omega)$ in Appendix (\ref{B-A.12}), (\ref{B-A.13}) as well as the steady state pressure $\pi_e$ from Theorem \ref{B-Thm-1.1}:
	\begin{equation}\label{B-9.4}
	\norm{\chi}_{\yipq} \leq \wti{C} \norm{y_0 - y_e}_{\Bto} \Big\{ \norm{y_0 - y_e}_{\Bto} + 1  \Big\}.
	\end{equation}
\end{thm}

\begin{proof}
	We first apply the full maximal-regularity up to $T = \infty$ (\ref{B-9.3}) to the Stokes component of problem (\ref{B-9.1}) with $F_q = P_q \big( mG(z) - L_e(z) - (z \cdot \nabla)z \big)$ and $h_0 = Fz$ to obtain
	\begin{align}
	\norm{z}_{\xipqs} + \norm{\chi}_{\yipq} &\leq C \Big\{ \norm{ P_q [m(Gz) - (z \cdot \nabla) z - L_e(z)]}_{\lplqs} + \norm{z_0}_{\Bto}  \nonumber \\
	& \hspace{7cm} + \norm{Fz}_{L^p(0, \infty; W^{1-\rfrac{1}{q},q}(\Gamma))} \Big\} \nonumber\\
	&\leq C \Big\{ \norm{P_q [m(Gz)]}_{\lplqs} + \norm{P_q (z \cdot \nabla) z}_{\lplqs} \nonumber\\
	& +\norm{P_q L_e(z)}_{\lplqs} + \norm{z_0}_{\Bto} + \norm{Fz}_{L^p(0, \infty; W^{1-\rfrac{1}{q},q}(\Gamma))} \Big\}. \label{B-9.5}
	\end{align}
	
	\noindent But $P_q [m G(z)] = m G(z)$ as the vectors $u_k$ in the definition of $\wti{G}$ in (\ref{B-9.1e}) are $\ds u_k \in W^u_N \subset \lso $. Moreover $G \in \calL (\lso)$, we obtain
	\begin{subequations}
		\begin{align}\label{B-9.6a}
		\norm{P_q[m(\wti{G}z)]}_{\lplqs} &\leq C_1 \norm{z}_{\xipqs}, \\ 
		\norm{Fz}_{L^p(0, \infty; W^{1-\rfrac{1}{q},q}(\Gamma))} & \leq C'_1 \norm{z}_{\xipqs}
		\end{align} 
	\end{subequations}
	
	\noindent recalling the space $\ds \xipqs$ from Appendix (\ref{B-A.12}). Next, recalling (\ref{B-7.22}) for $\ds \calN_q z = P_q \big[ (z \cdot \nabla) z\big]$, see (\ref{B-2.23}), we obtain
	\begin{equation}\label{B-9.7}
	\norm{P_q (z \cdot \nabla) z}_{\lplqs} \leq C_2 \norm{z}^2_{\xipqs}.
	\end{equation} 
	\noindent The equilibrium solution $\{y_e,\pi_e\}$ is given by Theorem \ref{B-Thm-1.1} as satisfying 
	\begin{equation}\label{B-9.8}
	\norm{y_e}_{W^{2,q}(\Omega)} + \norm{\pi_e}_{\widehat{W}^{q,1}} \leq c \norm{f}_{L^q(\Omega)}, \quad 1 < q < \infty.
	\end{equation}
	\noindent We next estimate the term $\ds P_q L_e(z) = P_q [(y_e \cdot \nabla)z + (z \cdot \nabla) y_e]$ in (\ref{B-9.5})
	\begin{align}
	\norm{P_q L_e(z)}_{\lplqs} &= \norm{P_q (y_e . \nabla) z + P_q (z. \nabla) y_e }_{\lplqs}\\
	&\leq \norm{P_q (y_e . \nabla) z}_{\lplqs} + \norm{P_q (z. \nabla) y_e }_{\lplqs}\\
	&\leq \norm{y_e}_{L^q(\Omega)} \norm{\nabla z}_{\lplqs} + \norm{z}_{\lplqs} \norm{\nabla y_e}_{L^q(\Omega)}\\
	&\leq 2C_2 \norm{f}_{L^q(\Omega)} \norm{z}_{\lplqs} \label{B-9.12}\\
	&\leq C_3 \norm{z}_{\xipqs} \label{B-9.13}
	\end{align}
	\noindent with the constant $C_3$ depending on the $L^q(\Omega)$-norm of the datum $f$. Setting now $C_4 = C \cdot \{ C_1, C_2, C_3 \} $ and substituting (\ref{B-9.6a}), (\ref{B-9.7}), (\ref{B-9.13}) in (\ref{B-9.5}), we obtain
	\begin{equation}\label{B-9.14}
	\norm{z}_{\xipqs} + \norm{\chi}_{\yipq} \leq C_4 \Big\{ \norm{z}^2_{\xipqs} + 2\norm{z}_{\xipqs} + \norm{z_0}_{\Bto}  \Big\}
	\end{equation}
	\noindent Next we drop the term $\ds \norm{z}_{\xipqs}$ on the left hand side of (\ref{B-9.14}) and invoking Appendix (\ref{B-A.10}) to estimate $\ds \norm{z}_{\xipqs}$. Thus we obtain
	\begin{align}
	\norm{\chi}_{\yipq} &\leq C_5 \Big\{ \norm{z_0}^2_{\Bto} + 2\norm{z_0}_{\Bto} + \norm{z_0}_{\Bto}  \Big\}\\
	&\leq \wti{C} \norm{z_0}_{\Bto} \Big\{ \norm{z_0}_{\Bto} + 1  \Big\}, \quad \wti{C} = 3C_5 \label{B-9.16}
	\end{align}
	\noindent and (\ref{B-9.16}) proves (\ref{B-9.4}), as desired, recalling (\ref{B-2.26}). 
\end{proof}	

\noindent \underline{The $y$-problem in feedback form} We return to the original $y$-problem however in feedback form as in (\ref{B-1.18}), (\ref{B-1.23}), (\ref{B-1.24}) for which Theorem \hyperref[B-Thm-A]{A} in Secntion \ref{B-Sec-1.7} proves a local well-posedness result. We now complement such well-posedness result for $y$ with the corresponding local well-posedness result for the pressure $\pi$.

\begin{thm}\label{B-Thm-9.2}
	Consider the setting of Theorem \hyperref[B-Thm-A]{A} for the $y$-problem in (\ref{B-1.18}), (\ref{B-1.23}), (\ref{B-1.24}). Then, the following well-posedness result for the pressure $\pi$ holds true.
	\begin{align}
	\norm{\pi - \pi_e}_{\ytpq} &\leq \norm{\pi - \pi_e}_{\yipq} \leq \wti{C} \norm{y_0 - y_e}_{\Bto} \Big\{ \norm{y_0 - y_e}_{\Bto} + 1  \Big\} \label{B-9.17}\\
	&\leq \widehat{C} \Big\{ \norm{y_0}_{\Bto} + \norm{y_e}_{W^{2,q}(\Omega)} \Big\} \Big\{ \norm{y_0}_{\Bto} + \norm{y_e}_{W^{2,q}(\Omega)} + 1  \Big\} \label{B-9.18}\\
	&\leq \widehat{C} \Big\{ \norm{y_0}_{\Bto} + \norm{f}_{L^q(\Omega)} \Big\} \Big\{ \norm{y_0}_{\Bto} + \norm{f}_{L^q(\Omega)} + 1  \Big\} \label{B-9.19}
	\end{align}
	\begin{align}
	\norm{\pi}_{\ytpq} &\leq \widehat{C} \norm{y_0 - y_e}_{\Bto} \Big\{ \norm{y_0 - y_e}_{\Bto} + 1  \Big\} + cT^{\rfrac{1}{p}} \norm{\pi_e}_{\widehat{W}^{1,q}(\Omega)}, \ 0 < T < \infty\\
	&\leq \widehat{C} \Big\{ \norm{y_0}_{\Bto} + \norm{f}_{L^q(\Omega)} \Big\} \Big\{ \norm{y_0}_{\Bto} + \norm{f}_{L^q(\Omega)} + 1  \Big\} \nonumber\\
	&\hspace{8cm}+ cT^{\rfrac{1}{p}} \norm{f}_{L^q(\Omega)}, \ 0 < T < \infty
	\end{align}
\end{thm}

\begin{proof}
	We return to the estimate (\ref{B-9.4}) for $\chi$ and recall $\chi = \pi - \pi_e$ from (\ref{B-2.26}) to obtain (\ref{B-9.17}). We next estimate $y-y_e$ by 
	\begin{equation}
	\norm{y_0 - y_e}_{\Bto} \leq C  \big\{ \norm{y_0}_{\Bto} + \norm{y_e}_{W^{2,q}(\Omega)} \big\}.
	\end{equation}
	\noindent which substituted in (\ref{B-9.17}) yields (\ref{B-9.18}). In turn, (\ref{B-9.18}) leads to (\ref{B-9.19}) by means of (\ref{B-9.8}).
\end{proof}

\section{Results on the real space setting.}\label{B-Sec-10}

\noindent Here we shall complement the results of Theorems \ref{B-Thm-5.1} through \ref{B-Thm-9.2} by giving their version in the real space setting. We shall quote from \cite{BT:2004}. In the complexified setting $\ds \lso + i \lso$
we have that the complex unstable subspace $W^u_N$ is, recall \eqref{B-3.10}:
\begin{align}
W^u_N &= W^1_N + i W^2_N \\
&= \text{space of generalized eigenfunctions } \{\phi_j\}_{j=1}^N \text{ of the operator } \calA_q (=\calA_q^u) \nonumber \\
& \quad \text{ corresponding to its } N \text{ unstable eigenvalues in \eqref{B-1.12}.}
\end{align}
\noindent Set $\ds \phi_j = \phi_j^1 + i \phi_j^2$ with $\phi_j^1,\phi_j^2$ real. Then:
\begin{equation}
W^1_N = \text{Re} \ W^u_N = \text{span} \{ \phi_j^1 \}_{j = 1}^N; \quad W^2_N = \text{Im} \ W^u_N = \text{span} \{ \phi_j^2 \}_{j = 1}^N.
\end{equation}
\noindent The stabilizing vectors $p_k, q_k, u_k, \ k = 1, \dots, K$ are complex valued, with $u_k \in W^u_N \subset \lso$, and $p_k, q_k \in (W^u_N)^* \subset \lo{q'}$ as in (\ref{B-5.1}), (\ref{B-5.2}) while $\ds f_k \in \calF \subset W^{1 - \rfrac{1}{q}, q}(\Gamma)$.\\

\noindent The complex-valued uniformly stable linear $w$-system in (\ref{B-2.26}) with $K$ complex valued stabilizing vectors admits the following real-valued uniformly stable counterpart
\begin{multline}
\frac{dw}{dt} = \calA_q w - \calA_q D \Bigg( \sum_{k=1}^{K} \text{ Re}\big( w_N(t), p_k \big)_{_{W^u_N}} \text{ Re}f_k - \sum_{k=1}^{K} \text{ Im}\big( w_N(t), p_k \big)_{_{W^u_N}} \text{ Im}f_k \Bigg) \\+ P_q \Bigg ( m \Bigg( \sum_{k=1}^{K} \text{Re} (w_N(t),q_k)_{_{W^u_N}} \text{ Re } u_k - \sum_{k=1}^{K} \text{Im }(w_N(t),q_k)_{_{W^u_N}} \text{ Im } u_k\Bigg) \cdot \tau \Bigg)
\end{multline}
\noindent with $2K \leq N$ real stabilizing vectors, see \cite[Eq 3.52a, p 1472]{BT:2004}.If $K = \text{sup } \{ \ell_i, i = 1, \dots,M \}$ is achieved for a real eigenvalue $\lambda_i$ (respectively, a complex eigenvalue $\lambda_i$), then the \textit{effective} number of stabilizing controllers is $K \leq N$, as the generalized functions are then real, since $y_e$ is real; respectively, $2K \leq N$, for, in this case, the complex conjugate eigenvalue $\bar{\lambda}_j$ contributes an equal number of components in terms of generalized eigenfunctions $\ds \phi_{\bar{\lambda}_j} = \bar{\phi}_{\lambda_j}$. In all cases, the actual (\textit{effective}) upper bound $2K$ is $2K \leq N$. For instance, if all unstable eigenvalues were real and simple then $K=1$, and only one stabilizing controller is actually needed.\\

\noindent Similarly, the complex-valued locally (near $y_e$) uniformly stable nonlinear $y$-system (\ref{B-1.18}) with $K$ complex-valued stabilizing vectors admits the following real-valued locally uniformly stable counterpart
\begin{multline}
\frac{dy}{dt} - \nu A_q y + \calN_q y = - \calA_q D \Bigg( \sum_{k=1}^{K} \text{ Re}\big( w_N(t), p_k \big)_{_{W^u_N}} \text{ Re}f_k - \sum_{k=1}^{K} \text{ Im}\big( w_N(t), p_k \big)_{_{W^u_N}} \text{ Im}f_k \Bigg) \\+ P_q \Bigg ( m \Bigg( \sum_{k=1}^{K} \text{Re }(y - y_e,q_k)_{\omega} \text{ Re } u_k - \sum_{k=1}^{K} \text{Im } (y - y_e,q_k)_{\omega} \text{ Im } u_k\Bigg) \cdot \tau \Bigg)
\end{multline} 
\noindent with $2K \leq N$ real stabilizing vectors, see \cite[p 43]{BLT1:2006}.

	\begin{appendices}
	\renewcommand{\appendixname}{Appendix}	
	\section{Some auxiliary results for the Stokes and Oseen operators: analytic semigroup generation, maximal regularity, domains of fractional powers.}\label{B-app-A}
	\setcounter{equation}{0}
	\renewcommand{\theequation}{{\rm A}.\arabic{equation}}
	\renewcommand{\thetheorem}{{\bf A}.\arabic{theorem}}
	
	In this subsection we collect mostly known results to be used in the sequel.
	
	\begin{enumerate}[(a)]
		\item \textbf{The Stokes and Oseen operators generate strongly continuous analytic semigroups on $\lso$, $1 < q < \infty$}.
		\begin{thm}\label{B-Thm-A.1}
			Let $d \geq 2, 1 < q < \infty$ and let $\Omega$ be a bounded domain in $\mathbb{R}^d$ of class $C^3$. Then
			\begin{enumerate}[(i)]
				\item the Stokes operator $-A_q = P_q \Delta$ in (\ref{B-2.20}), repeated here as 
				\begin{equation}\label{B-A.1}
				-A_q \psi  = P_q \Delta \psi , \quad
				\psi \in \mathcal{D}(A_q) = W^{2,q}(\Omega) \cap W^{1,q}_0(\Omega) \cap \lso
				\end{equation}
				generates a s.c analytic semigroup $e^{-A_qt}$ on $\lso$. See \cite{Gi:1981} and the review paper \cite[Theorem 2.8.5 p 17]{HS:2016}.			
				\item The Oseen operator $\calA_q$ in (\ref{B-2.22}) \label{B-Thm-3.1(ii)}
				\begin{equation}\label{B-A.2}
				\calA_q  = - (\nu_o A_q + A_{o,q}), \quad \calD(\calA_q) = \calD(A_q) \subset \lso
				\end{equation}
				generates a s.c analytic semigroup $e^{\calA_qt}$ on $\lso$. This follows as $A_{o,q}$ is relatively bounded with respect to $A^{\rfrac{1}{2}}_q$, to be formally defined in (\ref{B-A.6}): thus a standard theorem on perturbation of an analytic semigroup generator applies \cite[Corollary 2.4, p 81]{P:1983}.
				
				\item  \begin{subequations}\label{B-A.3} 
					\begin{align}
					0 \in \rho (A_q) &= \text{ the resolvent set of the Stokes operator } A_q\\
					\begin{picture}(0,0)
					\put(-40,10){ $\left\{\rule{0pt}{18pt}\right.$}\end{picture}
					A_q^{-1} &: \lso \longrightarrow \lso \text{ is compact}.
					\end{align}
				\end{subequations}		
			
			\item  The s.c. analytic Stokes semigroup $e^{-A_qt}$ is uniformly stable on $\lso$: there exist constants $M \geq 1, \delta > 0$ (possibly depending on $q$) such that 
			\begin{equation}\label{B-A.4}
			\norm{e^{-A_qt}}_{\calL(\lso)} \leq M e^{-\delta t}, \ t > 0.
			\end{equation}
		\end{enumerate}
		\end{thm}
		\item \textbf{Domains of fractional powers, $\calD(A_q^{\alpha}), 0 < \alpha < 1$ of the Stokes operator $A_q$ on $\lso, 1 < q < \infty$}, 
		\begin{thm}\label{B-Thm-A.2}
			For the domains of fractional powers $\calD(A_q^{\alpha}), 0 < \alpha < 1$, of the Stokes operator $A_q$ in (\ref{B-2.20}) = (\ref{B-A.1}), the following complex interpolation relation holds true \cite[Theorem 2.8.5, p 18]{HS:2016}, \cite{Gi:1985}
			\begin{equation}\label{B-A.5}
			[ \calD(A_q), \lso ]_{1-\alpha} = \calD(A_q^{\alpha}), \ 0 < \alpha < 1, \  1 < q < \infty;
			\end{equation}
			in particular, see (\ref{B-2.21})
			\begin{equation}\label{B-A.6}
			[ \calD(A_q), \lso ]_{\frac{1}{2}} = \calD(A_q^{\rfrac{1}{2}}) \equiv W_0^{1,q}(\Omega) \cap \lso.
			\end{equation}
			Thus, on the space $\calD(A_q^{\rfrac{1}{2}})$, the norms
			\begin{equation}\label{B-A.7}
			\norm{\nabla \ \cdot \ }_{L^q(\Omega)} \text{ and } \norm{ \ }_{L^q(\Omega)}
			\end{equation}
			are equivalent via the Poincar\'{e} inequality.
		\end{thm}
		
		\item \textbf{The Stokes operator $-A_q$ and the Oseen operator $\calA_q, 1 < q < \infty$ generate s.c. analytic semigroups on the Besov space.}
		\begin{subequations}\label{B-A.8}
			\begin{align}
			\Big( \lso,\mathcal{D}(A_q) \Big)_{1-\frac{1}{p},p} &= \Big\{ g \in \Bso : \text{ div } g = 0, \ g|_{\Gamma} = 0 \Big\} \quad \text{if } \frac{1}{q} < 2 - \frac{2}{p} < 2; \label{B-A.8a}\\
			\Big( \lso,\mathcal{D}(A_q) \Big)_{1-\frac{1}{p},p} &= \Big\{ g \in \Bso : \text{ div } g = 0, \ g\cdot \nu|_{\Gamma} = 0 \Big\} \equiv \Bt(\Omega) \label{B-A.8b}\\
			&\hspace{4cm} \text{ if } 0 < 2 - \frac{2}{p} < \frac{1}{q}, \text{ or } 1 < p < \frac{2q}{2q-1}.\nonumber
			\end{align}	
		\end{subequations}
		In fact, Theorem \hyperref[B-Thm-A.1]{A.1(i)} states that the Stokes operator $-A_q$ generates a s.c analytic semigroup on the space $\lso, \ 1 < q < \infty$, hence on the space $\calD(A_q)$ in (\ref{B-A.1}), with norm $\ds \norm{ \ \cdot \ }_{\calD(A_q)} = \norm{ A_q \ \cdot \ }_{\lso}$ as $0 \in \rho(A_q)$.  Then, one obtains that the Stokes operator $-A_q$ generates a s.c. analytic semigroup on the real interpolation spaces in (\ref{B-A.8}). Next, the Oseen operator $\calA = -(\nu_o A_q + A_{o,q})$ likewise generates a s.c. analytic semigroup $\ds e^{\calA_q t}$ on $\ds \lso$ by Theorem \hyperref[B-Thm-A.1]{A.1(ii)}. Moreover $\calA_q$ generates a s.c. analytic semigroup on $\ds \calD(\calA_q) = \calD(A_q)$ (equivalent norms). Hence $\calA_q$ generates a s.c. analytic semigroup on the real interpolation spaces (\ref{B-A.11}). Here below, however, we shall formally state the result only in the case $\ds 2-\rfrac{2}{p} < \rfrac{1}{q}$. i.e. $\ds  1 < p < \rfrac{2q}{2q-1}$, in the space $\ds \Bto$, as this does not contain B.C. The objective of the present paper is precisely to obtain stabilization results on spaces that do not recognize B.C.
		
		\begin{thm}\label{B-Thm-A.3}
			Let $1 < q < \infty, 1 < p < \rfrac{2q}{2q-1}$
			\begin{enumerate}[(i)]
				\item The Stokes operator $-A_q$ in (\ref{B-A.1}) generates a s.c analytic semigroup $e^{-A_qt}$ on the space $\Bt(\Omega)$ defined in \eqref{B-1.15b} which moreover is uniformly stable, as in (\ref{B-A.4}),
				\begin{equation}\label{B-A.9}
				\norm{e^{-A_qt}}_{\calL \big(\Bt(\Omega)\big)} \leq M e^{-\delta t}, \quad t > 0.
				\end{equation}
				\item The Oseen operator $\calA_q$ in (\ref{B-A.2}) generates a s.c. analytic semigroup $e^{\calA_qt}$ on the space $\Bt(\Omega)$ in \eqref{B-1.15b}.
			\end{enumerate}
		\end{thm}
		\item \textbf{Space of maximal $L^p$ regularity on $\lso$ of the Stokes operator $-A_q, \ 1 < p < \infty, \ 1 < q < \infty $ up to $T = \infty$.} We shall use the notation of \cite{Dore:2000} and write $\ds -A_q \in MReg (\lplqs)$ .	We return to the dynamic Stokes problem in $\{\varphi(t,x), \pi(t,x) \}$\label{B-A.d}
		\begin{subequations}\label{B-A.10}
			\begin{align}
			\varphi_t - \Delta \varphi + \nabla \pi &= F &\text{ in } (0, T] \times \Omega \equiv Q \label{B-A.10a}\\		
			div \ \varphi &\equiv 0 &\text{ in } Q\\
			\begin{picture}(0,0)
			\put(-80,5){ $\left\{\rule{0pt}{35pt}\right.$}\end{picture}
			\left. \varphi \right \rvert_{\Sigma} &\equiv h_0 &\text{ in } (0, T] \times \Gamma \equiv \Sigma\\
			\left. \varphi \right \rvert_{t = 0} &= \varphi_0 &\text{ in } \Omega,
			\end{align}
		\end{subequations}
		
		rewritten in abstract form, after applying the Helmholtz projection $P_q$ to (\ref{B-A.10a}) and recalling $A_q$ in (\ref{B-A.1}) as 
		\begin{equation}\label{B-A.11}
		\varphi' + A_q \varphi = \Fs \equiv P_q F, \quad \varphi_0 \in \lqaq,
		\end{equation}
		
		recall (\ref{B-A.8}). Next, we introduce the space of maximal regularity for $\{\varphi, \varphi'\}$, i.e. for $-A_q$, as \cite[p 2; Theorem 2.8.5.iii, p 17]{HS:2016}, \cite[p 1404-5]{GGH:2012}, with $T$ up to $\infty$:
		\begin{equation}\label{B-A.12}
		X^T_{p,q, \sigma} (A_q) = L^p(0,T;\calD(A_q)) \cap W^{1,p}(0,T;\lso)
		\end{equation}
		(recall (\ref{B-2.20}) for $\calD(A_q)$) and the corresponding space for the pressure as 
		\begin{equation}\label{B-A.13}
		Y^T_{p,q} = L^p(0,T;\widehat{W}^{1,q}(\Omega)), \quad \widehat{W}^{1,q}(\Omega) = W^{1,q}(\Omega) / \mathbb{R}. 
		\end{equation}
		The following embedding, also called trace theorem, holds true \cite[Theorem 4.10.2, p 180, BUC for $T=\infty$]{HA:2000}.
		\begin{equation}\label{B-A.14}
		\xtpqs \subset \xtpq \equiv L^p(0,T; W^{2,q}(\Omega)) \cap W^{1,p}(0,T; L^q(\Omega)) \hookrightarrow C \Big([0,T]; \Bso \Big).
		\end{equation}
		For a function $g$ such that $div \ g \equiv 0, \ g|_{\Gamma} = 0$ we have $g \in \xtpq \iff g \in \xtpqs$, by (\ref{B-1.4}).\\
		The solution of Eq(\ref{B-A.11}) is 
		\begin{equation}\label{B-A.15}
		\varphi(t) = e^{-A_qt} \varphi_0 + \int_{0}^{t} e^{-A_q(t-\tau)} \Fs(\tau) d \tau.
		\end{equation}
		The following is the celebrated result on maximal regularity on $\lso$ of the Stokes problem due originally to Solonnikov \cite{VAS:1977} reported in \cite[Theorem 2.8.5(iii) for $\varphi_0 = 0$ and Theorem 2.10.1 p24]{HS:2016}, \cite{S:2006}, \cite[Proposition 4.1 , p 1405]{GGH:2012}, \cite{PS:2016}. See also \cite[Theorem 3.1 p 31 for $p = q =2$ case]{BLT1:2006}. See also \cite{VAS:1968}, \cite{VAS:1981}, \cite{VAS:1996},  \cite{VAS:2001}. 
		\begin{thm}\label{B-Thm-A.4} 
			Let $1 < p,q < \infty, T \leq \infty$. With reference to problem (\ref{B-A.10}), (\ref{B-A.11}), assume
			\begin{equation}\label{B-A.16}
			\Fs \in L^p(0,T;\lso), \ \varphi_0 \in \Big( \lso, \calD(A_q)\Big)_{1-\frac{1}{p},p}.
			\end{equation}
			Then there exists a unique solution $\varphi \in \xtpqs$ continuously on the data: there exist constants $C_0, C_1$ independent of $T, \Fs, \varphi_0$ such that via (\ref{B-A.14})
			\begin{equation}\label{B-A.17}
			\begin{aligned}
			C_0 \norm{\varphi}_{C \big([0,T]; \Bso \big)} &\leq \norm{\varphi}_{\xtpqs} + \norm{\pi}_{\ytpq}\\
			&= \norm{\varphi'}_{L^p(0,T;\lso)} + \norm{A_q \varphi}_{L^p(0,T;\lso)} + \norm{\pi}_{\ytpq}\\
			&\leq C_1 \bigg \{ \norm{\Fs}_{L^p(0,T;\lso)}  + \norm{\varphi_0}_{\big( \lso, \calD(A_q)\big)_{1-\frac{1}{p},p}} + \norm{h_0}_{L^p(0,\infty; W^{1-\rfrac{1}{q},q}(\Gamma))} \bigg \}.
			\end{aligned}
			\end{equation}
			In particular,
			\begin{enumerate}[(i)]
				\item With reference to the variation of parameters formula (\ref{B-A.15}) of problem (\ref{B-A.11}) arising from the Stokes problem (\ref{B-A.10}), we have recalling (\ref{B-A.12}): the map
				\begin{align}
				\Fs &\longrightarrow \int_{0}^{t} e^{-A_q(t-\tau)}\Fs(\tau) d\tau \ : \text{continuous} \label{B-A.18}\\
				L^p(0,T;\lso) &\longrightarrow \xtpqs(A_q) \equiv L^p(0,T; \calD(A_q)) \cap W^{1,p}(0,T; \lso). \label{B-A.19}				
				\end{align}			
				\item The s.c. analytic semigroup $e^{-A_q t}$ generated by the Stokes operator $-A_q$ (see (\ref{B-A.1})) on the space $\ds \Big( \lso, \calD(A_q)\Big)_{1-\frac{1}{p},p}$ satisfies
				\begin{subequations}\label{B-A.20}
					\begin{multline}
					e^{-A_q t}: \ \text{continuous} \quad \Big( \lso, \calD(A_q)\Big)_{1-\frac{1}{p},p} \longrightarrow \\ \xtpqs(A_q) \equiv L^p(0,T; \calD(A_q)) \cap W^{1,p}(0,T; \lso). \label{B-A.20a}
					\end{multline}
					In particular via (\ref{B-A.8b}), for future use, for $1 < q < \infty, 1 < p < \frac{2q}{2q - 1}$, the s.c. analytic semigroup $\ds e^{-A_q t}$ on the space $\ds \Bto$, satisfies
					\begin{equation}
					e^{-A_q t}: \ \text{continuous} \quad \Bto \longrightarrow \xtpqs. \label{B-A.20b}
					\end{equation}
				\end{subequations}				 
				\item Moreover, setting $\nabla \pi = (Id - P_q)(\Delta + F)$, it follows that $\{ \varphi, \pi \} \in \xtpqs \times \ytpq$, see (\ref{B-A.13}), solves problem (\ref{B-A.10}) and there is a constant $C>0$ independent of $T,\Fs,\phi_0$ s.t.
				\begin{subequations}
					\begin{equation}\label{B-A.21a}
					\norm{\varphi}_{\xtpqs} + \norm{\pi}_{\ytpq} \leq C \bigg \{ \norm{\Fs}_{L^p(0,T;\lso)} + \norm{\varphi_0}_{\lqaq} \bigg \}
					\end{equation}
					while, for future use, for $1 < q < \infty, 1 < p < \frac{2q}{2q - 1}$, then (\ref{B-A.21a}) specializes to
					\begin{equation}\label{B-A.21b}
					\norm{\varphi}_{\xtpqs} + \norm{\pi}_{\ytpq} \leq C \bigg \{ \norm{\Fs}_{L^p(0,T;\lso)} + \norm{\varphi_0}_{\Bto} \bigg \}.
					\end{equation}
				\end{subequations}				
			\end{enumerate}		
		\end{thm}
		
		\item \textbf{Maximal $L^p$ regularity on $\lso$ of the Oseen operator $\ds \calA_q, \ 1 < p < \infty, \ 1 < q < \infty$: $\ds \calA_q \in MReg (L^p(0,T;\lso))$, $T$ finite arbitrary.} We next transfer the maximal regularity of the Stokes operator $(-A_q)$ on $\lso$-asserted in Theorem \ref{B-Thm-A.4} into the maximal regularity of the Oseen operator $\calA_q = -\nu_o A_q - A_{o,q}$ exactly on the same space $\xtpqs$ defined in (\ref{B-A.12}), except now only on $T < \infty$.\\
		
		\noindent Thus, consider the dynamic Oseen problem in $\{ \psi(t,x), \pi(t,x) \}$ with equilibrium solution $y_e$, see Theorem \ref{B-Thm-1.1} on (\ref{B-1.2})  :		
		\begin{subequations}\label{B-A.22}
			\begin{align}
			\psi_t - \nu_o \Delta \psi + L_e(\psi) + \nabla \pi &= F &\text{ in } (0, T] \times \Omega \equiv Q \label{B-A.22a}\\		
			div \ \psi &\equiv 0 &\text{ in } Q\\
			\begin{picture}(0,0)
			\put(-120,5){$\left\{\rule{0pt}{35pt}\right.$}\end{picture}
			\left. \psi \right \rvert_{\Sigma} &\equiv 0 &\text{ in } (0, T] \times \Gamma \equiv \Sigma\\
			\left. \psi \right \rvert_{t = 0} &= \psi_0 &\text{ in } \Omega,
			\end{align}
		\end{subequations}
		\begin{equation}
		L_e(\psi) = (y_e . \nabla) \psi + (\psi. \nabla) y_e \hspace{6cm} \label{B-A.23}
		\end{equation}
		rewritten in abstract form, after applying the Helmholtz projector $P_q$ to (\ref{B-A.22a}) and recalling $\calA_q$ in (\ref{B-A.2})
		\begin{equation}\label{B-A.24}
		\psi_t = \calA_q \psi + P_q F = -\nu_o A_q \psi - A_{o,q} \psi + \Fs, \quad \psi_0 \in \big( \lso, \calD(A_q)\big)_{1-\frac{1}{p},p}
		\end{equation}
		whose solution is 
		\begin{equation}\label{B-A.25}
		\psi(t) = e^{\calA_qt} \psi_0 + \int_{0}^{t} e^{\calA_q(t-\tau)} \Fs(\tau) d \tau.
		\end{equation}
		\begin{equation}\label{B-A.26}
		\psi(t) = e^{- \nu_o A_qt} \psi_0 + \int_{0}^{t} e^{-\nu_o A_q(t-\tau)} \Fs(\tau) d \tau - \int_{0}^{t} e^{- \nu_o A_q(t-\tau)} A_{o,q} \psi(\tau) d \tau.
		\end{equation}
		
		\begin{thm}\label{B-Thm-A.5}
			Let $1 < p,q < \infty$, $0 < T < \infty$. Assume (as in (\ref{B-A.16})) 
			\begin{equation}\label{B-A.27}
			\Fs \in L^p \big( 0, T; L^q_{\sigma} (\Omega) \big), \quad \psi_0 \in \lqaq
			\end{equation}
			where $\calD(A_q) = \calD(\calA_q)$, see (\ref{B-A.2}). Then there exists a unique solution $\psi \in \xtpqs$ of the dynamic Oseen problem (\ref{B-A.22}), continuously on the data: that is, there exist constants $C_{0T}, C_{1T}$ independent of $\Fs, \psi_0$ such that (recall (\ref{B-A.14})):
			\begin{align}
			C_{0T} \norm{\psi}_{C \big([0,T]; \Bso \big)} &\leq \norm{\psi}_{\xtpqs} + \norm{\pi}_{\ytpq} \nonumber\\ 
			&\equiv \norm{\psi'}_{L^p(0,T;L^q(\Omega))} + \norm{A_q \psi}_{L^p(0,T;L^q(\Omega))} +  \norm{\pi}_{\ytpq}\\
			&\leq C_{1T} \bigg \{ \norm{\Fs}_{L^p(0,T;\lso)}  + \norm{\psi_0}_{\lqaq} \bigg \}
			\end{align}
			Equivalently, for $1 < p, q < \infty$
			\begin{enumerate}[i.]
				\item The map
				\begin{equation}
				\begin{aligned}
				\Fs \longrightarrow \int_{0}^{t} e^{\calA_q(t-\tau)}\Fs(\tau) d\tau \ : \text{continuous}&\\
				L^p(0,T;\lso) &\longrightarrow L^p \big(0,T;\calD(\calA_q) = \calD(A_q) \big)\label{B-A.30}
				\end{aligned}			
				\end{equation}
				where then automatically, see (\ref{B-A.24}) 
				\begin{equation}
				L^p(0,T;\lso) \longrightarrow W^{1,p}(0,T;\lso) \label{B-A.31}
				\end{equation}
				and ultimately
				\begin{subequations}\label{B-A.32}
					\begin{equation}\label{B-A.32a}
					L^p(0,T;\lso) \longrightarrow \xtpqs(A_q) \equiv L^p \big(0,T;\calD(A_q) \big) \cap W^{1,p}(0,T;\lso). 
					\end{equation}
					\noindent Thus, 
					\begin{equation}\label{B-A.32b}
					\calA_q \in MReg (L^p(0,T;\lso)), \ 1 < T < \infty
					\end{equation}
					and the operator $\ds \calA_q$ has maximal $L^p$ regularity on $\lso$.
				\end{subequations}
				
				\item The s.c. analytic semigroup $e^{\calA_q t}$ generated by the Oseen operator $\calA_q$ (see (\ref{B-A.2})) on the space $\ds \lqaq $ satisfies for $1 < p, q < \infty$
				\begin{equation}
				e^{\calA_q t}: \ \text{continuous} \quad \lqaq \longrightarrow L^p \big(0,T;\calD(\calA_q) = \calD(A_q)  \big) \label{B-A.33}
				\end{equation}
				and hence automatically
				\begin{equation}
				e^{ \calA_q t}: \ \text{continuous} \quad \lqaq \longrightarrow \xtpqs. \label{B-A.34}
				\end{equation}
				In particular, for future use, for $1 < q < \infty, 1 < p < \frac{2q}{2q - 1}$, we have that the s.c. analytic semigroup $\ds e^{\calA_q t}$ on the space $\ds \Bto$, satisfies 
				\begin{equation}
				e^{\calA_q t}: \ \text{continuous} \quad \Bto \longrightarrow L^p \big(0,T;\calD(\calA_q) = \calD(A_q)  \big), \ T < \infty, \label{B-A.35}
				\end{equation}
				and hence automatically
				\begin{equation}
				e^{ \calA_q t}: \ \text{continuous} \quad \Bto \longrightarrow \xtpqs(A_q), \ T < \infty. \label{B-A.36}
				\end{equation}
				\item An estimate such as the one in (\ref{B-A.21a}) applies to the pressure $\pi$, where now $\ds \nabla \pi (Id - P_q)(\Delta - L_e + F)$. 
			\end{enumerate}
		\end{thm}
	\end{enumerate}
	\noindent A proof of Theorem \ref{B-Thm-A.5} is given in \cite{LPT.1}.

\section{The eigenvectors $\ds \varphi^*_{ij} \in W^{2,q'}(\Omega) \cap W^{1,q'}_0(\Omega) \cap \lo{q'}$ of $\ds \calA^* (=\calA_q^*)$ in $\ds L^{q'}(\Omega)$ may be viewed also as $\ds \varphi^*_{ij} \in W^{3,q}(\Omega)$, so that $\ds \frac{\partial \varphi^*_{ij}}{\partial \nu} \bigg|_{\Gamma} \in W^{2-\rfrac{1}{q},q}(\Gamma), \ q \geq 2$.}\label{B-app-B}
	\setcounter{equation}{0}
	\renewcommand{\theequation}{{\rm B}.\arabic{equation}}
	\renewcommand{\thetheorem}{{\bf B}.\arabic{theorem}}
	
	\noindent The eigenvectors $\ds \varphi^*_{ij}$ of $\ds \calA^* (=\calA_q^*)$, defined in (\ref{B-4.8a}), are in $\ds \calD((\calA_q^*)^n)$ for any $n$, so the are arbitrarily smooth in $\lo{q'}$, say $\ds \varphi^*_{ij} \in W^{s,q'}(\Omega)$, with $s$ as large as we please. We seek to view $\ds \varphi^*_{ij}$ in an $\ds L^q(\Omega)$-based space. To this end, we recall a Sobolev embedding theorem.
	
	\begin{thm}\label{B-Thm-B.1}\cite[p328]{HT:1980}, For a more restricted version \cite{A:1975}
		Let $\Omega$ be an arbitrary bounded domain, dim $\Omega = d$. Let $0 \leq t \leq s < \infty$ and $ \infty > q \geq \wti{q} > 1$. Then, the following embedding holds true:
		\begin{equation}
			W^{s,\wti{q}}(\Omega) \subset W^{t, q}(\Omega), \ s - \frac{d}{\wti{q}} \geq t - \frac{d}{q} \quad \qedsymbol \label{B-B.1}
		\end{equation}
	\end{thm}
	
	\begin{clr}\label{B-Clr-B.2}
		With $\ds 2 \leq q < \infty, \ \rfrac{1}{q} + \rfrac{1}{q'} = 1$, so that $\ds 1 < q' \leq 2 \leq q, \ 0 \leq r$, we have
		\begin{enumerate}[(i)]
			\item \begin{equation}
				\varphi^*_{ij} \in W^{r+m, q'}(\Omega) \subset W^{r,q}(\Omega), \quad m \geq d \left( \frac{1}{q} + \frac{1}{q'} \right) = \left\{\begin{array}{ll}
				0, & q' = q = 2 \\[2mm]
				d, & q' = 1, q = \infty
				\end{array}\right.  \label{B-B.2}
			\end{equation}
			\item \begin{equation}
				\frac{\partial \varphi^*_{ij}}{\partial \nu} \bigg|_{\Gamma} \in W^{r-1-\rfrac{1}{q},q}(\Gamma), \quad r > 1 + \frac{1}{q} \hspace{6.5cm} \label{B-B.3}
			\end{equation}
			\item With reference to the sub-space $\calF$ based on $\Gamma$, as defined in (\ref{B-1.25}), we have
			\begin{equation}
				\calF \equiv \mbox{span}\left\{\frac{\partial}{\partial\nu} \varphi^*_{ij}, \ i = 1,\ldots,M; \ j = 1,\ldots,\ell_i\right\} \subset W^{r-1-\rfrac{1}{q},q}(\Gamma), \ r > 1 + \frac{1}{q} \label{B-B.4}
			\end{equation}
			\noindent In particular, for our purposes, if will suffice to take $r=3$ in (\ref{B-B.2}), so that (\ref{B-B.2})-(\ref{B-B.4}) become 
			\begin{equation}
				\varphi^*_{ij} \in W^{3,q}(\Omega), \quad \frac{\partial \varphi^*_{ij}}{\partial \nu} \bigg|_{\Gamma} \in W^{2-\rfrac{1}{q},q}(\Gamma), \quad \calF \subset W^{2-\rfrac{1}{q},q}(\Gamma). \label{B-B.5}
			\end{equation}
			\item Thus, with reference to the boundary vector $v = v_N$ introduced in (\ref{B-5.1}) = (\ref{B-6.10}), we have 
			\begin{equation}
				v = \sum_{k=1}^{K} \nu_k(t) f_k \in W^{2-\rfrac{1}{q},q}(\Gamma), \ f_k \in \calF, \ f_k \cdot \nu |_{\Gamma} = 0, \ v \cdot \nu |_{\Gamma} = 0 \label{B-B.6}
			\end{equation}
			\item Recalling the Dirichlet map $D$ introduced to describe the solution of problem (\ref{B-2.1}), we have 
			\begin{equation}
				Dv \in W^{2,q}(\Omega), \quad 2 \leq q < \infty \hspace{5.3cm} \label{B-B.7} 
			\end{equation}
		\end{enumerate}
	\end{clr}
\begin{proof}
	\noindent (i) Apply Theorem \ref{B-Thm-B.1} with $\ds s = r + m \geq t = r, \ \wti{q} = q', \ \rfrac{1}{q} + \rfrac{1}{q'} = 1, \ q \geq 2$, so that $\ds q' = \wti{q} \leq q$, to verify that the required condition (\ref{B-B.1})
	\begin{equation}
		s - \frac{d}{\wti{q}} = r + m - \frac{d}{q'} \geq t - \frac{d}{q} = r - \frac{d}{q}, \ \text{or} \ m \geq d \left( \frac{1}{q'} - \frac{1}{q} \right) \geq 0 \label{B-B.8}
	\end{equation}
	\noindent can always be satisfied by taking $m \geq 0$ suitable as in (\ref{B-B.8}). This is possible, since $\ds \varphi^*_{ij}$ is arbitrarily smooth.\\
	
	\noindent (ii) Then (\ref{B-B.3}) follows by the usual trace theory \cite{A:1975}.\\
	
	\noindent Then, (iii)-(v) readily follow, as $D$ improves regularity by $\rfrac{1}{q}$ from the boundary to the interior.\\
\end{proof}	
	\noindent Next, we return to the operator $\ds F: L^q(\Omega) \subset \lso \longrightarrow L^q(\Gamma)$ in \eqref{B-5.6}. Its adjoint $F^*$ is
	
	\begin{equation}\label{B-B.9}
	F^*g = \sum_{k = 1}^{K} (f_k, g)_{_{\Gamma}} p_k \in (W^u_N)^* \subset \lo{q'}, \quad g \in L^q(\Gamma)
	\end{equation}
	\noindent where we have seen in (\ref{B-3.41}) that $\ds D: L^q(\Gamma) \supset U_q \longrightarrow W^{\rfrac{1}{q},q}(\Omega) \cap \lso \subset \calD \Big( A^{\rfrac{1}{2q} - \varepsilon}_q \Big)$
	\begin{equation}\label{B-B.10}
	F^*D^*h = \sum_{k = 1}^{K} (f_k, D^*h)_{_{\Gamma}} p_k = \sum_{k = 1}^{K} \big< Df_k, h \big>_{_{W^u_N}} p_k \in (W^u_N)^*
	\end{equation}
	\noindent where we have conservatively: $\ds f_k \in L^{q'}(\Gamma), \ f_k \cdot \nu = 0$ on $\Gamma$, thus by (\ref{B-2.13}), $\ds Df_k \in W^{\rfrac{1}{q'},q'}(\Omega) = W^{\rfrac{1}{q'},q'}_0(\Omega)$ by (\ref{B-2.69a}) since $\ds \rfrac{1}{q'} \leq q'$ for $1 < q' \leq 2$. Thus, in (\ref{B-B.10}) we can take $\ds h \in W^{-\rfrac{1}{q},q}(\Omega)$. In particular 
	\begin{equation}\label{B-B.11}
		F^*D^* \in \calL \big(L^{q'}(\Omega)\big), \quad 1 < q' \leq 2.
	\end{equation}
	
	\section{Relevant unique continuation properties for overdetermined Oseen eigenvalue problems}\label{B-app-C}
	\setcounter{equation}{0}
	\renewcommand{\theequation}{{\rm C}.\arabic{equation}}
	\renewcommand{\thetheorem}{{\bf C}.\arabic{theorem}}
	
	In this Appendix \ref{B-app-C}, we assemble a comprehensive account of unique continuation problems for Oseen eigenproblems, as they pertain to the problem of controllability of finite dimensional projected system (\hyperref[B-4.8a]{4.8a-b}) of the linearized $w$-problem (\ref{B-1.28}) (with interior, tangential-like localized control $u \equiv 0 $). Positive solution, or lack thereof, of this finite dimensional problem is a key step, or obstruction, for the uniform stabilization of the Navier Stokes equations. This issue has been known since the study of boundary feedback stabilization of a parabolic equation with Dirichlet boundary trace in the feedback loop, as acting on the Neumann boundary conditions \cite{LT:1983}. We return to the bounded domain $\Omega, \ d = 2,3,$ with boundary $\Gamma = \partial \Omega$. As before, $\wti{\Gamma}$ is a subportion of $\Gamma$.\\
	
	\noindent \textbf{Problem \#1} (over-determination only on a portion $\wti{\Gamma}$ of $\Gamma$) Let $\{\varphi, p\} \in W^{2,q}(\Omega) \times W^{1,q}(\Omega)$ solve the over-determined problem 
	\begin{subequations}\label{B-C.1}
	\begin{align}
	(- \nu_o \Delta) \varphi + L_e(\varphi) + \nabla \pi & = \lambda \varphi  &\text{ in } \Omega \label{B-C.1a}\\ 
	\div \ \varphi &= 0  &\text{ in } \Omega \label{B-C.1b}\\
	\begin{picture}(100,0)
	\put(50,20){$\left\{\rule{0pt}{35pt}\right.$}\end{picture}
	\varphi|_{\wti{\Gamma}} \equiv 0, \quad \left. \frac{\partial \varphi}{\partial \nu} \right|_{\wti{\Gamma}} &\equiv 0 &\text{ on } \wti{\Gamma} \label{B-C.1c}
	\end{align}
	\end{subequations}
	\noindent with over-determination only on the portion $\wti{\Gamma}$ of $\Gamma$. Does (\hyperref[B-C.1a]{C.1a-b-c}) imply
	\begin{equation}
		\varphi \equiv 0 \quad \text{and} \quad p \equiv \text{const} \quad \text{in } \Omega \text{ ?} \label{B-C.2}
	\end{equation}
	\noindent The answer is \underline{negative} even in the Stokes case: $L_e(\varphi) \equiv 0$. This follows from \cite{FL:1996}, where the following counterexample is given in the 2-dimensional half-space $\Omega = \{(x,y) : x \in \BR^+, y \in \BR \}$ with boundary $\Gamma = \{ x = 0 \}$. On $\Omega$ take
	\begin{equation}
		u_1(x,y) \equiv 0, \quad u_2(x,y) = ax^2, \quad p = 2ay, \quad a \neq 0, \label{B-C.3}
	\end{equation}
	\noindent so that with $u = \{u_1, u_2\}$, it follows that
	\begin{equation}
		\Delta u = \nabla p \text{ in } \Omega, \quad u|_{\Gamma} = \nabla u|_{\Gamma} = 0, \label{B-C.4}
	\end{equation}
	\noindent to obtain a nontrivial solution of the Stokes overdetermined eigenproblem with $\lambda = 0$. Such half-space example can then be transformed into a counterexample over the bounded domain $\Omega$ where the over-determination is active on any subset $\wti{\Gamma}$ of the boundary $\Gamma = \partial \Omega$.\\
	
	\noindent \textbf{Implications of failure of unique continuation under Problem \hyperref[B-C.1]{\#1}:} A negative consequence of the lack of unique continuation (\ref{B-C.1}) $\implies$ (\ref{B-C.2}) with over-determination only in a portion $\wti{\Gamma}$ of the boundary $\Gamma$ is as follows: that global uniform stabilization of the linearized $w$-problem \eqref{B-1.28} by means of a purely tangential (finite or infinite dimensional) feedback boundary control $v$ (as given by (\ref{B-5.1}) in the finite dimensional case) acting only on a small subportion $\wti{\Gamma}$ of the boundary $\Gamma$ (and thus with localized interior tangential-like control $u \equiv 0$) is \underline{not possible}. This is so since the algebraic rank condition \eqref{B-4.11b} (with $u \equiv 0$) fails, as boundary traces 
	\begin{equation}\label{B-C.5}
		\left\{ \left.\frac{\partial \varphi^*_{i1}}{\partial \nu}\right|_{\wti{\Gamma}}, \left.\frac{\partial \varphi^*_{i2}}{\partial \nu}\right|_{\wti{\Gamma}}, \dots, \left.\frac{\partial \varphi^*_{i \ell_i}}{\partial \nu}\right|_{\wti{\Gamma}} \right\} \text{ fail to be linearly independent on } \wti{\Gamma}
	\end{equation}
	\noindent since, equivalently, the implication (\ref{B-C.1})$\implies$(\ref{B-C.2}) fails. See Orientation.\\
	
	\noindent \textbf{Problem \#2} (dual of the statement of Lemma \ref{B-lem-4.3}): necessity to complement the localized control $v$ on $\wti{\Gamma}$ with a localized interior tangential-like control $u$ supported on $\omega$ in terms of $\wti{\Gamma}$. Let now $\{ \varphi , p \} \in W^{2,q}(\Omega) \times W^{1,q}(\Omega)$ solve the problem	
	\begin{subequations}\label{B-C.6}
		\begin{align}
		(- \nu_o \Delta) \varphi + L_e(\varphi) + \nabla \pi & = \lambda \varphi  &\text{ in } \Omega \label{B-C.6a}\\ 
		\div \ \varphi &= 0  &\text{ in } \Omega \label{B-C.6b}\\
		\begin{picture}(100,0)
		\put(80,20){$\left\{\rule{0pt}{35pt}\right.$}\end{picture}
		\varphi|_{\wti{\Gamma}} \equiv 0, \quad \left. \frac{\partial \varphi}{\partial \nu} \right|_{\wti{\Gamma}} \equiv 0, \quad \varphi \cdot \tau & \equiv 0 &\text{ in } \omega \label{B-C.6c}
		\end{align}
	\end{subequations}
	\noindent Then, \cite[Theorem 6.2]{LT2:2015},
	\begin{equation}\label{B-C.7}
		\varphi \equiv 0 \quad \text{and} \quad p \equiv \text{const} \quad \text{in } \Omega.
	\end{equation}
	\noindent It is as a consequence of such unique continuation property that the Kalman algebraic rank conditions (\hyperref[B-6.28b]{6.28b}) are satisfied. This is the basic result upon which the uniform stabilization of the present paper relies. Thus we can conclude that the results of the present paper (as in \cite{LT2:2015}) are \underline{optimal} in terms of the required extra condition of the localized interior, tangential-like control needed to supplement the insufficient role of the localized tangential boundary control $v$ on $\wti{\Gamma}$. Optimality is in terms of the smallness of the required control action for $v$ and $u$.\\
	
	\noindent \textbf{Problem \#3} (over-determination on the entire boundary $\Gamma = \partial \Omega$). Let now $\{ \varphi , p \} \in W^{2,q}(\Omega) \times W^{1,q}(\Omega)$ solve the over-determined problem	
	\begin{subequations}\label{B-C.8}
		\begin{align}
		(- \nu_o \Delta) \varphi + L_e(\varphi) + \nabla \pi & = \lambda \varphi  &\text{ in } \Omega \label{B-C.8a}\\ 
		\div \ \varphi &= 0  &\text{ in } \Omega \label{B-C.8b}\\
		\begin{picture}(100,0)
		\put(50,20){$\left\{\rule{0pt}{35pt}\right.$}\end{picture}
		\varphi|_{\Gamma} \equiv 0, \quad \left. \frac{\partial \varphi}{\partial \nu} \right|_{\Gamma} &\equiv 0 &\text{ on } \Gamma \label{B-C.8c}
		\end{align}
	\end{subequations}
	\noindent with over-determination on all of $\Gamma$. Then, does (\hyperref[B-C.8a]{C.8a-b-c}) imply
	\begin{equation}\label{B-C.9}
	\varphi \equiv 0 \quad \text{and} \quad p \equiv \text{const} \quad \text{in } \Omega \text{ ?}
	\end{equation}
	\noindent It seems that a general definitive answer is not known at present. Only partial results are known.\\
	
	\noindent The desired unique continuation (\ref{B-C.8})$\implies$(\ref{B-C.9}) holds true, if the equilibrium solution $y_e \equiv 0$ (Stokes eigenproblem) or, more generally, if $y_e$ is sufficiently small in the $W^{1,q}(\Omega)$-norm. Several different proofs are given in \cite{RT:2008} and \cite{RT1:2009}. \\
	
	\noindent The case $y_e \equiv 0$ is actually physically quite important as it occurs for instance when the forcing function in (\ref{B-1.1a}) or (\ref{B-1.2a}) is a conservative vector field (say an electrostatic or gravitational field) $f = \nabla g$. In this case, a solution (\hyperref[B-1.2a]{1.2a-b-c}) is: $\ds y_e \equiv 0, \ \pi_e = g$.\\
	
	\noindent \textbf{When $y_e \equiv 0$ (or $y_e$ small) the tangential boundary feedback control $v$ alone, in the form such as (\ref{B-5.1}), as acting on the entire boundary $\Gamma$ produces enhancement of stability at will for the linearized $w$-problem.}\\
	
	\noindent Of course, with $y_e \equiv 0$, the corresponding Oseen problems reduces to the Stokes problem. The Stokes semigroup is already uniformly stable, see (\ref{B-3.7}), with margin of stability $\delta > 0$. When $y_e \equiv 0$ a most valuable variation of the problem under investigation of the present paper is to \uline{enhance the original margin of stability} $\delta > 0$ of the original linearized uncontrolled $w$-problem (\ref{B-1.11}) (with $u \equiv 0, \ v \equiv 0$) to obtain \uline{an arbitrary decay rate}, say $k^2$, by means of only a tangential boundary finite dimensional feedback control, of the same form as the operator $F$ in (\ref{B-5.6b}) but applied to all of $\Gamma$. To this, it suffices to apply the procedure of the present paper to a finite dimensional projected space spanned by the eigenvectors of the Stokes operator corresponding to finitely many eigenvalues $\lambda_i, \ i = 1, \dots, I$, 
	\begin{equation}\label{B-C.10}
		-k^2 \leq - Re \ \lambda_I \leq \dots \leq Re \ \lambda_1 \leq -\delta
	\end{equation}
	
	\noindent \textbf{Problem \#4} over-determination on a portion of the boundary $\wti{\Gamma}$ involving also the pressure $p$. Let $\{\varphi, p\} \in W^{2,q}(\Omega) \cap W^{1,q}(\Omega)$ solve the over-determined problem
	\begin{subequations}\label{B-C.11}
		\begin{align}
		(- \nu_o \Delta) \varphi + L_e(\varphi) + \nabla \pi & = \lambda \varphi  &\text{ in } \Omega \label{B-C.11a}\\ 
		\div \ \varphi &= 0  &\text{ in } \Omega \label{B-C.11b}\\
		\begin{picture}(100,0)
		\put(80,20){$\left\{\rule{0pt}{35pt}\right.$}\end{picture}
		\varphi|_{\Gamma} \equiv 0, \quad \left[ \frac{\partial \varphi}{\partial \nu} - p\nu \right]_{\Gamma} & \equiv 0 & \ \label{B-C.11c}
		\end{align}
	\end{subequations}
	\noindent Does this imply
	\begin{equation}
		\varphi \equiv 0 \quad \text{and} \quad p \equiv \text{const} \quad \text{in } \Omega \text{ ?}
	\end{equation}
	\noindent This answer is in the affirmative. The argument, given in the \cite{RT:2008} is along more classical elliptic arguments \cite{Ko:1994}. Here however the new condition in (\ref{B-C.11c}) contains the pressure, which must be viewed as unknown in general. Application of this result to the present paper will result in substituting $\ds \partial_{\nu} \varphi_{ij}^*|_{\wti{\Gamma}}$ with $\ds [\partial_{\nu} \varphi_{ij}^* - p_i \nu]|_{\wti{\Gamma}}$ in the matrix $W_i$ in (\ref{B-4.9}), which then - with this modification - becomes full rank, as desired. Thus, the stabilizing control will be expressed in terms of the pressure on the boundary, which is typically unknown.
\end{appendices}


	\medskip
	Received xxxx 20xx; revised xxxx 20xx.
	\medskip

\begin{thebibliography}{99}
		
		\bibitem[Adams]{A:1975}
		\newblock R. A. Adams,
		\newblock Sobolev Spaces.
		\newblock {\textit{Academic Press}, 1975. pp268}
		
		\bibitem[Ama.1]{HA:1995}
		\newblock H. Amann,
		\newblock{Linear and Quasilinear Parabolic Problems}.
		\newblock {\textit{Birkh\"{a}user}, 1995}.
		
		\bibitem[Ama.2]{HA:2000}
		\newblock H. Amann,
		\newblock{On the Strong Solvability of the Navier-Stokes Equations}.
		\newblock {\textit{J. Math. Fluid Mech. 2 }, 2000}.
		
		\bibitem[A-R]{AR:2010}
		\newblock C. Amrouche, M. A. Rodriguez-Bellido,
		\newblock Stationary Stokes, Oseen and Navier-Stokes equations with singular data. 
		\newblock {\textit{hal-00549166}, 2010.}
		

        \bibitem[Ba-Ta]{BT:2011}
        \newblock M. Badra, T. Takahashi,
        \newblock Stabilization of Parabolic Nonlinear Systems with Finite Dimensional Feedback or Dynamical Controllers: Application to the Navier–Stokes System,
        \newblock SIAM J. Control Optim., 49(2), 420–463. (44 pages), https://doi/10.1137/090778146

		\bibitem[Bal]{Bal:1981}
		\newblock A. V. Balakrishnan, 
		\newblock \textit{Applied Functional Analysis}, 
		\newblock Springer Verlag, Applications of mathematics Series, 2nd Edit 1981, pp 369.

		\bibitem[B-L-K]{BLK:2001}
		\newblock A. Balogh, W. J. Liu, M. Krstic,
		\newblock Stability enhancement by boundary control in 2D channel flow,
		\newblock \textit{IEEE Transactions on Automatic Control}, vol. 46, pp. 1696-1711, 2001.	
		
		\bibitem[B.1]{B:2011}
		\newblock V. Barbu,
		\newblock \textit{Stabilization of Navier–Stokes Flows}
		\newblock Springer Verlag, 2011, p 276.
		
		\bibitem[B.2]{B:2018}
		\newblock V. Barbu,
		\newblock \textit{Controllability and Stabilization of Parabolic Equations}
		\newblock Birkh\"{a}user Bessel, 2018, p 226.
		
		\bibitem[B-L]{BL:2012}
		\newblock V. Barbu, I. Lasiecka,
		\newblock The unique continuation property of eigenfunctions to Stokes–Oseen operator is generic with respect to the coefficients
		\newblock \textit{Nonlinear Analysis: Theory, Methods \& Applications}, 75(2012), pp 4384-4397.
				
		\bibitem[B-T.1]{BT:2004} 
		\newblock V. Barbu, R. Triggiani,
		\newblock Internal Stabilization of Navier-Stokes Equations with Finite-Dimensional Controllers,
		\newblock \textit{Indiana University Mathematics}, 2004, 123 pp.	
		
		\bibitem[B-L-T.1]{BLT1:2006} 
		\newblock V. Barbu, I. Lasiecka, R. Triggiani,
		\newblock Tangential Boundary Stabilization of Navier-Stokes Equations. 
		\newblock {\textit{ Memoires of American Math Society}, 2006.}
		
		\bibitem[B-L-T.2]{BLT2:2007} 
		\newblock V. Barbu, I. Lasiecka, R. Triggiani,
		\newblock{Abstract Settings for Tangential Boundary Stabilization of Navier-Stokes Equations by High- and Low-gain Feedback Controllers}. 
		\newblock {\textit{Nonlinear Analysis}, 2006}.
		
		\bibitem[B-L-T.3]{BLT3:2006} 
		\newblock V. Barbu, I. Lasiecka, R. Triggiani,
		\newblock{Local Exponential Stabilization Strategies of the Navier-Stokes Equations, d = 2,3 via Feedback Stabilization of its Linearization}. 
		\newblock {\textit{ Control of Coupled Partial Differential Equations, ISNM Vol 155, Birkhauser}, 2007, pp13-46}.
		
		
		\bibitem[Bo]{Bo:2016}
		\newblock J. Borggard,
		\newblock Talk presented at the seminar in Control Theory at the IMA, University of Minnesota, March 2016.
		
		\bibitem[B-G]{BG:2017}
		\newblock Z. Bradshaw, Z. Grujic,
		\newblock Frequency Localized Regularity Criteria for the 3D Navier Stokes Equations,
		\newblock ARNA 224 (2017) 125-133.
		
		\bibitem[C-V]{CV:1986}
		\newblock P. Cannarsa, V.Vespri,
		\newblock On Maximal $L^p$ regularity for the abstract Cauchy problem,
		\newblock Boll. Un. Mat. Ital B (6) 5 (1986) n 1, 165-175.
		
		
		
		
		\bibitem[C-M-R]{CMR:2018}
		\newblock S. Chowdhury, D. Mitra, M. Renardy,
		\newblock Null controllability of the incompressible Stokes equations in a 2-D channel using normal boundary control,
		\newblock \textit{Evolution Equations \& Control Theory}, 7(3) 447-463, 2018.
				
		\bibitem[C-F]{CF:1980}
		\newblock P. Constantin, C. Foias,
		\newblock Navier-Stokes Equations 
		\newblock \textit{(Chicago Lectures in Mathematics) 1st Edition}, 1980.
		
		\bibitem[DaP-V]{DPV:2002}
		\newblock G. DaPrato, V. Vespri, 
		\newblock Maximal $L^p$ regularity for elliptic equations with unbounded coefficients,
		\newblock NonLinear Analysis 49 (2002) n 6 Ser A: Theory Methods, 747-755
		
		\bibitem[D-G]{DG:2016}
        \newblock R. Dascaliuc, Z. Gruji\'{c},
        \newblock On Energy Cascades in the Forced 3D Navier-Stokes Equations.
        \newblock Journal of Nonlinear Science . Jun 2016, Vol. 26 Issue 3, p683-715. 33p.
		
		
		
		\bibitem[Dore]{Dore:2000}
		\newblock G. Dore,
		\newblock Maximal regularity in $L^p$ spaces for an abstract Cauchy problem,
		\newblock \textit{Advances in Differential Equations}, 2000.
		
		\bibitem[E-S-S]{ESS:1991}
		\newblock L. Escauriaza, G. Seregin, V. \v{S}ver\'{a}k,
		\newblock $L_{3,\infty}$-Solutions of Navier-Stokes Equations and Backward Uniqueness, 
		\newblock \textit{Mathematical subject classification (Amer. Math. Soc.):} 35K, 76D, 1991.
						
		\bibitem[F-L]{FL:1996}
		\newblock C. Fabre and G. Lebeau,
		\newblock Prolongement unique des solutions de l'équation de Stokes
		\newblock \textit{Comm. Part. Diff. Eq.}, 21, 1996, 573-596.
		
		\bibitem[F-M-M]{FMM:1998}
		\newblock E. Fabes, O. Mendez, M. Mitrea
		\newblock Boundary Layers of Sobolev-Besov Spaces and Poisson's Equations for the Laplacian for the Lipschitz Domains
		\newblock {\textit{J. Func. Anal} 159(2):323-368, 1998}
		
		\bibitem[Fat.1]{Fat.1}
		\newblock H. O. Fattorini,
		\newblock The Cauchy Problem
		\newblock {\textit{Encyclopedia of Mathematics and its Applications (18), Cambridge University Press}, 1984, ISBN: 9780511662799}		
		
		\bibitem[Fur.1]{F.1}
		\newblock A. Fursikov,
		\newblock Real processes corresponding to the 3D Navier-Stokes system, and its feedback stabilization from the boundary
		\newblock \textit{Partial Differential Equations, Amer. Math. Soc. Transl}, Ser. 2, Vol.\,260, AMS, Providence, RI, 2002.		
		
		\bibitem[Fur.2]{F.2}
		\newblock A. Fursikov, 
		\newblock Stabilizability of two dimensional Navier--Stokes equations with help of a boundary feedback control, \newblock {\textit{J. Math. Fluid Mech.} 3 (2001), 259--301.}
		
		\bibitem[Fur.3]{F.3}
		\newblock A. Fursikov, 
		\newblock Stabilization for the 3D Navier--Stokes system by feedback boundary control,
		\newblock {\textit{DCDS} 10 (2004), 289--314.}
		
		\bibitem[F-G]{FG:2012}
		\newblock A. Fursikov, A. Gorshkov, 
		\newblock Certain questions of feedback stabilization	for Navier-Stokes equations,
		\newblock Evolut. Equations and Control Theory June 2012, 109-140.
		
		\bibitem[FT]{FT:1984}
		\newblock C. Foias, R. Temam,
		\newblock Determination of the Solution of the Navier-Stokes Equations by a Set of Nodal Volumes,
		\newblock \textit{Mathematics of Computation}, Vol 43, N 167, 1984 , pp 117-133.
		
		
		
		\bibitem[Ga.1]{Ga:2011}
		\newblock G. P. Galdi,
		\newblock An Introduction to the Mathematical Theory of the Navier-Stokes Equations.
		\newblock {\textit{Springer-Verlag New York}, 2011.}
		
		\bibitem[G-K-P]{GKP:2010}
		\newblock I. Gallagher, G. S. Koch, F. Planchon,
		\newblock A profile decomposition approach to the $L^{\infty}_t (L^3_x)$ Navier-Stokes regularity criterion.
		\newblock {\textit{Math. Ann.},  355 (2013), no. 4, pp 1527-1559.}

		
		\bibitem[G-G-H.1]{GGH:2012} 
		\newblock M. Geissert, K. G\"otze, M. Hieber,
		\newblock {${L}_p$-Theory for Strong Solutions to Fluid-Rigid Body Interaction in Newtonian and Generalized Newtonian Fluids}. 
		\newblock {\textit{Transaction of American Math Society}, 2013.}
		
		\bibitem[Gi.1]{Gi:1981}
		\newblock Y. Giga,
		\newblock Analyticity of the semigroup generated by the Stokes operator in $L_r$ spaces,
		\newblock \textit{Math.Z.178}(1981), n 3, pp 279-329.
		
		\bibitem[Gi.2]{Gi:1985}
		\newblock Y. Giga,
		\newblock Domains of fractional powers of the Stokes operator in $L_r$ spaces,
		\newblock \textit{Arch. Rational Mech. Anal.} 89(1985), n 3, pp 251-265.

		
		\bibitem[H-S]{HS:2016}
		\newblock M. Hieber, J. Saal,
		\newblock The Stokes Equation in the $L^p$-setting: Well Posedness and Regularity Properties
		\newblock {\textit{Handbook of Mathematical Analysis in Mechanics of Viscous Fluids, Springer, Cham}, 2016.}
		
		\bibitem[J-S]{JS:2013}
		\newblock H. Jia, V. \v{S}ver\'{a}k, 
		\newblock Minimal $L^3$-initial data for potential Navier-Stokes singularities,
		\newblock \textit{SIAM J. Math. Anal.} 45 (2013), no. 3, 1448–1459.
		
		
		
		\bibitem[K-1]{TK:1966}
		\newblock T. Kato,
		\newblock Perturbation Theory of Linear Operators.
		\newblock {\textit{Springer-Verlag},1966.}
		
		\bibitem[Kee]{Kee:1998}
		\newblock L. R. Keefe,
		\newblock Method and apparatus for reducing the drag of flows over surfaces,
		\newblock US Patent - US 5803409, 1998.
		
		\bibitem[Kes]{SK:1989}
		\newblock S. Kesavan,
		\newblock {Topics in Functional Analysis and Applications}
		\newblock {\textit{New Age International Publisher},1989}
		
		
		\bibitem[Ko]{Ko:1994}
		\newblock V. Komornik,
		\newblock \textit{Exact Controllability and Stabilization, The Multiplier Theory}
		\newblock {Wiley-Masson Series Research in Applied Mathematics, 1996, pp 166}	

		\bibitem[K-T]{KT:2012}
		\newblock I. Kukavica, A. Tuffaha,
		\newblock \textit{Regularity solutions to a free boundary problem of fluid-structure interaction}
		\newblock \textit{Indiana Univ. Math. J.}, 61(2012), no. 5, pp 1817-1859.
		
		\bibitem[K-W.1]{KW:2001}
		\newblock P. C. Kunstmann, L. Weis,
		\newblock Perturbation theorems for maximal $L^p$-regularity
		\newblock {\textit{Annali della Scuola Normale Superiore di Pisa - Classe di Scienze,} Série 4 : Volume 30 (2001) no. 2 , p. 415-435}
		
		\bibitem[K-W.2]{KW:2004}
		\newblock P. C. Kunstmann, L. Weis,
		\newblock Maximal $L^p$-regularity for Parabolic Equations, Fourier Multiplier Theorems and   $H^{\infty}$-functional Calculus
		\newblock {\textit{Functional Analytic Methods for Evolution Equations, Lecture Notes in Mathematics, vol 1855. Springer, Berlin, Heidelberg} pp 65-311}
	
		\bibitem[Lad]{Lad:1969}
		\newblock O. A. Ladyzhenskaya,
		\newblock \textit{The Mathematical Theory of Viscous Incompressible Flow},
		\newblock Gordon and Breach, New York English transl., $2^{\text{nd}}$ Edition, 1969. 

		
		\bibitem[L-P-T.1]{LPT.1}
		\newblock I. Lasiecka, B. Priyasad, R. Triggiani,
		\newblock Uniform Stabilization of Navier–Stokes Equations in Critical $L^q$-Based Sobolev and Besov Spaces by Finite Dimensional Interior Localized Feedback Controls. Appl. Math Optim. (2019). https://doi.org/10.1007/s00245-019-09607-9
		
		\bibitem[L-T.1]{L-T.1}
		\newblock I. Lasiecka, R. Triggiani,
		\newblock Control Theory for Partial Differential Equations: Continuous and Approximation Theories, Vol. 1, Abstract Parabolic Systems (680 pp.), 
		\newblock {\textit{Encyclopedia of Mathematics and its Applications Series}, Cambridge University Press, January 2000.}		
		
		\bibitem[L-T.2]{LT1:2015} 
		\newblock I. Lasiecka, R. Triggiani,
		\newblock Uniform Stabilization with Arbitrary Decay Rates of the Oseen Equation by Finite-Dimensional Tangential Localized Interior and Boundary Controls. 
		\newblock {\textit{ Semigroups of Operators -Theory and Applications}, Proms 113, 2015, 125-154.}
		
		\bibitem[L-T.3]{LT2:2015} 
		\newblock I. Lasiecka, R. Triggiani,
		\newblock{Stabilization to an Equilibrium of the Navier-Stokes Equations with Tangential Action of Feedback Controllers}. 
		\newblock {\textit{Nonlinear Analysis}, 121 (2015), 424-446.}
		
		\bibitem[L-T.4]{LT:1983} 
		\newblock I. Lasiecka, R. Triggiani,
		\newblock{Stabilization of Neumann boundary feedback of parabolic equations: The case of trace in the feedback loop}. 
		\newblock {\textit{Appl. Math. and Optimz.}, Volume 10, Issue 1, pp 307–350, 1983.}
		
		\bibitem[Ler]{Ler:1934}
		\newblock J. Leray, 
		\newblock {Sur le mouvement d'un liquide visquex emplissent l'espace,}
		\newblock \textit{ Acta Math. J.}, 63 (1934), pp 193–248.
		
		\bibitem[Li]{Li:1969}
		\newblock J. L. Lions,
		\newblock \textit{Quelques Methodes de Resolutions des Problemes aux Limites Non Lineaire},
		\newblock {Dunod, Paris, 1969.}
		
		\bibitem[M-B]{MB:1986}
		\newblock V. Maslenniskova, M. Bogovskii,
		\newblock{Elliptic Boundary Values in Unbounded Domains with Non Compact and Non Smooth Boundaries}
		\newblock {\textit{Rend. Sem. Mat. Fis. Milano, 56:125-138}, 1986}
		
		\bibitem[M-S]{MS:1985}
		\newblock V. G. Mazya, T. O. Shaposhnikova,
		\newblock Theory of multipliers in spaces of differentiable functions,
        \newblock Monographs and Studies in Mathematics 23, Pittman 1985, https://doi.org/10.1016/B978-0-08-026491-2.50010-9.
		
		\bibitem[Pazy]{P:1983}
		\newblock A. Pazy,
		\newblock Semigroups of Linear Operators and Applications to Partial Differential Equations,
		\newblock {\textit{Springer-Verlag}, 1983.}
		
		
		\bibitem[P-S]{PS:2016}
		\newblock J. Pr\"{u}ss, G. Simonett,
		\newblock {\textit{Moving Interfaces and Quasilinear Parabolic Evolution Equations}}
		\newblock Birkh\"{u}user Basel, Monographs in Mathematics 105, 2016. 609pp.
		
		
		\bibitem[Ray.1]{JP:2007}
		\newblock J-P. Raymond,
		\newblock Feedback boundary stabilization of the three-dimensional incompressible Navier–Stokes equations,
		\newblock \textit{Journal de Math\'{e}matiques Pures et Appliqu\'{e}es}, Vol. 87, Issue 6, June 2007, pp 627-669.		
		
		\bibitem[R-S]{RS:2009}
		\newblock W. Rusin, V. Sverak, 
		\newblock \textit{Minimal initial data for potential Navier-Stokes singularities},
		\newblock https://arxiv.org/abs/0911.0500.	
			
		
		\bibitem[Saa]{S:2006}
		\newblock J. Saal,
		\newblock Maximal regularity for the Stokes system on non-cylindrical space-time domains,
		\newblock \textit{J. Math. Soc. Japan} 58 (2006), no. 3, 617-641. 
		
		
		\bibitem[Sch]{Sch:1979}
		\newblock C. Schneider,
		\newblock Traces of Besov and Triebel-Lizorkin spaces on domains,
		\newblock {\textit{Math. Nachr.284}, No. 5–6, 572 – 586 (2011).}
		
		\bibitem[Ser.1]{Ser:1962}
		\newblock J. Serrin,
		\newblock On the interior regularity of weak solutions of the Navier-Stokes equations,
		\newblock \textit{Arch. Rational Mech. Anal.}  (1962) 9: 187. https://doi.org/10.1007/BF00253344.
		
		\bibitem[Ser.2]{Ser:1963} 
		\newblock J. Serrin,
		\newblock The initial value problem for the Navier-Stokes equations, 
		\newblock {1963 \textit{Nonlinear Problems}, (Proc. Sympos., Madison, Wis., 1962) pp. 69–98 Univ. of Wisconsin Press, Madison, Wis. 35.79.}	
		
		\bibitem[Sh]{ZS:2012}
		\newblock Z. Shen,
		\newblock Resolvent Estimates in $L^p$ for the Stokes Operator in Lipschitz Domains,
		\newblock {\textit{Arch. Rational Mech. Anal. 205 395-424},2012.}
		
		\bibitem[Sol.1]{VAS:1968}
		\newblock V. A. Solonnikov,
		\newblock \textit{Estimates of the solutions of a nonstationary linearized system of Navier-
			Stokes equations},
		\newblock A.M.S. Translations, 75 (1968), 1-116.
		
		\bibitem[Sol.2]{VAS:1977}
		\newblock V. A. Solonnikov,
		\newblock Estimates for solutions of non-stationary Navier-Stokes equations,
		\newblock \textit{ J. Sov. Math.}, 8, 1977, pp 467-529.
		
		\bibitem[Sol.3]{VAS:1981}
		\newblock V. A. Solonnikov,
		\newblock On the solvability of boundary and initial-boundary value problems for the Navier-Stokes system in domains with noncompact boundaries.
		\newblock \textit{Pacific J. Math.} 93 (1981), no. 2, 443-458. https://projecteuclid.org/euclid.pjm/1102736272.
		
		\bibitem[Sol.4]{VAS:1996} 
		\newblock V. A. Solonnikov, 
		\newblock On Schauder Estimates for the Evolution Generalized Stokes Problem. 
		\newblock \textit{Ann. Univ. Ferrara} 53, 1996, 137-172.
		
		\bibitem[Sol.5]{VAS:2001}
		\newblock V. A. Solonnikov,
		\newblock $L^p$-Estimates for Solutions to the Initial Boundary-Value Problem for the Generalized Stokes System in a Bounded Domain,
		\newblock \textit{J. Math. Sci.}, Volume 105, Issue 5, pp 2448–2484.
		
		
		
		\bibitem[T-L.1]{TL:1980} 
		\newblock A. E. Taylor, D. Lay
		\newblock Introduction to Functional Analysis 2nd Edition. 
		\newblock {\textit{Wiley Publication, ISBN-13: 978-0471846468}, 1980.}
		
		\bibitem[Te]{Te:1979}
		\newblock R. Temam, 
		\newblock \textit{Navier-Stokes Equations},
		\newblock North Holland, 1979, pp 517.
		
		\bibitem[Tr.1]{RT:1975}
		\newblock R. Triggiani,
		\newblock On the Stability Problem of Banach Spaces.
		\newblock {\textit{J. Math. Anal. Appl. 52 303-403}, 1975.}
		
		\bibitem[Tr.2]{RT:1980}
		\newblock R. Triggiani,
		\newblock Feedback Stability of Parabolic Equations.
		\newblock {\textit{Appl. Math. Optimiz. 6 201-220}, 1975.}
		
		\bibitem[Tr.3]{RT:1980:2}
		\newblock R. Triggiani,
		\newblock Boundary feedback stabilizability of parabolic equations,
		\newblock {\textit{Appl. Math.~Optimiz.} 6 (1980), 201--220.}
		
		\bibitem[Tr.4]{RT:2008}
		\newblock R. Triggiani,
		\newblock Linear independence of boundary traces of eigenfunctions of elliptic and Stokes Operators and applications, invited paper for special issue,
		\newblock {\textit{Applicationes Mathematicae} 35(4) (2008), 481--512, Institute of Mathematics, Polish Academy of Sciences.}
		
		\bibitem[Tr.5]{RT1:2009}
		\newblock R. Triggiani,
		\newblock Unique continuation of boundary over-determined Stokes and Oseen eigenproblems,
		\newblock \textit{Discrete \& Continuous Dynamical Systems - S}, Vol. 2 , N. 3, Sept 2009, 645-677.
		

		\bibitem[Trie]{HT:1980}
		\newblock H. Triebel,
		\newblock Interpolation Theory, Function Spaces, Differential Operators.
		\newblock {\textit{Bull. Amer. Math. Soc. (N.S.) 2, no. 2, 339-345 }, 1980.}
		
		\bibitem[Wahl]{W:1985}
		\newblock W. von Wahl,
		\newblock The Equations of Navier-Stokes and Abstract Parabolic Equations.
		\newblock {\textit{Springer Fachmedien Wiesbaden, Vieweg+Teubner Verlag }, 1985.}	
		
		\bibitem[Weid]{PW:2002}
		\newblock P. Weidemaier,
		\newblock Maximal Regularity for Parabolic Equations with Inhomogeneous Boundary Conditions in Sobolev Spaces with Mixed $L^p$-norm.
		\newblock {\textit{Electronic Research Announcements of the AMS,} volume 8, pp 47-51, 2002.}
		
		\bibitem[Weis]{We:2001}
		\newblock L. Weis,
		\newblock A new approach to maximal Lp-regularity.
		\newblock {\textit{In Evolution Equ. and	Appl. Physical Life Sci.,} volume 215 of Lect. Notes Pure and Applied Math., pages 195–214, New York, 2001. Marcel Dekker.}
		
		\bibitem[Za]{Za}
		\newblock J. Zabczyk,
		\newblock Mathematical Control Theory: An Introduction,
		\newblock Systems \& Control: Foundations \& Applications, Birkh\"auser Boston, Inc., Boston, MA, 1992.
		
	\end{thebibliography}
\end{document}